\documentclass[a4paper,reqno,11pt]{amsart}
\usepackage{amsfonts}
\usepackage{amsmath}
\usepackage{amsthm, amsmath, amssymb, enumerate}
\usepackage{amssymb, enumitem}
\usepackage{bbm}
\usepackage{amscd}
\usepackage[all]{xy}
\usepackage[left=2cm,right=2cm,bottom=3cm,top=3cm]{geometry}
\usepackage[hidelinks]{hyperref}
\usepackage{amsmath}
\usepackage{amssymb}
\usepackage{verbatim}
\usepackage{amsthm}
\usepackage[sort&compress,numbers]{natbib}

\newcounter{theoremintro}
\newtheorem{thmintro}[theoremintro]{Theorem}
\newtheorem{corintro}[theoremintro]{Corollary}
\newtheorem{theorem}{Theorem}[section]
\newtheorem*{theorem*}{Theorem}

\newtheorem*{claim*}{Claim}

\newtheorem{corollary}[theorem]{Corollary}

\newtheorem{example}[theorem]{Example}

\newtheorem{notation}[theorem]{Notation}
\newtheorem*{notation*}{Notation}

\newtheorem{proposition}[theorem]{Proposition}

\newtheorem*{question*}{Question}

\theoremstyle{definition}

\newtheorem{remark}[theorem]{Remark}
\newtheorem{lemma}[theorem]{Lemma}
\newtheorem{definition}[theorem]{Definition}
\AtBeginDocument{   \def\MR#1{}
}

\input{tcilatex}

\begin{document}
\def\cprime{$'$}

\title[The complexity of actions of nonamenable groups]{The complexity of
conjugacy, orbit equivalence, and von Neumann equivalence of actions of
nonamenable groups}
\author{Eusebio Gardella}
\address{Eusebio Gardella\\
Westf\"{a}lische Wilhelms-Universit\"{a}t M\"{u}nster, Fachbereich
Mathematik, Einsteinstrasse 62, 48149 M\"{u}nster, Germany}
\email{gardella@uni-muenster.de}
\urladdr{https://wwwmath.uni-muenster.de/u/gardella/}
\author{Martino Lupini}
\address{Martino Lupini\\
Mathematics Department\\
California Institute of Technology\\
1200 E. California Blvd\\
MC 253-37\\
Pasadena, CA 91125}
\email{lupini@caltech.edu}
\urladdr{http://www.lupini.org/}
\date{\today }
\thanks{The first-named author was partially funded by SFB 878 \emph{Groups,
Geometry and Actions}, and by a postdoctoral fellowship from the Humboldt
Foundation. The second-named author was partially supported by the NSF Grant
DMS-1600186. This work was initiated while the authors were visiting the
Centre de Recerca Matem{\`{a}}tica in March 2017, in occasion of the
Intensive Research Programme on Operator Algebras. Part of the work was done
while the authors where visiting the University of Houston
in July and August 2017, in occasion of the Workshop on Applications of Model Theory
to Operator Algebras, supported by the NSF grant DMS-1700316.  The authors gratefully
acknowledge the hospitality of both institutions.}
\dedicatory{}

\begin{abstract}
Building on work of Popa, Ioana, and Epstein--T\"{o}rnquist, we show that,
for every nonamenable countable discrete group $\Gamma$, the relations of
conjugacy, orbit equivalence, and von Neumann equivalence of free ergodic (or weak
mixing) measure preserving actions of $\Gamma$ on the standard
atomless probability space are not Borel, thus answering questions of
Kechris. This is an optimal and definitive result, which establishes a neat
dichotomy with the amenable case, since any two free ergodic actions of an
amenable group on the standard atomless probability space are orbit equivalent
by classical results of Dye and Ornstein--Weiss. The statement about conjugacy
solves the nonamenable case of Halmos' conjugacy problem in Ergodic Theory, 
originally posed by Halmos in 1956 for ergodic transformations.

In order to obtain these results, we study ergodic (or weak mixing)
class-bijective extensions of a given ergodic countable probability measure preserving
equivalence relation $R$. When $R$ is nonamenable, we show that the relations
of isomorphism and von Neumann equivalence of extensions of $R$ are not Borel.
When $R$ is amenable, all the extensions of R are again amenable, and hence isomorphic
by classical results of Dye and Connes--Feldman--Weiss. This approach allows us to
extend the results about group actions mentioned above to the case of nonamenable
locally compact unimodular groups, via the study of their cross-section equivalence relations.
\end{abstract}

\maketitle

\renewcommand*{\thetheoremintro}{\Alph{theoremintro}}

\section{Introduction}

Let $\Gamma $ be countable discrete group, and let $(Y,\nu )$ be a standard
atomless probability space. A\emph{\ probability-measure-preserving }(pmp)
\emph{action} of $\Gamma $ on $(Y,\nu )$ is a homomorphism $\theta \colon
\Gamma \rightarrow \mathrm{Aut}(Y,\nu )$ from $\Gamma $ to the group of
measure-preserving Borel automorphisms of $(Y,\nu )$.
%Such an action can be identified with the action $\theta $ of
%$\Gamma $ on $L^{\infty }(Y,\nu )$, defined by $\theta _{\gamma } (
%f ) \colon x\mapsto f ( \gamma ^{-1}\cdot ^{\theta }x ) $ for $%
%\gamma \in \Gamma $ and $x\in Y$. The action $\theta $ is \emph{ergodic }if,
%for $f\in L^{\infty }(Y,\nu )$, $\theta _{\gamma } ( f ) =f$ for
%every $\gamma \in \Gamma $ implies that $f$ is constant up to null sets. The
%action $\theta $ is \emph{weak mixing }if the diagonal action $\gamma \cdot
%^{\theta \times \theta } ( x,y ) = ( \gamma \cdot ^{\theta
%}x,\gamma \cdot ^{\theta }y ) $ on $Y\times Y$ is ergodic. Finally, the
%action $\theta $ is \emph{free }if, for every nonidentity element $\gamma $
%of $\Gamma $, the set $ \{ x\in Y\colon \gamma \cdot ^{\theta }x=x \} $
%is null.
Two pmp actions $\theta $ and $\theta ^{\prime }$ of $\Gamma $ on $(Y,\nu )$
are \emph{conjugate} if there exists $T\in \mathrm{Aut}(Y,\nu )$ such that $%
T\circ \theta _{\gamma }=\theta _{\gamma }^{\prime }\circ T$ for every $%
\gamma \in \Gamma $. The \emph{conjugacy problem} in ergodic theory,
initially formulated by Halmos for $\mathbb{Z}$-actions \cite%
{halmos_lectures_1956}, asks whether there exists a \emph{method} to
determine whether two given pmp actions of $\Gamma $ on $(Y,\nu )$ are
conjugate.

By Halmos' own admission, this is a vague question, but it can be given a
precise meaning in the context of Borel complexity. The space $\mathrm{%
\mathrm{Act}}_{\Gamma }(Y,\nu )$ of pmp actions of $\Gamma $ on $(Y,\nu )$
is endowed with a canonical Polish topology \cite[Section 10]%
{kechris_global_2010}, and the subset $\mathrm{FE}_{\Gamma }(Y,\nu )$ of
free ergodic actions is a Polish space with the induced topology. The same
holds for the subset $\mathrm{FWM}_{\Gamma }(Y,\nu )$ of free weak mixing
actions. An instance of Halmos' conjugacy problem is the following:

\begin{question*}
(Kechris, \cite[18.(IVb)]{kechris_global_2010}). Let $\Gamma$ be a countable
discrete group and let $(Y,\nu)$ be a standard atomless probability space.

\begin{enumerate}
\item Is the relation of conjugacy of free ergodic pmp actions of $\Gamma $
on $(Y,\nu )$ a Borel subset of $\mathrm{\mathrm{FE}}_{\Gamma }(Y,\nu
)\times \mathrm{\mathrm{FE}}_{\Gamma }(Y,\nu )$ endowed with the product
topology?

\item Is the relation of conjugacy of free weak mixing pmp actions of $%
\Gamma $ on $(Y,\nu )$ a Borel subset of $\mathrm{\mathrm{FWM}}_{\Gamma
}(Y,\nu )\times \mathrm{\mathrm{FWM}}_{\Gamma }(Y,\nu )$ endowed with the
product topology?
\end{enumerate}
\end{question*}

Observe that a negative answer for part~(2) implies a negative answer for
part~(1), since $\mathrm{FWM}_{\Gamma }(Y,\nu )$ is a Borel subset of $%
\mathrm{\mathrm{FE}}_{\Gamma }(Y,\nu )$. The question above has been
addressed in \cite{foreman_conjugacy_2011} in the case of $\mathbb{Z}$%
-actions, where it is shown that the relation of conjugacy of free ergodic
pmp transformations is \emph{not} Borel. Epstein and T\"{o}rnquist showed in
\cite{epstein_borel_2011} that when $\Gamma $ is a group containing a
nonabelian free group as an (almost) normal subgroup, then the relation of
conjugacy of free weak mixing actions of $\Gamma $ on $(Y,\nu )$ is not
Borel.

In the present paper, we give a complete answer to both parts of the
question above for actions of arbitrary nonamenable groups:

\begin{thmintro}
\label{Theorem:nB-conj}Let $\Gamma $ be a nonamenable countable group and
let $(Y,\nu)$ be a standard atomless probability space. Then the relations
of conjugacy of free weak mixing (or ergodic) pmp actions of $\Gamma$ on $%
(Y,\nu )$ is not Borel.
\end{thmintro}

In recent years, ergodic theory has focused on classification of actions up
to other two equivalence relations: orbit equivalence and von Neumann
equivalence. A \emph{countable Borel equivalence relation} on a standard
probability space $(Y,\nu )$ is a Borel equivalence relation $R$ on $Y$ such
that the $R$-class of almost every point in $Y$ is countable. A countable
Borel equivalence relation $R$ is said to be \emph{%
probability-measure-preserving} (pmp) if every Borel automorphisms of $Y$
that maps almost everywhere $R$-classes to $R$-classes is automatically a
measure-preserving automorphism of $(Y,\nu )$. A countable pmp Borel
equivalence relation $R$ is called \emph{ergodic }if every $R$-invariant
Borel subset of $Y$ is either null or co-null. One can associate with a
countable pmp Borel equivalence relation $R$ a von Neumann algebra $L(R)$,
which is a II$_{1}$ factor if and only if $\left( Y,\nu \right) $ is
atomless and $R$ is ergodic \cite%
{feldman_ergodic_1977,feldman_ergodic_1977-1}.

Two ergodic countable Borel equivalence relations $R$ and $R^{\prime }$ on
the standard atomless probability space are:

\begin{itemize}
\item \emph{isomorphic} if there is a measure-preserving automorphism of $%
(Y,\nu )$ that maps almost everywhere $R$-classes to $R^{\prime }$-classes;

\item \emph{stably isomorphic }if there exist Borel subsets $X$ and $%
X^{\prime }$ of $Y$ that meet every class of $R$ and, respectively, $%
R^{\prime }$, and a measure-preserving Borel isomorphism $\varphi
:X\rightarrow X^{\prime }$ that maps $R$-classes to $R^{\prime }$-classes;

\item \emph{von Neumann equivalent }if the II$_{1}$ factors $L(R)$ and $%
L(R^{\prime })$ are isomorphic;

\item \emph{stably von Neumann equivalent }if the II$_{1}$ factors $L(R)$
and $L(R^{\prime })$ are stably isomorphic, i.e.\ there exist nonzero
projections $p\in L(R)$ and $q\in L(R^{\prime })$ such that the corners $%
pL(R)p$ and $qL(R^{\prime })q$ are isomorphic.
\end{itemize}

Given a pmp action $\theta $ of a countable discrete group $\Gamma $ on a
standard atomless probability space $\left( Y,\nu \right) $, one can define
its \emph{orbit equivalence relation} $R$ on $\left( Y,\nu \right) $ by
setting $xRy$ if and only if there exists $\gamma \in \Gamma $ such that $%
\theta _{\gamma }(x)=y$. When $\theta $ is free and ergodic, the von Neumann
algebra $L(R)$ is a II$_{1}$ factor isomorphic to the group-measure space
construction $\Gamma \ltimes ^{\theta }L^{\infty }(Y,\nu )$ of Murray and
von Neumann \cite{murray_rings_1936}. Two free ergodic action $\theta $ and $%
\theta ^{\prime }$ of $\Gamma $ on the standard atomless probability space
are:

\begin{itemize}
\item \emph{(stably) }$\emph{orbit}$\emph{\ equivalent }if their orbit
equivalence relations are (stably) isomorphic,

\item \emph{(stably) von Neumann equivalent }if their orbit equivalence
relations are (stably) von Neumann equivalent.
\end{itemize}

When $\Gamma$ is an \emph{amenable group}, it follows from results of Dye
\cite{dye_groups_1959} and Ornstein--Weiss \cite{ornstein_ergodic_1980} that
any two free ergodic actions of $\Gamma $ on $(Y,\nu )$ are orbit
equivalent. On the other hand, when $\Gamma $ is nonamenable, Epstein and
Ioana \cite{epstein_orbit_2007,ioana_orbit_2011} showed that the relations
of orbit equivalence and von Neumann equivalence for free, ergodic, pmp
actions have uncountably many classes. This has motivated Kechris to ask the
following:

\begin{question*}
(Kechris, \cite[18.(IVb)]{kechris_global_2010}). Let $\Gamma$ be a
nonamenable countable group, and let $(Y,\nu)$ be a standard atomless
probability space. Are the relations of orbit equivalence and von Neumann
equivalence for free ergodic actions of $\Gamma $ on $(Y,\nu)$ Borel?
\end{question*}

The results, as well as the proofs, from \cite%
{epstein_orbit_2007,ioana_orbit_2011} that such relations have uncountably
many classes do not address the question above. Indeed, in these papers one
encodes irreducible representations of $\mathbb{F}_{2}$ up to unitary
equivalence within free ergodic actions of $\Gamma $ up to orbit
equivalence; see also \cite{ioana_subequivalence_2009}. However, the
relation of unitary equivalence of irreducible representations of $\mathbb{F}%
_{2}$ \textit{is} Borel. Therefore, a new set of ideas and techniques is
needed to answer this question. In this paper, we completely settle the
matter:

\begin{thmintro}
\label{Theorem:nB-oe}Let $\Gamma $ be the nonamenable countable discrete
group, and let $(Y,\nu )$ be the standard atomless probability space. Then
the relations of orbit equivalence, stable orbit equivalence, von Neumann
equivalence, and stable von Neumann equivalence of free weak mixing (or
ergodic) pmp actions of $\Gamma $ on $(Y,\nu )$ are not Borel.
\end{thmintro}

Together with the results of Dye and Ornstein--Weiss for amenable groups
recalled above, Theorem \ref{Theorem:nB-oe} implies the following dichotomy.

\begin{corintro}
\label{cor:dychotomyDiscrGrp} Let $\Gamma $ be a countable discrete group,
and let $(Y,\nu )$ be the standard atomless probability space. Suppose that $%
E$ is one of the following equivalence relations: orbit equivalence, stable
orbit equivalence, von Neumann equivalence, or stable von Neumann
equivalence of free weak mixing (or ergodic) pmp actions of $\Gamma $ on $%
(Y,\nu )$.

\begin{enumerate}
\item[(a)] If $\Gamma $ is amenable, then $E$ has a single equivalence class.

\item[(b)] If $\Gamma $ is not amenable, then $E$ is not Borel.
\end{enumerate}
\end{corintro}

The methods that we use to obtain Theorem \ref{Theorem:nB-conj} and Theorem %
\ref{Theorem:nB-oe} contain two main innovations. First, we introduce the
notion of property (T) for a \emph{triple }$\Delta \leq \Lambda \leq \Gamma $
consisting a group $\Gamma $ together with nested subgroups $\Delta
\subseteq \Lambda \subseteq \Gamma $. The classical notion of property (T)
for pairs $\Lambda \leq \Gamma $ corresponds to the case when $\Delta $ is
the trivial subgroup. In this context, we present an extension of Popa's
cocycle superrigidity theorem for malleable actions \cite{popa_cocycle_2007}%
, which plays a crucial role in our construction. Specifically, these
techniques allow us to construct a family, parametrized by \emph{countable
abelian groups}, of free ergodic actions of $\mathbb{F}_{2}$. The same
construction applies after replacing $\mathbb{F}_{2}$ with more general
nonamenable groups.

%This is very relevant to our purposes, since the families constructed in all previous works
%(see, for example, \cite{gaboriau_uncountable_2005}) were parametrized by
%the positive real numbers, thus offering no information in terms of Borel complexity.

The next fundamental ingredient, to carry over the argument from free groups
to arbitrary amenable groups, is the combination of the Gaboriau--Lyons
measurable solution to von Neumann's problem from \cite%
{gaboriau_measurable-group-theoretic_2009} with the theory of coinduction
for actions of groupoids, initially considered for principal groupoids by
Epstein in \cite{epstein_orbit_2007}. The case of groups containing a free
group is technically easier, as it does not require the Gaboriau--Lyons
theorem, and it only uses the theory of coinduction for group actions. Since
this case contains \emph{in nuce }all the fundamental tools and ideas needed
to prove the general result, while also being significantly less technical,
a sketch of its proof is presented in Section~2.

Rigidity questions have also been intensively studied for actions of locally
compact, second countable, unimodular groups. For example, it follows from
\cite{connes_amenable_1981} that all free pmp ergodic actions of an amenable
locally compact, second countable, unimodular group on standard probability
spaces are orbit equivalent. On the other hand, the nonamenable case was
considerably more difficult to approach in comparison with the discrete
case. In \cite{zimmer_ergodic_1984}, Zimmer showed that any connected
semisimple Lie group with real rank at least two, finite center, and no
compact quotients, admits uncountably many pairwise not orbit equivalent
free ergodic pmp actions. In a very recent breakthrough, Bowen, Hoff, and
Ioana \cite{bowen_neumanns_2015} showed that the same conclusion holds for
any nonamenable locally compact, second countable, unimodular group.
However, their proof offers no information on whether the relation of orbit
equivalence is Borel or not. Here, we also settle this matter:

\begin{thmintro}
\label{Theorem:nB-oeLCSCU}Let $G$ be a nonamenable locally compact, second
countable, unimodular group, and let $(Y,\nu )$ be the standard atomless
probability space. Then the relations of conjugacy, orbit equivalence,
stable orbit equivalence, von Neumann equivalence, and stable von Neumann
equivalence of free ergodic pmp actions of $G$ on $(Y,\nu )$ are not Borel.
\end{thmintro}

%Combined with the results in \cite{connes_amenable_1981}, Theorem~D implies
%a dichotomy similar to the one stated in Corollary~\ref{cor:dychotomyDiscrGrp}.

We will obtain Theorem~\ref{Theorem:nB-oeLCSCU} as a consequences of our
main result, Theorem~\ref{Theorem:nB-oe-relation}, which we proceed to
motivate. A \emph{pmp class-bijective extension} of a countable pmp Borel
equivalence relation $R$ on a standard probability space $(Y,\nu )$ is a
pair $(\widehat{R},\pi )$, where $\widehat{R}$ is a countable pmp
equivalence relation on a standard probability space $(X,\mu )$, and $\pi
\colon X\rightarrow Y$ is a Borel map with $\pi _{\ast }(\mu )=\nu $, which
maps $[x]_{\widehat{R}}$ bijectively onto $[\pi (x)]_{R}$ for almost every $%
x\in X$; see \cite%
{kechris_spaces_2017,feldman_subrelations_1989,bowen_neumanns_2015}.
Ergodicity and weak mixing for (class-bijective extensions of) countable pmp
equivalence relations are defined in a natural way; see Subsection~\ref%
{Subsection:weak-mixing}. We consider a natural strengthening of the
relation of isomorphism for pmp class-bijective extensions of a given
countable pmp Borel equivalence relation $R$.

\begin{definition}
Let $(\widehat{R}_{1},\pi _{1})$ and $(\widehat{R}_{2},\pi _{2})$ be pmp
class-bijective extensions of a countable Borel equivalence relation $R$ on $%
\left( Y,\nu \right) $. Then $(\widehat{R}_{1},\pi _{1})$ and $(\widehat{R}%
_{2},\pi _{2})$ are \emph{isomorphic relatively to }$R$ if there exists $%
\theta \in \mathrm{Aut}\left( Y,\nu \right) $ such that, up to discarding a
null set, $\pi _{2}\circ \theta =\pi _{1}$, and $\theta $ maps $\widehat{R}%
_{1}$-classes to $\widehat{R}_{2}$-classes.
\end{definition}

Countable pmp equivalence relations and class-bijective extensions arise
naturally in the context of pmp actions of locally compact groups, as we
briefly explain. A free, ergodic, pmp action of a locally compact, second
countable group $G$ on a standard probability space $\left( X,\mu \right) $
admits a \emph{cocompact cross section }\cite[Theorem 4.2]%
{kyed_l2-betti_2015}; see also \cite[Proposition 2.10]{forrest_virtual_1974}%
. This is a Borel subset $Y\subseteq X$ for which there exist an open
neighborhood $U$ of the identity in $G$ and a compact subset $K$ of \ G such
that the restriction of the action $U\times Y\rightarrow X$ is injective,
and the image $K\cdot Y$ of $K\times Y$ under the action is a $G$-invariant
and conull subset of $X$. One then defines its associated \emph{cross
section equivalence relation} as the restriction to $Y$ of the orbit
equivalence relation of $G$, which is a countable Borel equivalence relation
on $Y$. It is proved in \cite[Proposition 4.3]{kyed_l2-betti_2015} that
there exists a unique $R$-invariant probability measure $\nu $ on $Y$ such
that the push-forward measure of $\lambda |_{U}\times \nu $ by the
restriction of the action $U\times Y\rightarrow X$ is equal to a multiple of
$\mu |_{U\cdot Y}$; see also \cite[Remark 8.2]{bowen_neumanns_2015}. It is
also shown in \cite[Proposition~8.3]{bowen_neumanns_2015} that if $G$ is
unimodular, then every countable class-bijective pmp extension of $R$ is
isomorphic to the cross section equivalence relation of some other free
ergodic pmp action of $G$.

It follows from classical results of Dye \cite{dye_groups_1959}, and of
Connes--Feldman--Weiss \cite{connes_amenable_1981}, that any two ergodic
\emph{amenable} countable pmp equivalence relations on the standard atomless
probability space are isomorphic. In particular, any two ergodic
class-bijective pmp extensions of an ergodic amenable countable pmp
equivalence relation are isomorphic, since these are automatically amenable.
On the other hand, it has recently been shown in \cite{bowen_neumanns_2015}
that a \emph{nonamenable} ergodic countable pmp equivalence relation has
uncountably many pairwise not stably isomorphic ergodic class-bijective pmp
extensions. Their arguments do not give information on whether the relation
of (stable) isomorphism of ergodic class-bijective pmp extensions of a given
nonamenable ergodic pmp equivalence relation is Borel. In this paper, we
strengthen their result by solving this problem:

\begin{thmintro}
\label{Theorem:nB-oe-relation} Let $R$ be an ergodic nonamenable countable
pmp equivalence relation on a standard probability space $(X,\mu )$, and let
$(Y,\nu )$ be the standard atomless probability space. Then the relations of
isomorphism relative to $R$, isomorphism, stable isomorphism, von Neumann
equivalence, and stable von Neumann equivalence of weak mixing (or ergodic)
class-bijective extensions of $R$ on $(Y,\nu )$ are not Borel.
\end{thmintro}

One technical difference between our approach and the one used in \cite%
{bowen_neumanns_2015} is that we regard equivalence relations as groupoids,
and that we work with certain discrete subgroups of the full groups. This
allows us, in some sense, to work with discrete groups throughout, which
gives us access to many arguments that would otherwise not be available.
%(Thus, roughly speaking, we reduce the proof ofTheorem~\ref{Theorem:nB-oe-relation} to that of Theorem~\ref{Theorem:nB-oe}.)

Together with the results of Dye and Connes--Feldman--Weiss recalled above,
Theorem \ref{Theorem:nB-oe-relation} implies the following dichotomy.

\begin{corintro}
Let $R$ be an ergodic pmp countable Borel equivalence relation on a standard
Borel space $(X,\mu )$, and let $(Y,\nu )$ be the standard atomless
probability space. Suppose that $E$ is one of the following equivalence
relations: isomorphism, stable isomorphism, von Neumann equivalence, or
stable von Neumann equivalence of class-bijective pmp extensions of $R$ on $%
(Y,\nu )$.

\begin{enumerate}
\item[(a)] If $R$ is amenable, then $E$ has a single equivalence class.

\item[(b)] If $R$ is not amenable, then $E$ is not Borel.
\end{enumerate}
\end{corintro}

The present paper is divided into four sections, apart from this
introduction. In Section \ref{Section:free} we introduce the notion of
property (T) for triples of groups, and use it to prove Theorem \ref%
{Theorem:nB-conj} and Theorem \ref{Theorem:nB-oe} in the case when the given
nonamenable group $\Gamma $ contains a nonabelian free group. Although not
necessary for the rest of the paper, we decided to include a self-contained
proof of this particular case to illustrate the main ideas involved in the
proof of Theorem~\ref{Theorem:nB-conj} and Theorem~\ref{Theorem:nB-oe}.
Additionally, it provides a different proof of the main result of \cite%
{ioana_orbit_2011}. The main difference in our approach is the use of an
analog of Popa's superrigidity theorem in the context of triples of groups
with property (T). This grants us access to a weak form of rigidity for
groups containing a free group, which is not available in the context of
the previously known superrigidity results.

%The methods used by Ioana are necessarily specific to free groups,
%and do not unfortunately reveal any information in terms of the Borel
%complexity of orbit equivalence. While Ioana's actions are essentially
%parametrized by real numbers, our construction gives a parametrization
%by (isomorphism classes of) countable abelian groups, via 1-cohomology.

In Section \ref{Section:actions} we prove several facts concerning discrete
pmp groupoids, their actions and representations. For future reference, we
present these facts in a slightly more general form than strictly needed in
this paper. In Section~\ref{Section:coinduction}, we develop a coinduction
theory for actions of groupoids, which can be seen as a common
generalization of coinduction for groups and for countable pmp equivalence
relations as defined by Epstein \cite{epstein_orbit_2007}. In Section \ref%
{Section:main}, we present the proof of \ref{Theorem:nB-oe-relation}, and
show how to deduce from it Theorems~\ref{Theorem:nB-conj}, \ref%
{Theorem:nB-oe}, and \ref{Theorem:nB-oeLCSCU}.

\begin{notation*}
\label{Notation:leg} Let $I$ be an index set, and let $(H_{i})_{i\in I}$ be
a family of Hilbert spaces with distinguished unit vectors $\xi _{i}\in
H_{i} $, for $i\in I$. Define $H$ to be the tensor product $\bigotimes_{i\in
I}(H_{i},\xi )$ as defined in \cite[Section III.3]{blackadar_operator_2006}.
For every $i_{0}\in I$, there is a canonical isometric inclusion $%
H_{i_{0}}\rightarrow \bigotimes_{i\in I}(H_{i},\xi _{i})$. For $\eta \in
H_{i_{0}}$, we denote by $\eta _{(i_{0})}\in H$ the image of $\eta $ under
such an inclusion. We will adopt similar notations for tensor products of
von Neumann algebras with respect to a distinguished normal trace as defined
in \cite[Section III.3]{blackadar_operator_2006}. (In particular, we denote
von Neumann algebraic tensor products by $\otimes $ instead of the more
standard $\overline{\otimes }$.) More generally, we will adopt the
leg-numbering notation for linear operators on tensor products; see \cite[%
Notation 7.1.1]{timmermann_invitation_2008}.
\end{notation*}

Throughout the paper, we will tacitly use Borel selection theorems for Borel
relations with countable fibers; see \cite[Section 18.C]%
{kechris_classical_1995}.

\subsubsection*{Acknowledgments}

We are grateful to Alexander Kechris for a large number of useful
conversations throughout the preparation of the present manuscript. We also
thank Daniel Hoff and Anush Tserunyan for several helpful comments and
suggestions.

\section{Actions of groups containing a nonabelian free subgroup\label%
{Section:free}}

\subsection{Cocycles for measure preserving actions}

Let $\Gamma $ be a countable discrete group, and let $( X,\mu ) $ be the
standard atomless probability space. In the following, we will identify a
pmp action of $\Gamma $ on $( X,\mu ) $ with the corresponding action on the
tracial von Neumann algebra $(M,\tau)=L^{\infty } ( X,\mu ) $. (By a \emph{%
tracial von Neumann algebra }we mean a finite von Neumann algebra endowed
with a distinguished faithful normal tracial state.) Here $M$ is endowed
with the canonical tracial state $\tau ( f ) =\int fd\mu $ associated with
the measure $\mu $. We also identify the group $\mathrm{Aut} ( X,\mu ) $ of
measure-preserving automorphisms of $( X,\mu ) $ with the group $\mathrm{Aut}
( M,\tau ) $ of trace-preserving automorphisms of $M$.

If $\theta \colon \Gamma \rightarrow \mathrm{Aut}(M,\tau )$ is an action on
a tracial von Neumann algebra $(M,\tau )$, a (scalar, unitary) \emph{%
1-cocycle for $\theta $} is a function $w\colon \Gamma \rightarrow U(M)$
satisfying $w_{\gamma }\theta _{\gamma }(w_{\rho })=w_{\gamma \rho }$ for $%
\gamma ,\rho \in \Gamma $. The trivial cocycle for $\theta $ is the cocycle
constantly equal to $1$. Given $a,b\in M$, we write $a=b\ \mathrm{\mathrm{mod%
}}\mathbb{C}$ if there exists a nonzero $\lambda \in \mathbb{C}$ such that $%
a=\lambda b$. Two cocycles $w,w^{\prime }$ for $\theta $ are \emph{weakly
cohomologous }if there exists a unitary $z\in M$ such that%
\begin{equation*}
w_{\gamma }^{\prime }\theta _{\gamma }(z)=zw_{\gamma }\ \mathrm{mod}\mathbb{C%
}
\end{equation*}%
for every $\gamma \in \Gamma $. Given two cocycles $w,w^{\prime }$, we will
denote by $ww^{\prime }$ the cocycle $\gamma \mapsto w_{\gamma }w_{\gamma
}^{\prime }$.

\begin{definition}
Let $\theta$ be an action of a discrete group $\Gamma$ on a standard
probability space $(X,\mu)$, and let $\Delta \leq \Lambda \leq \Gamma$ be
nested subgroups.

\begin{enumerate}
\item A $\theta $-cocycle $w$ is said to be $\Delta $-\emph{invariant} if $%
w_{\delta }=1$ and $\theta _{\delta }(w_{\gamma })=w_{\gamma }$ for every $%
\gamma \in \Gamma $ and $\delta \in \Delta $. We write $Z_{:\Delta ,\mathrm{w%
}}^{1}(\theta )$ for the set of $\Delta $-invariant $1$-cocycles for $\theta
$, which we consider as a topological space with respect to the topology of
pointwise convergence in $2$-norm.

\item Two $\theta $-cocycles $w$ and $w^{\prime }$ are said to be $\Lambda $-%
\emph{relatively weakly} \emph{cohomologous} if there exists a unitary $z$
in $L^{\infty }(X,\mu )$ such that $w_{\lambda }^{\prime }\theta _{\lambda
}(z)=zw_{\lambda }$ $\mathrm{\ \mathrm{mod}}\mathbb{C}$ for every $\lambda
\in \Lambda $.

\item The $\Delta $-\emph{invariant weak }$1$\emph{-cohomology group} $%
H_{:\Delta ,\mathrm{w}}^{1} ( \theta ) $ is the space of weak cohomology
classes of $\Delta $-invariant cocycles for $\theta $, endowed with the
group operation defined by $[ w ] [ w^{\prime } ] = [ ww^{\prime}]$.

\item The $\Delta $-\emph{invariant} $\Lambda $-\emph{relative} \emph{weak} $%
1$-\emph{cohomology} \emph{group} $H_{:\Delta ,\Lambda ,\mathrm{w}}^{1} (
\theta ) $ is the space of $\Lambda $-relative weak cohomology classes of $%
\Delta $-invariant cocycles for $\theta $ endowed with the group operation
defined as above.
\end{enumerate}
\end{definition}

It is clear that the $\Delta $-invariant $\Lambda $-relative weak $1$%
-cohomology group $H_{:\Delta ,\Lambda ,\mathrm{w}}^{1} ( \theta ) $ of an
action $\theta $ is an invariant of $\theta $ up to conjugacy.
%(although it is not in general invariant up to cocycle conjugacy).

\subsection{Property\ (T) for triples of groups}

We fix a countable discrete group $\Gamma$ and nested subgroups $\Delta \leq
\Lambda \leq \Gamma $. Given a unitary representation $\pi\colon \Gamma\to
U(H)$ of $\Gamma$ on a Hilbert space $H$, a subset $F\subseteq \Gamma$, and $%
\varepsilon >0$, a unit vector $\xi \in H$ is said to be \emph{$(
F,\varepsilon ) $-invariant} if $\Vert \pi _{\gamma }(\xi) -\xi \Vert
<\varepsilon $ for every $\gamma \in F$. Moreover, $\xi $ is said to be
\emph{$\Delta $-invariant} for $\pi $ if $\pi _{\delta }(\xi) =\xi $ for
every $\delta \in \Delta $. The representation $\pi $ is said to have \emph{%
almost invariant $\Delta $-invariant vectors} if for every finite subset $%
F\subseteq \Gamma $ and every $\varepsilon >0$, there exists an $(
F,\varepsilon ) $-invariant $\Delta $-invariant unit vector for $\pi $.

\begin{definition}
\label{Definition:trivial-relative}The triple $\Delta \leq \Lambda \leq
\Gamma $ has \emph{property (T)} if every unitary representation of $\Gamma $
with almost invariant $\Delta $-invariant unit vectors has a $\Lambda $%
-invariant unit vector.
\end{definition}

When $\Delta \leq \Lambda \leq \Gamma $ has property (T), we also say that $%
\Lambda $ has the $\Delta $-invariant relative property (T) in $\Gamma $.
Moreover, if $\Delta \leq \Lambda \leq \Lambda $ has property (T), we say
that $\Lambda $ has the $\Delta $-invariant property (T).

\begin{example}
It is clear that when $\Delta $ is the trivial group, the triple $\Delta
\leq \Lambda \leq \Gamma $ has properly (T) if and only if the pair $\Lambda
\leq \Gamma $ has the relative property (T) in the usual sense; see \cite%
{delaroche_relation_1995,kazhdan_connection_1967,jolissaint_property_2005}.
\end{example}

\begin{example}
\label{eg:QuotTtripleT} It is also easy to observe that if $\Delta $ is a
normal subgroup of $\Lambda $ such that $\Lambda /\Delta $ has property (T),
then $\Lambda $ has the $\Delta $-invariant property (T). Indeed, under
these assumptions, for any unitary representation $\pi $ for $\Lambda $, the
space of $\Delta $-invariant vectors is $\Lambda $-invariant.
\end{example}

As in the case of property (T) for pairs of groups, one can provide several
equivalent reformulations of the notion of property (T) for triples of
groups. We collect some in Proposition~\ref{Propostion:trivial-relative}
below.

\begin{definition}
Let $\psi \colon\Gamma \to\mathbb{C}$ be a function.

\begin{itemize}
\item $\psi$ is said to be of \emph{conditionally negative type }if it
satisfies $\psi ( \gamma ^{-1} ) =\overline{\psi (\gamma )}$ for $\gamma \in
\Gamma $ and
\begin{equation*}
\sum_{i,j=1}^{n}\overline{\alpha }_{i}\alpha _{j}\psi ( \gamma
_{i}^{-1}\gamma _{j} ) \leq 0
\end{equation*}
for every $n\geq 1$, elements $\gamma _{1},\ldots ,\gamma _{n}\in \Gamma $,
and $\alpha _{1},\ldots ,\alpha _{n}\in \mathbb{C}$ satisfying $\alpha
_{1}+\cdots +\alpha _{n}=0$.

\item $\psi$ is said to be of \emph{positive type} if it satisfies
\begin{equation*}
\sum_{i,j=1}^{n}\overline{\alpha }_{i}\alpha _{j}\psi ( \gamma
_{i}^{-1}\gamma _{j} ) \geq 0
\end{equation*}
for every $n\geq 1$, elements $\gamma _{1},\ldots ,\gamma _{n}\in \Gamma $,
and $\alpha _{1},\ldots ,\alpha _{n}\in \mathbb{C}$.

\item \emph{normalized} if $\psi(1)=1$, and $\Delta $-\emph{invariant} if $%
\psi ( \gamma \delta ) =\psi ( \delta \gamma ) =\psi (\gamma )$ for every $%
\gamma \in \Gamma $ and $\delta \in \Delta$.
\end{itemize}
\end{definition}

Observe that if $\pi $ is a unitary representation of $\Gamma $ on $H$ and $%
\xi $ is a $\Delta $-invariant unit vector in $H$, then the function $\psi
(\gamma )= \langle \pi_{\gamma}(\xi) ,\xi \rangle $ is a normalized $\Delta $%
-invariant function of positive type on $\Gamma $. Conversely, the GNS
construction for functions of positive type shows that any normalized $%
\Delta $-invariant function of positive type on $\Gamma $ arises in this
fashion; see \cite[Appendix C]{bekka_kazhdans_2008}.

The same proof as \cite[Theorem 1.2]{jolissaint_property_2005} allows one to
prove the following characterization of property (T) for triples of groups.

\begin{proposition}
\label{Propostion:trivial-relative} Let $\Delta \leq \Lambda \leq \Gamma $
be nested discrete groups. Then the following are equivalent:

\begin{enumerate}
\item $\Delta \leq \Lambda \leq \Gamma $ has property (T);

\item there exist $\varepsilon >0$ and a finite subset $F\subseteq\Gamma $
such that whenever $\pi $ is a unitary representation of $\Gamma $, if $\pi $
has an $( F,\varepsilon ) $-invariant $\Delta $-invariant unit vector, then
the restriction of $\pi $ to $\Lambda $ contains a nonzero
finite-dimensional subrepresentation;

\item the restriction to $\Lambda $ of every $\Delta $-invariant
complex-valued function on $\Gamma $ which is conditionally of negative type
is bounded;

\item for every $\varepsilon >0$, there exist a finite subset $F\subseteq
\Gamma $ and $\delta>0$ such that whenever $\pi $ is a unitary
representation of $\Gamma $, if $\pi $ has an $( F,\delta ) $-invariant $%
\Delta $-invariant unit vector $\xi $, then there is a $\Lambda $-invariant
vector $\eta $ satisfying $\Vert \xi -\eta \Vert <\varepsilon $;

\item if $( \psi _{n} )_{n\in\mathbb{N}} $ is a sequence of normalized $%
\Delta $-invariant functions of positive type on $\Gamma $ which converges
pointwise to $1$, then $( \psi _{n}|_{\Lambda } )_{n\in\mathbb{N}} $
converges uniformly to $1$.
\end{enumerate}
\end{proposition}

Towards proving an extension of Popa's cocycle superrigidity theorem from
\cite{popa_cocycle_2007}, we present the following lemma, which is a natural
generalization of \cite[Lemma~4.2]{popa_cocycle_2007}.

\begin{lemma}
\label{lma:keyPopaSuperr} Let $\Delta \leq \Lambda \leq \Gamma $ be a triple
with property (T), and let $\theta \colon \Gamma \rightarrow \mathrm{Aut}%
(M,\tau )$ be an action on a tracial von Neumann algebra $(M,\tau )$. If $%
w\colon \Gamma \rightarrow U(M)$ is a $\Delta $-invariant cocycle for $%
\theta $, then for every $\varepsilon >0$ there exists a neighborhood $%
\Omega $ of $w$ in $Z_{:\Delta }^{1}(\theta )$ such that for all $w^{\prime
}\in \Omega $ there exists a partial isometry $v\in M$ with $w_{h}^{\prime
}\sigma _{\lambda }(v)=vw_{\lambda }$ for all $\lambda \in \Lambda $ and $%
\Vert v-1\Vert _{2}\leq \varepsilon $. Define the $w$-perturbation $\theta
^{w}$ of $\theta $ to be the action $\theta^w(\gamma)= \mathrm{Ad}\left(
w_{\gamma }\right) \circ \theta _{\gamma }$ for all $\gamma\in\Gamma$.
If $\theta ^{w}$ is ergodic,
then the restriction to $\Lambda $ of any $\Delta $-invariant cocycle in $%
\Omega $ is cohomologous to $w|_{\Lambda }$.
\end{lemma}

\begin{proof}
The proof is similar to that of \cite[Lemma~4.2]{popa_cocycle_2007}. Let $%
F\subseteq \Gamma $ and $\delta >0$ be as in part~(4) of Proposition~\ref%
{Propostion:trivial-relative} for $\varepsilon ^{2}/4$. Set
\begin{equation*}
\Omega =\{w^{\prime }\in Z_{:\Delta }^{1}(\theta )\colon \Vert w_{\gamma
}-w_{\gamma }^{\prime }\Vert _{2}\leq \delta \mbox{ for all }\gamma \in F\}.
\end{equation*}%
Given $w^{\prime }\in \Omega $, define a representation $\pi \colon \Gamma
\rightarrow U(L^{2}(M,\tau ))$ by $\pi _{\gamma }(x)=w_{\gamma }^{\prime
}\theta _{\gamma }(x)w_{\gamma }^{\ast }$ for all $\gamma \in \Gamma $ and
all $x\in L^{2}(M,\tau )$ coming from $M$. Then $\Vert \pi _{\gamma
}(1)-1\Vert _{2}\leq \delta $ for all $\gamma \in F$. Moreover, the unit
vector $1\in L^{2}(M,\tau )$ is $\Delta $-invariant because $w$ and $%
w^{\prime }$ are trivial on $\Delta $. By part~(4) of Proposition~\ref%
{Propostion:trivial-relative}, there exists a $\Lambda $-invariant unit
vector $\xi \in L^{2}(M,\tau )$ satisfying $\Vert \xi -1\Vert \leq
\varepsilon ^{2}/4$. Let $\xi _{0}\in L^{2}(M,\tau )$ denote the vector of
minimal norm in the weakly-closed convex hull of $\{\pi _{\lambda }(\xi
)\colon \lambda \in \Lambda \}$. By convexity, $\xi _{0}$ is also a unit
vector satisfying $\Vert \xi _{0}-1\Vert \leq \varepsilon ^{2}/4$. By
construction, we have $w_{\lambda }^{\prime }\theta _{\lambda }(\xi
_{0})w_{\lambda }^{\ast }=\xi _{0}$ for all $\lambda \in \Lambda $, so if $%
v\in M$ is the partial isometry in the polar decomposition of $\xi _{0}$
\cite[Subsection 1.2]{popa_cocycle_2007}, then $v$ is $\Delta $-invariant
and we have $w_{\lambda }^{\prime }\theta _{\lambda }(v)=vw_{\lambda }$ for
all $\lambda \in \Lambda $, as desired.

Finally, if $\theta ^{w}$ is ergodic, then $v^{\ast }v$, which belongs to
its fixed point algebra, must be a scalar. Thus $v$ is a unitary and the
restrictions of $w^{\prime }$ and $w$ to $\Lambda $ are cohomologous.
\end{proof}

Recall that the \emph{centralizer }$\mathrm{\mathrm{Aut}} (M) ^{\theta }$ of
an action $\theta $ of $\Gamma $ on $M$ is the subgroup of $\mathrm{Aut} (M)
$ consisting of the elements $\alpha $ such that $\alpha \circ \theta
_{\gamma }=\theta _{\gamma }\circ \alpha $ for every $\gamma \in \Gamma $.

\begin{definition}
We say that $\theta\colon \Gamma\to\mathrm{Aut}(M,\tau) $ is \emph{malleable
}if the connected component of the identity in $\mathrm{Aut} ( M\otimes M )
^{\theta \otimes \theta }$ contains an element of the form $( \alpha \otimes
\beta ) \circ \sigma $, where $\alpha ,\beta \in \mathrm{Aut} (M) ^{\theta }$
and $\sigma \in \mathrm{Aut} ( M\otimes M ) $ is the flip automorphism; see
\cite[Definition 2.9]{popa_cocycle_2007}.
\end{definition}

\begin{example}
The standard example of a malleable action is the Bernoulli action $\beta
_{\Gamma \curvearrowright I}$ associated with an action $\Gamma
\curvearrowright I$ of $\Gamma $ on an infinite set $I$; see \cite[Example
4.4 and Lemma 4.5]{popa_cocycle_2007} and \cite[Section 3]%
{vaes_rigidity_2007}. Such an action is weak mixing whenever the action $%
\Gamma \curvearrowright I$ has infinite orbits; see \cite[Lemma 4.5]%
{popa_cocycle_2007} and \cite[Proposition 2.1]{kechris_amenable_2008}. It is
moreover free whenever the action $\Gamma \curvearrowright I$ is faithful;
see \cite[Lemma 4.5]{popa_cocycle_2007} and \cite[Proposition 2.4]%
{kechris_amenable_2008}.
\end{example}

The following result is proved similarly to \cite[Theorem 5.2]%
{popa_cocycle_2007}, where \cite[Lemma~4.2]{popa_cocycle_2007} is replaced
with Lemma \ref{lma:keyPopaSuperr}; see also \cite[Lemma 4.9 and Lemma 4.10]%
{vaes_rigidity_2007}, \cite[Section 3, Section 4]{furman_popas_2007}, \cite[%
Section 30]{kechris_global_2010}.

\begin{theorem}[Popa]
\label{Theorem:trivial-superrigidity} Let $\Delta \leq \Lambda \leq \Gamma $
be a triple with property (T), and let $\theta $ be a malleable action of $%
\Gamma $ on the standard probability space $( X,\mu ) $ such that $\theta
|_{\Lambda }$ is weak mixing. Then the $\Delta $-invariant $\Lambda $%
-relative weak $1$-cohomology group $H_{:\Delta ,\Lambda ,\mathrm{w}}^{1} (
\theta ) $ is trivial.
\end{theorem}

\begin{proof}
The way that property (T) enters in Popa's argument in \cite[Theorem 5.2]%
{popa_cocycle_2007} is exclusively via \cite[Lemma 4.6]{popa_cocycle_2007},
where property (T) is used through \cite[Lemma 4.2]{popa_cocycle_2007}. In
our context, the analog of Popa's Lemma 4.6 for triples of groups with
property (T) and $\Delta$-invariant cocycles can be proved using Lemma~\ref%
{lma:keyPopaSuperr}. The proof of the present theorem then follows from \cite%
[Theorem~3.1 and Proposition~3.5]{popa_cocycle_2007}.
\end{proof}

\subsection{Actions with prescribed cohomology}

Theorem \ref{Theorem:trivial-superrigidity} allows one to construct, in the
presence of property (T), actions of groups with prescribed cohomology.

\begin{theorem}
\label{Theorem:prescribed-cohomology} Let $\Delta \leq \Lambda \leq \Gamma $
be a triple with property (T), and assume that $\Delta $ has infinite index
in $\Lambda $. Then there exists an assignment $A\mapsto \alpha _{A}$ from
countably infinite discrete abelian groups to free weak mixing actions of $%
\Gamma $ on the standard atomless probability space $(X,\mu )$ such that:

\begin{enumerate}
\item $\alpha _{A}|_{\Lambda }$ is weak mixing;

\item $A$ is isomorphic to $A^{\prime }$ if and only if $\alpha _{A}$ is
conjugate to $\alpha _{A^{\prime }}$;

\item for every action $\rho $ of $\Gamma $ on a standard probability space $%
(Y,\nu )$ such that $\rho |_{\Delta }$ is weak mixing, the $\Delta $%
-invariant, $\Lambda $-relative weak 1-cohomology group $H_{:\Delta ,\Lambda
,\mathrm{w}}^{1}(\alpha _{A}\otimes \rho )$ is isomorphic to $A$.
\end{enumerate}
\end{theorem}

\begin{proof}
Fix a countably infinite abelian group $A$, and let $G$ denote its
Pontryagin dual, endowed with its Haar (probability) measure $\nu$. Set $%
(M,\tau)=L^{\infty } ( G,\nu ) $ endowed with the trace-preserving action $%
\mathtt{Lt}\colon G\to\mathrm{Aut}(M,\tau)$ given by left translation. Set
\begin{equation*}
\beta=\beta _{\Gamma \curvearrowright \Gamma /\Delta }\otimes \beta _{\Gamma
\curvearrowright \Gamma }\colon \Gamma \curvearrowright M^{\otimes \Gamma
/\Delta }\otimes M^{\otimes \Gamma }\text{,}
\end{equation*}
where $\Gamma $ acts on $\Gamma /\Delta $ and on $\Gamma $ by left
translation. Define also the action
\begin{equation*}
\mathtt{Lt} ^{\otimes \Gamma /\Delta }\otimes \mathrm{id}_{M^{\otimes \Gamma
}}\colon G\curvearrowright M^{\otimes \Gamma /\Delta }\otimes M^{\otimes
\Gamma },
\end{equation*}
and let $M_A$ denote its fixed point algebra, which can be identified with
\begin{equation*}
M_A=(M^{\otimes \Gamma /\Delta }\otimes M^{\otimes \Gamma })^{ \mathtt{Lt}
^{\otimes \Gamma /\Delta }\otimes \mathrm{id}_{M^{\otimes \Gamma
}}}=(M^{\otimes \Gamma /\Delta })^{ \mathtt{Lt} ^{\otimes \Gamma /\Delta
}}\otimes M^{\otimes \Gamma }\text{.}
\end{equation*}
Notice that $M_A=L^{\infty } ( X_{A},\mu _{A} ) $ for a standard atomless
probability space $( X_{A},\mu _{A} ) $.

The actions defined above commute, so $\beta$ restricts to an action $%
\alpha_A\colon \Gamma\to\mathrm{Aut}(M_A)$. Since $\Lambda/\Delta$ is
infinite, $\beta|_{\Lambda}$ is weak mixing, and thus $\alpha_A|_{\Lambda}$
is weak mixing as well. It is clear that the conjugacy class of $\alpha _{A}$
only depends from the isomorphism class of $A$. This proves (1) and (2), so
we check (3).

Let $(Y,\nu )$ be a standard probability space, and let $\rho \colon \Gamma
\curvearrowright L^{\infty }(Y,\nu )$ be a trace-preserving action such that
$\rho |_{\Delta }$ is weak mixing. We claim that $H_{:\Delta ,\Lambda ,%
\mathrm{w}}^{1}(\alpha _{A}\otimes \rho )$ is isomorphic to $A$. Once we
prove this, the proof of the theorem will be complete.

Let $w$ be a $\Delta $-invariant cocycle for $\alpha _{A}\otimes \rho $.
Then $w$ is also a $\Delta $-invariant cocycle for $\beta \otimes \rho $.
Since $\rho |_{\Delta }$ is weak mixing and $\Delta $ is infinite, \cite[%
Proposition D.2]{vaes_rigidity_2007} implies that%
\begin{equation*}
(M^{\otimes \Gamma /\Delta }\otimes M^{\otimes \Gamma }\otimes L^{\infty
}(Y,\nu ))^{(\beta _{\Gamma \curvearrowright \Gamma /\Delta }\otimes \beta
_{\Gamma \curvearrowright \Gamma }\otimes \rho )|_{\Delta }}=M^{\otimes
\Gamma /\Delta }\otimes \mathbb{C}\otimes \mathbb{C}\subseteq M^{\otimes
\Gamma /\Delta }\otimes M^{\otimes \Gamma }\otimes L^{\infty }(Y,\nu ).
\end{equation*}%
Therefore $w_{\lambda }\in M^{\otimes \Gamma /\Delta }\otimes \mathbb{C}%
\otimes \mathbb{C}$ for every $\lambda \in \Lambda $. We will identify $%
M^{\otimes \Gamma /\Delta }\otimes \mathbb{C}\otimes \mathbb{C}$ with $%
M^{\otimes \Gamma /\Delta }$. Since $\Lambda /\Delta $ is infinite, the
shift $\beta _{\Lambda \curvearrowright \Gamma /\Delta }$ is weak mixing. By
Theorem \ref{Theorem:trivial-superrigidity}, there exists a unitary $v\in
M^{\otimes \Gamma /\Delta }$ such that
\begin{equation*}
v^{\ast }(\beta _{\Gamma \curvearrowright \Gamma /\Delta })_{\lambda
}(v)=w_{\lambda }\ \mathrm{mod}\mathbb{C}
\end{equation*}%
for every $\lambda \in \Lambda $. Fix $g\in G$. Then $\mathtt{Lt}%
_{g}^{\otimes \Gamma /\Delta }(w_{\lambda })=w_{\lambda }$ for every $%
\lambda \in \Lambda $, and hence
\begin{equation*}
v^{\ast }\mathtt{Lt}_{g}^{\otimes \Gamma /\Delta }(v)=(\beta _{\Gamma
\curvearrowright \Gamma /\Delta }\otimes \beta _{\Gamma \curvearrowright
\Gamma })_{\lambda }(v^{\ast }\mathtt{Lt}_{g}(v))\ \mathrm{mod}\mathbb{C}.
\end{equation*}%
In other words, $v^{\ast }\mathtt{Lt}_{g}^{\otimes \Gamma /\Delta }(v)$
generates a one-dimensional subspace which is invariant under $\beta
_{\Gamma \curvearrowright \Gamma /\Delta }\otimes \beta _{\Gamma
\curvearrowright \Gamma }$. Since this action is weak mixing, there exists $%
\chi _{w}\colon G\rightarrow \mathbb{T}$ such that $\mathtt{Lt}_{g}^{\otimes
\Gamma /\Delta }(v)=\chi _{w}(g)v$ for every $g\in G$. It is straightforward
to verify that $\chi _{w}$ is a group homomorphism, so that it can be
regarded as an element of $A\cong \widehat{G}$.

\textbf{Claim:} If $w,w^{\prime }\in Z_{:\Delta ,\mathrm{w}}^{1}(\alpha
_{A}\otimes \rho )$ are $\Lambda $-relatively weakly cohomologous, then $%
\chi _{w}=\chi _{w^{\prime }}$.

Given $w$ and $w^{\prime }$ as above, find a unitary $z\in M_{A}\otimes
L^{\infty }(Y,\nu )$ with $w_{\lambda }^{\prime }=z^{\ast }w_{\lambda
}(\alpha _{A}\otimes \rho )_{\lambda }(z)$ \textrm{mod}$\mathbb{C}$\emph{\ }%
for every $\lambda \in \Lambda $. Use $\Delta $-invariance of $w$ and $%
w^{\prime }$ to get $(\alpha _{A}\otimes \rho )_{\delta }(z)=z\ $for all $%
\delta \in \Delta $. This implies that $z\in M^{\otimes \Gamma /\Delta }$
and $w_{\lambda }^{\prime }=z^{\ast }w_{\lambda }(\beta _{\Gamma
\curvearrowright \Gamma /\Delta })_{\lambda }(z)$ \textrm{mod}$\mathbb{C}$
for every $\lambda \in \Lambda $. As above, there exists a unitary $v\in
M^{\otimes \Gamma /\Delta }$ such that
\begin{equation*}
w_{\lambda }=v^{\ast }(\beta _{\Gamma \curvearrowright \Gamma /\Delta
})_{\lambda }(v)\ \mathrm{mod}\mathbb{C}\ \ \mbox{ and }\ \ \mathtt{Lt}%
_{g}^{\otimes \Gamma /\Delta }(v)=\chi _{w}(g)v
\end{equation*}%
for every $\lambda \in \Lambda $ and every $g\in G$. Thus $w_{\lambda
}^{\prime }=(vz)^{\ast }(\beta _{\Gamma \curvearrowright \Gamma /\Delta
})_{\lambda }(vz)\ \mathrm{mod}\mathbb{C}$ for every $\lambda \in \Lambda $,
and
\begin{equation*}
\mathtt{Lt}_{g}^{\otimes \Gamma /\Delta }(vz)=\mathtt{Lt}_{g}^{\otimes
\Gamma /\Delta }(v)\mathtt{Lt}_{g}^{\otimes \Gamma /\Delta }(z)=\chi
_{w}(g)vz
\end{equation*}%
for every $g\in G$. This shows that $\chi _{w^{\prime }}=\chi _{w}$.

In particular, there is a well-defined assignment $\kappa \colon H_{:\Delta
,\Lambda ,\mathrm{w}}^{1}(\alpha _{A}\otimes \rho )\rightarrow \widehat{G}$
given by $\kappa ([w])=\chi _{w}$ for all $[w]\in H_{:\Delta ,\Lambda ,%
\mathrm{w}}^{1}(\alpha _{A}\otimes \rho )$. The map $\kappa $ is easily seen
to be a group homomorphism, so it remains to show that it is a bijection.

\textbf{Claim:} $\kappa$ is injective.

Let $w,w^{\prime }\in Z_{:\Delta ,\mathrm{w}}^{1}(\alpha _{A}\otimes \rho )$
satisfy $\chi _{w}=\chi _{w^{\prime }}=\chi \in G$. Then there exist
unitaries $v,v^{\prime }\in M^{\otimes \Gamma /\Delta }$ such that
\begin{equation*}
v^{\ast }(\beta _{\Gamma \curvearrowright \Gamma /\Delta })_{\lambda
}(v)=w_{\lambda }\ \mathrm{mod}\mathbb{C}\ \ \mbox{ and }v^{\prime \ast
}(\beta _{\Gamma \curvearrowright \Gamma /\Delta })_{\lambda }(v^{\prime
})=w_{\lambda }^{\prime }\ \mathrm{mod}\mathbb{C}
\end{equation*}%
for every $\lambda \in \Lambda $, and
\begin{equation*}
\mathtt{Lt}_{g}^{\otimes \Gamma /\Delta }(v)=\chi (g)v\ \ \mbox{
and }\ \ \mathtt{Lt}_{g}^{\otimes \Gamma /\Delta }(v^{\prime })=\chi
(g)v^{\prime }
\end{equation*}%
for every $g\in G$. Therefore, the unitary $z=v^{\ast }v^{\prime }\in M_{A}$
satisfies $z^{\ast }w_{\lambda }(\beta _{\Gamma \curvearrowright \Gamma
/\Delta })_{\lambda }(z)=w_{\lambda }^{\prime }\ \mathrm{mod}\mathbb{C}$ for
every $\lambda \in \Lambda $. This shows that $w$ and $w^{\prime }$ are $%
\Lambda $-relatively weakly cohomologous, proving the claim.

\textbf{Claim:} $\kappa$ is surjective.

Fix $\chi \in \widehat{G} $, and regard it as a (spectral) unitary in $%
C(G)\subseteq L^{\infty }(G,\nu )=M$. This gives a $\Delta $-invariant
unitary
\begin{equation*}
v_{\chi}\in M^{\otimes \Gamma /\Delta }\otimes \mathbb{C}\otimes \mathbb{C}%
\subset M^{\otimes \Gamma /\Delta }\otimes M^{\otimes \Gamma }\otimes
L^{\infty } ( Y,\nu )
\end{equation*}
satisfying $\mathtt{Lt} _{g}^{\otimes \Gamma /\Delta } ( v_{\chi} ) =\chi (
g ) v_{\chi}$ for every $g\in G$. Define $w_{\chi}(\lambda)=v_{\chi}^*%
\beta_{\lambda}(v_{\chi})$ for all $\lambda\in\Lambda$. Then $w_{\chi}$ is a
$\Delta $-invariant cocycle, and one checks that $\chi_{w_{\chi}}=\chi$, as
desired.
\end{proof}

We will often use Theorem \ref{Theorem:prescribed-cohomology} in the
particular case when $\Gamma =\Lambda $.

\subsection{Countable to one homomorphism\label{Subsection:rigid}}

Let $\Gamma $ be a countable discrete group, let $\Lambda$ be a subgroup,
and let $\Lambda \curvearrowright^{\alpha} (X,\mu)$ be an action on the
atomless standard probability space $( X,\mu )$. We recall the construction
of the coinduced action $\mathrm{CInd}_{\Lambda }^{\Gamma } ( \alpha ) $ of $%
\Gamma $ as defined in \cite{luck_type_2000,dooley_non-bernoulli_2008} and
used in \cite{ioana_orbit_2011,epstein_borel_2011}.

\begin{definition}
Adopt the notation above. Set
\begin{equation*}
Y=\{f\colon \Gamma \to X\colon f(\gamma
\lambda)=\alpha_{\lambda^{-1}}(f(\gamma)) \mbox{ for
every }\gamma \in \Gamma \mbox{ and } \lambda\in \Lambda \},
\end{equation*}
endowed with the measure $\nu $ induced by the product measure on $X$. Then $%
Y$ is an atomless standard probability space. The\emph{\ coinduced action} $%
\widehat{\alpha}=\mathrm{CInd}_{\Lambda }^{\Gamma } ( \alpha ) $ is the
action $\Gamma\curvearrowright^{\widehat{\alpha}}Y$ defined by setting $(%
\widehat{\alpha}_{\gamma_0}(f))_{\gamma }=f(\gamma _{0}^{-1}\gamma )$ for
every $\gamma ,\gamma_0\in \Gamma $.
\end{definition}

\begin{remark}
\label{Remark:null} Consider the map $\pi \colon Y\to X$ given by evaluation
at the unit of $\Gamma$. It is easy to see that $\{ f\in Y\colon \pi (
\widehat{\alpha}_{\gamma}(f) ) =\pi ( f ) \} $ is $\nu$-null for every $%
\gamma\in \Gamma\setminus\{1\}$.
\end{remark}

For $2\leq d\leq \infty $, fix a subgroup $\Lambda\leq \mathrm{SL}_{2} (
\mathbb{Z} ) $ isomorphic to $\mathbb{F}_{d}$. Then the canonical action of
\textrm{SL}$_{2} ( \mathbb{Z} )\curvearrowright \mathbb{Z}^{2}$ by group
automorphisms induces, by passing to the dual group, a free weak mixing pmp
action $\Lambda \curvearrowright^{\rho} \mathbb{T}^{2}$. This action has
been initially studied by Popa \cite{popa_class_2006}, and used by many
authors due to its rigidity properties; see \cite%
{gaboriau_uncountable_2005,ioana_orbit_2011,epstein_orbit_2007,epstein_borel_2011,tornquist_orbit_2006,ioana_subequivalence_2009}%
. A crucial property of such an action has been established by Ioana in \cite%
[Theorem 1.3]{ioana_orbit_2011} building on previous work from \cite%
{connes_factor_1980,gaboriau_uncountable_2005,popa_class_2006}; see also
\cite[Theorem 6.9]{epstein_borel_2011} and \cite[Lemma 7.4]%
{bowen_neumanns_2015}.

\begin{theorem}[Ioana]
\label{Theorem:separability-groups} Let $\Gamma $ be a countable discrete
group containing a nonabelian free group as a subgroup. Let $\mathcal{S}$ be
a class of free, ergodic, pmp actions of $\Gamma $ on the standard
probability space, and let $\mathcal{S}|_{\mathbb{F}_{d}}$ be the class of
restrictions of actions from $\mathcal{S}$ to $\mathbb{F}_{d}$. Suppose that

\begin{enumerate}
\item the actions in $\mathcal{S}$ are pairwise von Neumann equivalent,

\item the actions in $\mathcal{S}|_{\mathbb{F}_{d}}$ are pairwise not
conjugate;

\item every action $\sigma \colon \Gamma \curvearrowright X_{\sigma }$ in $%
\mathcal{S}$ has the property that $\rho $ is a factor of $\sigma |_{\Lambda
}$ as witnessed by a Borel map $\pi _{\sigma }\colon X_{\sigma }\to \mathbb{T%
}^{2}$ such that $\{ x\in X_{\sigma }\colon \pi _{\sigma } (
\sigma_{\gamma}(x)) =\pi _{\sigma } ( x ) \} $ is null for every $\gamma\in
\Gamma\setminus\{1\} $.
\end{enumerate}

Then $\mathcal{S}$ is countable.
\end{theorem}

We now prove Theorems~\ref{Theorem:nB-conj} and~\ref{Theorem:nB-oe} from the
introduction, in the case when $\Gamma $ contains $\mathbb{F}_{2}$.

\begin{theorem}
\label{Theorem:c-to-1} Let $\Gamma $ be a group containing a nonabelian free
group as a subgroup $\Lambda$, and fix an infinite normal subgroup $\Delta$
of $\Lambda $ such that the quotient $\Lambda /\Delta $ is an infinite
property (T) group. Then there exists an assignment $A\mapsto \theta _{A}$
from countably infinite abelian groups to free weak mixing actions of $%
\Gamma $ on the standard atomless probability space such that:

\begin{enumerate}
\item $\theta _{A}|_{\Lambda }$ is weak mixing;

\item Two groups $A$ and $A^{\prime }$ are isomorphic if and only if $\theta
_{A}$ and $\theta _{A^{\prime }}$ are conjugate;

\item if $\mathcal{A}$ is a family of pairwise nonisomorphic countably
infinite abelian groups such that the actions $\{ \theta _{A}\colon A\in
\mathcal{A} \} $ are pairwise von Neumann equivalent, then $\mathcal{A}$ is
countable.
\end{enumerate}
\end{theorem}

\begin{proof}
Let $\Lambda \curvearrowright ^{\rho }\mathbb{T}^{2}$ be the rigid action
described above, and set $\widehat{\rho }=\mathrm{CInd}_{\Lambda }^{\Gamma
}(\rho )$. Let $A\mapsto \alpha _{A}$ be the assignment from countably
infinite abelian groups to weak mixing measure preserving actions of $\Gamma
$ obtained in Theorem \ref{Theorem:prescribed-cohomology} for the triple $%
\Delta \leq \Lambda \leq \Gamma $. (Observe that $\Delta \leq \Lambda \leq
\Gamma $ has property (T) by Example~\ref{eg:QuotTtripleT}.) By
construction, $\alpha _{A}|_{\Lambda }$ is weak mixing, and hence $\widehat{%
\rho }|_{\Delta }$ is weak mixing by \cite[Lemma 6.8]{epstein_borel_2011}.

Fix a countably infinite discrete abelian group $A$, and set $\theta
_{A}=\alpha _{A}\otimes \widehat{\rho}$. Then $\theta _{A}|_{\Lambda }$ is a
weak mixing extension of $\rho $. If $A$ and $A^{\prime }$ are isomorphic,
then $\alpha _{A}$ and $\alpha _{A^{\prime }}$ are conjugate, and hence $%
\theta _{A}$ and $\theta _{A^{\prime }}$ are conjugate. Conversely, if $%
\theta _{A}$ and $\theta _{A^{\prime }}$ are conjugate, then $A\cong
H_{:\Delta ,\mathrm{w}}^{1} ( \theta _{A}|_{\Lambda } ) \cong H_{:\Delta ,%
\mathrm{w}}^{1} ( \theta _{A^{\prime }}|_{\Lambda } ) \cong A^{\prime }$ by
Theorem~\ref{Theorem:prescribed-cohomology}.

Suppose now that $\mathcal{A}$ is a family of pairwise nonisomorphic
countably infinite abelian groups such that $\{ \theta _{A}\colon A\in
\mathcal{A} \} $ are pairwise von\ Neumann equivalent. By Theorem \ref%
{Theorem:prescribed-cohomology}, we have $H_{:\Delta ,\mathrm{w}}^{1} (
\theta _{A}|_{\Lambda } ) \cong A$ for every $A\in \mathcal{A}$. Therefore
the actions $\{ \theta _{A}|_{\Lambda }\colon A\in \mathcal{A} \} $ are
pairwise not conjugate weak mixing extensions of $\rho $. By Theorem \ref%
{Theorem:separability-groups}, we conclude that $\mathcal{A}$ is countable.
\end{proof}

The proof in the general case of nonamenable $\Gamma $ follows similar
ideas, but a lot more work is required. The induction argument presented in
this section is not sufficient for our purposes, and this forces us to
replace countable groups with countable pmp groupoids. The proof of
Theorem~A and Theorem~B can then be obtained by combining this with
Gaboriau-Lyons' measurable solution to von Neumann's problem \cite%
{gaboriau_measurable-group-theoretic_2009} and Epstein's coinduction
construction \cite{epstein_orbit_2007}; the main idea is to produce
actions of an arbitrary nonamenable group whose (weak, invariant, relative)
1-cohomology can be prescribed via superrigidity. The abstract setting of pmp
groupoids can in fact be used to prove the more general Theorem~D and
Theorem~E, from which Theorems~A and B follow.

\section{Actions of groupoids\label{Section:actions}}

\subsection{Groupoids\label{Subsection:groupoids}}

We recall here some basic definitions concerning groupoids as can be found,
for instance, in \cite%
{bowen_entropy_2014,paterson_groupoids_1999,ad_amenable_2000,renault_groupoid_1980}%
. A \emph{groupoid} $G$ is a small category where every morphism (also
called \emph{arrow}) is invertible. We denote by $G^{0}$ the set of objects
of $G$, which are also called \emph{units }of $G$, while $G^{0}$ is called
the \emph{unit space }of $G$. We identify each unit of $G$ with the
corresponding identity arrow. Consistently, we regard $G^{0}$ as a subset of
$G$. The source and range maps are denoted by $s,r\colon G\to G^{0}$,
respectively. We let $G^{2}$ denote the set the pairs of \emph{composable
arrows}:
\begin{equation*}
G^{2}=\{( \gamma ,\rho ) \in G\colon s(\gamma )=r(\rho )\}.
\end{equation*}

A \emph{bisection} for a groupoid $G$ is a subset $t\subseteq G$ such that
source and range maps are injective on $t$. Given subsets $A,B\subseteq G$,
we set
\begin{equation*}
AB=\{\gamma \rho \colon ( \gamma ,\rho ) \in G^{2}\cap ( A\times B )\} \ \ %
\mbox{ and } \ \ A^{-1}= \{ \gamma ^{-1}\colon \gamma \in A \}.
\end{equation*}
For $\gamma \in G$ and $A\subseteq G$, we set $\gamma A= \{ \gamma \} A$ and
$A\gamma =A \{ \gamma \} $. In particular, $Ax= \{ \gamma \in A\colon
s(\gamma )=x \} $ and $xA= \{ \gamma \in A\colon r(\gamma )=x \} $ for $x\in
G^0$. If $A$ and $B$ are bisections, then so are $AB$ and $A^{-1}$.

\begin{definition}
Let $G$ and $H$ be groupoids with $H\subseteq G$.

\begin{itemize}
\item We say that $G$ is a \emph{standard Borel groupoid} if it is endowed
with a standard Borel structure such that $G^{0}$ is a Borel subset of $G$,
and composition and inversion of arrows, as well as the source and range
maps, are Borel functions.

\item We say that $H$ is a \emph{Borel subgroupoid} of $G$ if $H^{0}=G^{0}$,
and the inclusion map $G\hookrightarrow H$ is a Borel homomorphism.

\item We say that $G$ is a \emph{discrete measurable groupoid} \cite%
{dykema_sofic_2014,bowen_entropy_2014,bowen_neumanns_2015} (also called
\emph{countable Borel groupoid} \cite%
{renault_topological_2015,ad_haagerup_2013}), if it is a standard Borel
groupoid whose source and range maps are countable-to-one.

\item We say that $G$ is a \emph{discrete probability-measure-preserving }%
(pmp) \emph{groupoid} if it is a discrete measurable groupoid endowed with a
Borel measure $\mu _{G}$ satisfying
\begin{equation*}
\mu _{G} ( A ) =\int_{x\in G^{0}} \vert xA \vert d\mu _{G} ( x ) =\int_{x\in
G^{0}} \vert Ax \vert d\mu _{G} ( x )
\end{equation*}%
for every Borel subset $A$ of $G$.

\item If $G$ is a discrete pmp groupoid, we say that $H$ is a \emph{discrete
pmp subgroupoid} if it is a Borel subgroupoid with $\mu_H=\mu_G$.

\item A Borel subset $A\subseteq G^{0}$ is said to be \emph{invariant }if $r
( GA ) =A$.

\item A discrete pmp groupoid $G$ is said to be \emph{ergodic }if every
non-null invariant Borel subset of $G^{0}$ has full measure.
\end{itemize}
\end{definition}

For a non-null invariant Borel subset $A$ of $G^{0}$ one can define the
\emph{reduction} $G|_{A}=AGA$, which is a discrete pmp groupoid with unit
space $A$ with respect to the measure $\frac{1}{\mu _{G}(A)}\mu _{G}|_{A}$.
If $A$ has full measure, such a reduction is called \emph{inessential}. In
the following, we will identify two discrete pmp groupoids whenever they
have isomorphic inessential reductions.

\begin{remark}
Let $G$ be a discrete pmp groupoid with measure $\mu_G$. Then $\mu _{G}$ is
completely determined by its restriction $\mu _{G^{0}}$. Moreover, $\mu _{G}$
induces a measure $\mu _{G^{2}}$ on $G^{2}$ defined by
\begin{equation*}
\mu _{G^2} ( A ) =\int_{x\in G^{0}} \vert A\cap ( Gx\times xG ) \vert d\mu
_{G} ( x ) \text{.}
\end{equation*}
\end{remark}

\begin{definition}
Let $G$ be a discrete pmp groupoid. We define the \emph{inverse semigroup of
partial automorphisms} of $G$ to be the set $[ [ G ] ] $ of Borel bisections
of $G$ (identified when they agree almost everywhere).

Similarly, the \emph{full group} $[G]$ of $G$ is defined as the group, under
composition, of Borel bisections $t$ of $G$ (identified when they agree
almost everywhere) with the property that source and range maps restricted
to $t$ are onto.
\end{definition}

One checks that the topology induced by the metric $d ( t_{0},t_{1} ) =\mu (
t_{0}\bigtriangleup t_{1} ) $ turns $[ G ] $ into a Polish group. For a
Borel bisection $t$ of $G$ and $x\in G^{0}$, we let $tx$ be the unique
element of $t$ with source $x$, and $xt $ be the unique element of $t$ with
range $x$. We say that a subset $S$ of $[ [ G ] ] $ \emph{covers} $G$ if the
union of $S$ is co-null in $G$.

\begin{definition}
Let $G$ be a discrete pmp groupoid, and let $H$ is a standard Borel
groupoid. A \emph{homomorphism} $\pi\colon G\to H$ is a Borel functor such
that for a.e.\ $\gamma ,\rho \in G$, $\pi ( s( \gamma )
) =s( \pi ( \gamma ) ) $, $\pi ( r(
\gamma ) ) =r( \pi ( \gamma ) ) $, and $%
\pi ( \gamma \rho ) =\pi ( \gamma ) \pi ( \rho
) $. (By \cite[Lemma 5.2]{ramsay_virtual_1971}, $\pi$ can be seen as a
Borel functor from an inessential reduction of $G$ to $H$.)

Let $G$ and $H$ be discrete pmp groupoids, and let $\pi\colon G\to H$ be a
homomorphism. Following \cite[Section 4]{bowen_entropy_2014}, \cite[%
Definition 2.3]{bowen_neumanns_2015}, \cite[Section 16]{kechris_spaces_2017}%
, we say that $\pi$ is:

\begin{itemize}
\item a \emph{pmp extension }if $\mu _{H^{0}}$ is equal to the push-forward
measure $\pi _{\ast }(\mu _{G^{0}})$;

\item \emph{class-bijective} if it maps $Gx$ bijectively onto $H\pi ( x ) $
for almost every $x\in G^{0}$.
\end{itemize}
\end{definition}

\begin{definition}
The \emph{orbit equivalence relation} $E_{G}$ of a discrete pmp groupoid $G$
is the countable pmp equivalence relation on $G^{0}$ obtained as the image
of the map $(r,s)\colon G\to G^{0}\times G^{0}$. The groupoid $G$ is called
\emph{principal }if such a map is one-to-one. In this case, one can identify
$G$ with $E_G$.
\end{definition}

%It is clear that $R$ is
%ergodic if and only if $G$ is ergodic. Furthermore, the map $ [ G ]
%\to  [ G ] $ mapping a section $t$ to the function $%
%x\mapsto r ( tx ) $, is a continuous surjection.

In the following, we will regard any countable pmp equivalence relation as a
discrete pmp groupoid. For every ergodic countable pmp equivalence relation $%
R$ on a standard probability space $(X,\mu )$, there exists $\theta \in
\lbrack R]$ acting ergodically on $(X,\mu )$ \cite{kechris_global_2010}.
Therefore, a groupoid $G$ is ergodic if and only if there exists $t\in
\lbrack G]$ such that the map $x\mapsto r(tx)$ acts ergodically on $G^{0}$.
Furthermore, when $G$ is ergodic, there exist countably many pairwise
essentially disjoint elements of $[G]$ which cover $G$, and if $A,B\subseteq
G^{0}$ are Borel sets satisfying $\mu _{G^{0}}(A)=\mu _{G^{0}}(B)$, then
there exists $t\in \lbrack G]$ such that $r(tx)\in B$ for almost every $x\in
A$ and $s(yt)\in A$ for almost every $y\in B$.

\subsection{Bundles of metric spaces and Hilbert spaces\label%
{Subsections:Hilbert-bundles}}

Let $X$ be a standard probability space. We say that a standard Borel space $%
Z$ is \emph{fibered over }$X$ if it is endowed with a Borel surjection $%
q\colon Z\to X$. In this case, we denote by $Z_{x}=q^{-1} \{ x \} $ the
\emph{fiber over }$x\in X$. A (Borel) \emph{section }for $Z$ is a Borel
function $\sigma \colon X\to Z$ such that $q\circ \sigma=\mathrm{id}_X $.
For a section $\sigma$ and $x\in X$, we write $\sigma _{x}=\sigma(x)$.

\begin{definition}
Let $Z$ and $Z^{\prime }$ be standard Borel bundles over $X$.

\begin{enumerate}
\item The \emph{fibered product} $Z\ast Z^{\prime }$ is the standard Borel
space fibered over $X$ defined by%
\begin{equation*}
Z\ast Z^{\prime }= \{ ( z,z^{\prime } ) \in Z\times Z^{\prime }\colon
q(z)=q^{\prime }(z^{\prime }) \}.
\end{equation*}

\item A Borel \emph{fibered function} $f\colon Z\to Z^{\prime }$ is a Borel
map satisfying $q^{\prime }\circ f=q$ almost everywhere.
\end{enumerate}
\end{definition}

When $( Z,\lambda ) $ is a standard probability space, we say that $(
Z,\lambda ) $ is \emph{fibered over the standard probability space} $( X,\mu
) $ with respect to the Borel surjection $q\colon Z\to X$ if $q_{\ast
}(\lambda) =\mu $. In this case, one can consider the disintegration $(
\lambda _{x} ) _{x\in X}$ of $\lambda $ with respect to $\mu $. This turns
each fiber $( Z_{x},\lambda _{x} ) $ into a standard probability space.

\begin{definition}
A \emph{(Borel) bundle of metric spaces }over a standard Borel space $X$ is
a standard Borel space $\mathcal{Z}$ fibered over $X$ endowed with a Borel
function $d\colon \mathcal{Z}\ast \mathcal{Z}\to \mathbb{R}$ with the
following properties:

\begin{itemize}
\item for every $x\in X$, the restriction $d_{x}$ of $d$ to $\mathcal{Z}%
_{x}\times \mathcal{Z}_{x}$ is a metric on $\mathcal{Z}_{x}$;

\item there exists a sequence $( \sigma _{n} )_{n\in\mathbb{N}}$ of sections
for $\mathcal{Z}$ such that $\{ \sigma _{n,x}\colon n\in \mathbb{N} \} $ is
a dense subset of $\mathcal{Z}_{x}$, for every $x\in X$.
\end{itemize}

When $(X,\mu)$ is a standard probability space, the space $S ( X,\mathcal{Z}
) $ of sections for $\mathcal{Z}$ has a canonical topology induced by the
metric%
\begin{equation*}
d ( b,b^{\prime } ) =\int_X \frac{d_{x} ( b_{x},b_{x}^{\prime } ) }{1+d_{x}
( b_{x},b_{x}^{\prime } ) }d\mu ( x ) \text{.}
\end{equation*}
\end{definition}

When $X$ is the unit space of a discrete pmp groupoid $G$, we let $S ( G,%
\mathcal{Z} ) $ denote the space of Borel functions $z\colon G\to \mathcal{Z}
$ satisfying $z_{\gamma }\in \mathcal{Z}_{r(\gamma )}$ for all $\gamma\in G$%
, endowed with the canonical topology induced by the pseudometrics%
\begin{equation*}
d_{t} ( z,z^{\prime } ) =\int_X \frac{ \Vert z_{tx}-z_{tx}^{\prime } \Vert }{%
1+ \Vert z_{tx}-z_{tx}^{\prime } \Vert }d\mu ( x ),
\end{equation*}%
for $z,z^{\prime }\in S(G,\mathcal{Z})$, where $t$ ranges within (a dense
subset of) $[ G ] $.

\begin{definition}
A \emph{(Borel) Hilbert bundle} over a standard Borel space $X$ is a
standard Borel space $\mathcal{H}$ fibered over $X$ endowed with Borel
fibered functions $\boldsymbol{0}\colon X\to \mathcal{H}$, $+\colon \mathcal{%
H\ast H}\to \mathcal{H}$, $\cdot \colon \mathcal{H}\ast \mathcal{H}\to
\mathcal{H}$, and $\mathbb{C}\times \mathcal{H}\to \mathcal{H}$, together
with a sequence of sections $( \sigma _{n} ) _{n\in \mathbb{N}}$ such that,
for every $x\in X$, the fiber $\mathcal{H}_{x}$ is a Hilbert space when
endowed with the operations induced by the given Borel fibered functions,
and $(\sigma _{n,x})_{n\in \mathbb{N}}$ enumerates a subset of the unit
sphere of $\mathcal{H}_{x}$ with dense linear span.

We denote the Hilbert bundle $\mathcal{H}$ also by $\bigsqcup_{x\in X}%
\mathcal{H}_{x}$, and the space of sections by $S ( X,\mathcal{H} ) $.
\end{definition}

\begin{remark}
The Grahm-Schmidt orthogonalization process shows that one can always assume
that $( \sigma _{n,x} ) _{n\in \mathbb{N}}$ is an orthonormal basis for $%
\mathcal{H}_{x}$. In this case, the sequence $( \sigma _{n} ) _{n\in \mathbb{%
N}}$ is called an \emph{orthonormal basic sequence} for $\mathcal{H}$. The
fibered function $\mathcal{H}\ast \mathcal{H}\to \mathbb{R}$ given by $( x,y
) \mapsto \Vert x-y \Vert $ turns $\mathcal{H}$ into a bundle of metric
spaces over $X$.
\end{remark}

We denote by $L^{\infty } ( X,\mathcal{H} ) $ the space of (essentially)
bounded sections of $\mathcal{H}$, and by $L^{2} ( X,\mathcal{H} ) $ the
space of square-integrable sections of $\mathcal{H}$. The latter is a
Hilbert space with respect to the inner product $\langle \xi,\eta \rangle
=\int_{X} \langle \xi _{x},\eta _{x} \rangle d\mu ( x ) $ for all $%
\xi,\eta\in L^2(X,\mathcal{H})$. A section $\xi \colon X\to \mathcal{H}$ is
a \emph{unit section }if $\Vert \xi _{x} \Vert =1$ for almost every $x\in X$.

%As in the case of bundles of metric
%spaces, if $G$ is a discrete pmp groupoid and $\mathcal{H}$ is a Hilbert
%bundle on $G^{0}$, we let $S ( G,\mathcal{H} ) $ be the space
%of Borel functions $z\colon G\to \mathcal{H}$ satisfying
%$z_{\gamma}\in \mathcal{H}_{r(\gamma)}$ for all $\gamma\in G$,
%endowed with the topology described above.

\subsection{Representations of groupoids\label{Subsection:representations}}

Given a Hilbert bundle $\mathcal{H}$ over a standard probability space $X$,
its \emph{unitary groupoid} $U ( \mathcal{H} ) $ is the set of unitary
operators $U^{(s,t)}\colon \mathcal{H}_{s}\to \mathcal{H}_{t}$ for $s,t\in X$%
. This is a standard Borel groupoid when endowed with the standard Borel
structure generated by the source and range maps together with the functions
$U^{(s,t)} \mapsto \langle \sigma _{n,t},U\sigma _{m,s} \rangle $ for $%
n,m\in \mathbb{N}$. The unit space of $U ( \mathcal{H} ) $ can be identified
with $X$.

\begin{definition}
\label{df:WMErgRep} Let $G$ be a discrete pmp groupoid, and let $\mathcal{H}$
be a Hilbert bundle over $G^{0}$. A (unitary) \emph{representation }of $G$
on $\mathcal{H} $ is a homomorphism $\pi \colon G\to U ( \mathcal{H} ) $
that fixes $G^0$ pointwise. Moreover, a vector $\xi\in L^{2} ( G^{0},%
\mathcal{H} ) $ is said to be $\pi $-\emph{invariant }if $\pi _{\gamma }(\xi
_{s(\gamma )})=\xi _{r(\gamma )}$ for almost every $\gamma \in G$. The space
of $\pi $-invariant vector is denoted by $L^{2} ( G^{0},\mathcal{H} ) ^{\pi
} $.

A representation $\pi\colon G\to U(\mathcal{H})$ is said to be

\begin{enumerate}
\item \emph{ergodic}, if $L^{2} ( G^{0},\mathcal{H} ) ^{\pi }$ is the
trivial subspace of $L^{2} ( G^{0},\mathcal{H} ) $.

\item \emph{weak mixing}, if for every $\varepsilon >0$, every $n\in \mathbb{%
N}$, and sections $\xi _{1},\ldots ,\xi _{n}$ for $\mathcal{H}$, there
exists $t\in [ G ] $ such that $\int_{G^0}|\langle \xi _{j,x},\pi _{xt}(\xi
_{i,s ( xt ) }) \rangle | d\mu _{G^{0}} ( x ) \leq \varepsilon $ for every $%
i,j=1,\ldots,n$.
\end{enumerate}
\end{definition}

\begin{definition}
Let $\mathcal{H}$ be a Hilbert bundle over a standard probability space $%
(X,\mu)$. A sub-bundle $\mathcal{K}$ of $\mathcal{H}$ is a Hilbert bundle $%
\mathcal{K}$ over $X$ such that $\mathcal{K}_{x}$ is a subspace of $\mathcal{%
H}_{x}$ for every $x\in X$. Moreover, we say that $\mathcal{K}$ is

\begin{itemize}
\item \emph{finite-dimensional} if there exists $d\in \mathbb{N}$ such that $%
\dim \mathcal{K}_{x}\leq d$ for almost every $x\in X$;

\item $d$-\emph{dimensional} for some $d\in \mathbb{N}\cup \{ \infty \} $ if
$\dim \mathcal{K}_{x}=d$ for almost every $x\in X$;

\item \emph{nonzero} if there exists a non-null Borel subset $A$ of $X$ such
that $\dim \mathcal{K}_{x}>0$ for every $x\in A$.
\end{itemize}

When $G$ is a groupoid with $G^0=X$ and $\pi\colon G\to U(\mathcal{H})$ is a
representation, we say that $\mathcal{K}$ is $\pi $-\emph{invariant} if $\pi
_{\gamma }(\mathcal{K}_{s(\gamma )})=\mathcal{K}_{r(\gamma )}$ for every $%
\gamma \in G$. We further denote by $\mathrm{Proj}_{\mathcal{K}_x}$ the
orthogonal projection from $\mathcal{H}_x$ onto $\mathcal{K}_x$ for $x\in X$.

The restriction of $\pi $ to a (finite-dimensional) nonzero $\pi $-invariant
sub-bundle is called a \emph{(finite-dimensional) subrepresentation}.
\end{definition}

Let $\pi \colon G\to U(\mathcal{H})$ and $\sigma \colon G\to U(\mathcal{K})$
be representations. Define their fiber-wise tensor product as follows. Set
\begin{equation*}
\mathcal{H}\otimes \mathcal{K}=\bigsqcup_{x\in G^{0}} ( \mathcal{H}%
_{x}\otimes \mathcal{K}_{x} ),
\end{equation*}
which is also a Hilbert bundle over $G^0$, and define $\pi \otimes \sigma
\colon G \to U(\mathcal{H}\otimes \mathcal{K})$ by $( \pi \otimes \sigma )
_{\gamma }=\pi _{\gamma }\otimes \sigma _{\gamma }$ for $\gamma \in G$. One
can also define the conjugate representation $\overline{\pi }$ of $G$ on $%
\overline{\mathcal{H}}=\bigsqcup_{x\in \mathcal{H}_{x}}\overline{\mathcal{H}}%
_{x}$ in a similar way.

\begin{remark}
\label{rem:HS} In the context of the comments above, $\mathcal{H}\otimes
\overline{\mathcal{K}}$ can be identified with the Hilbert-Schmidt bundle
\textrm{HS}$( \mathcal{K},\mathcal{H} ) =\bigsqcup_{x\in G}\mathrm{HS} (
\mathcal{K}_{x},\mathcal{H}_{x} ) $, where $\mathrm{HS} ( \mathcal{K}_{x},%
\mathcal{H}_{x} ) $ denotes the space of Hilbert-Schmidt operators $\mathcal{%
K}_{x}\to \mathcal{H}_{x}$. If $\vert \xi \rangle \langle \eta \vert \colon
\mathcal{K}_{x}\to \mathcal{H}_{x}$ is the rank-one operator $\vert \zeta
\rangle \mapsto \langle \eta ,\zeta \rangle \vert \xi \rangle $, then the
isomorphism is induced by the assignment $\xi \otimes \overline{\eta }%
\mapsto \vert \xi \rangle \langle \eta \vert $ for $\eta \in \mathcal{K}_{x}
$ and $\xi \in \mathcal{H}_{x}$.

Under this identification, the representation $\pi \otimes \overline{\sigma }
$ can be identified with the representation on \textrm{HS}$( \mathcal{K},%
\mathcal{H} ) $ defined by $T\mapsto \pi _{\gamma }T\sigma _{\gamma }^{\ast
} $ for $\gamma \in G$ and $T\in \mathrm{HS} ( \mathcal{K}_{s(\gamma )},%
\mathcal{H}_{s(\gamma )} ) $.
\end{remark}

\begin{definition}
Let $G$ be a discrete pmp groupoid, and let $\pi $ be a representation of $G$
on a Hilbert bundle $\mathcal{H}$. Then $\pi $ induces a (group)
representation $[ \pi ] \colon [ G ] \to U(L^{2} ( G^{0},\mathcal{H} )) $
defined by $[ \pi ] _{t}(\xi) = ( \pi _{xt} ( \xi _{s ( xt ) } ) ) _{x\in
G^{0}}$ for all $t\in [ G ] $ and all $\xi \in L^{2} ( G^{0},\mathcal{H} )$.

Similarly, $\pi$ induces a (semigroup) representation $[ [ \pi ] ] \colon [
[ G ] ] \to U(L^{2} ( G^{0},\mathcal{H} )) $ defined by
\begin{equation*}
( [ [ \pi ] ] _{\sigma }\xi ) _{x}=\left\{
\begin{array}{ll}
\pi _{x\sigma } ( \xi _{s ( x\sigma ) } ), &
\hbox{if $x\in
\sigma\sigma^{-1}$;} \\
0, & \hbox{otherwise.}%
\end{array}
\right.
\end{equation*}
for all $\sigma\in [[G]]$, all $\xi \in L^{2} ( G^{0},\mathcal{H} )$, and
all $x\in G^0$.
\end{definition}

\begin{remark}
One can check that a representation $\pi$ of a groupoid on a Hilbert bundle
is weak mixing in the sense of Definition~\ref{df:WMErgRep} if and only if
the associated representation $[\pi]$ is weak mixing in the usual sense.
\end{remark}

The proof of the following lemma is immediate, so we leave it to the reader.

\begin{lemma}
\label{Lemma:fixed-point-representation} Let $\pi $ be a representation of
an ergodic discrete pmp groupoid $G$ on a Hilbert bundle $\mathcal{H}$. For
an element $\xi = ( \xi _{x} ) _{x\in G^{0}}$ of $L^{2} ( G^{0},\mathcal{H}
) $, the following assertions are equivalent:

\begin{enumerate}
\item $\xi $ is fixed by $[ \pi ]$;

\item there exists a countable subset $S\subseteq [ G ] $ that covers $G$
such that $[\pi]_t(\xi)=\xi$ for every $t\in S$;

\item $\pi _{\gamma } ( \xi _{s(\gamma )} ) =\xi _{r(\gamma )}$ for almost
every $\gamma \in G$.
\end{enumerate}
\end{lemma}

Several standard facts about representations of discrete groups admit
natural generalizations to the setting of representations of discrete pmp
groupoids. We present here some of them, together with the main ideas used
in their proofs.

\begin{proposition}
\label{Proposition:equivalences-groupoid-rep} Let $G$ be an ergodic pmp
groupoid, and let $\pi\colon G\to U(H)$ be a representation.

\begin{enumerate}
\item If there exist $c>0$ and a unit section $\xi \in L^{2} ( G^{0},%
\mathcal{H} ) $ such that $\langle \xi , [ \pi ] _{t}\xi \rangle \geq c$ for
every $t\in [ G ] $, then $\pi $ contains a nonzero invariant section.

\item The following assertions are equivalent:

\begin{enumerate}
\item[(2.a)] The representation $\pi \otimes \overline{\pi }$ contains
invariant vectors;

\item[(2.b)] The representation $\pi \otimes \sigma $ contains invariant
vectors for some representation $\sigma$ of $G$;

\item[(2.c)] The representation $\pi $ contains a finite-dimensional
subrepresentation.
\end{enumerate}

\item The following assertions are equivalent:

\begin{enumerate}
\item[(3.a)] $\pi $ is weak mixing;

\item[(3.b)] $\pi $ has no finite-dimensional subrepresentations;

\item[(3.c)] If $\mathcal{K}$ is a finite-dimensional sub-bundle of $%
\mathcal{H}$, then for every $\varepsilon >0$, there exists $t\in [ G ] $
such that, for every unit section $\xi $ for $\mathcal{K}$, $\int_{G^0} ||%
\mathrm{Proj}_{\mathcal{K}_{x}} ( \pi _{xt}\xi _{s ( xt ) } ) ||\ d\mu
_{G^{0}} ( x ) \leq \varepsilon $.
\end{enumerate}
\end{enumerate}
\end{proposition}

\begin{proof}
(1). In view of Lemma \ref{Lemma:fixed-point-representation}, this is a
particular instance of \cite[Proposition 1.5.2]{peterson_ergodic}.

(2). (2.a)$\Rightarrow $(2.b) Trivial. (2.b)$\Rightarrow $(2.c) Let $\sigma $
be a representation of $G$ on $\mathcal{H}^{\prime }$ such that $\pi \otimes
\overline{\sigma }$ contains invariant vectors. By Remark~\ref{rem:HS}, we
can identify $\pi \otimes \overline{\sigma }$ with a representation of $G$
on \textrm{HS}$(\mathcal{H}^{\prime },\mathcal{H})$. By assumption, there
exists $T\in L^{2}(G^{0},\mathrm{HS}(\mathcal{H}^{\prime },\mathcal{H}))$
nonzero such that $\pi _{\gamma }T_{s(\gamma )}\sigma _{\gamma }^{\ast
}=T_{r(\gamma )}$ for every $\gamma \in G$. Thus $\pi _{\gamma }T_{s(\gamma
)}T_{s(\gamma )}^{\ast }\pi _{\gamma }^{\ast }=T_{r(\gamma )}T_{r(\gamma
)}^{\ast }$ for every $\gamma \in G$. Since $T$ is nonzero, there exists $%
d>0 $ such that $T_{x}T_{x}^{\ast }$ has eigenvalue $d$ for almost every $%
x\in G^{0}$. If $\mathcal{K}_{x}\subseteq \mathcal{H}_{x}$ is the eigenspace
of $T_{x}T_{x}^{\ast }$ corresponding to $d$ for almost every $x\in G^{0}$,
then $\mathcal{K}$ defines a nonzero finite-dimensional $\pi $-invariant
sub-bundle of $\mathcal{H}$.

(2.c)$\Rightarrow $(2.a) Suppose that $\mathcal{K}$ is a nonzero
finite-dimensional $\pi $-invariant sub-bundle of $\mathcal{H}$. Then $%
x\mapsto \mathrm{Proj}_{\mathcal{K}_{x}}$ determines a $\pi \otimes
\overline{\pi }$-invariant element of $\mathrm{HS} ( \mathcal{H} ) =\mathcal{%
H}\otimes \overline{\mathcal{H}}$.

(3). (3.a)$\Rightarrow $(3.b) Suppose, by contradiction, that $\mathcal{K}$
is a finite-dimensional $\pi $-invariant nonzero sub-bundle of $\mathcal{H}$%
, of dimension $d\in \mathbb{N}$. Fix sections $\xi _{1},\ldots ,\xi
_{d}\colon G^{0}\rightarrow \mathcal{K}$ such that $\{\xi _{1,x},\ldots ,\xi
_{d,x}\}$ is an orthonormal basis of $\mathcal{K}$ for almost every $x\in
G^{0}$. Fix $\varepsilon ,\delta >0$. By assumption, there exists $t\in
\lbrack G]$ such that $|\langle \xi _{i},[\pi ]_{t}(\xi _{j})\rangle
|<\varepsilon $ for $1\leq i,j\leq d$. By choosing $\varepsilon >0$ small
enough, this guarantees that there exists a Borel subset $A$ of $G^{0}$ of
measure at least $1-\delta $ such that $|\langle \xi _{i,x},\pi _{xt}(\xi
_{j,s(tx)})\rangle |<\delta $ for all $x\in A$ and $1\leq i,j\leq n$.
Therefore,
\begin{equation*}
1=\Vert \pi _{xt}(\xi _{1,s(xt)})\Vert ^{2}=\sum_{i=1}^{d}|\langle \xi
_{i,x},\pi _{xt}(\xi _{j,xt})\rangle |^{2}<\delta ^{2}d
\end{equation*}%
for almost every $x\in A$. Choosing $\delta =d^{-1/2}$, we reach a
contradiction.

(3.b)$\Rightarrow $(3.c) Let $\mathcal{K}$ be a finite-dimensional $\pi$%
-invariant sub-bundle of $\mathcal{H}$, and suppose that the conclusion
fails. Then there exists $c>0$ such that for every $t\in [ G ] $, there
exists a unit section $\xi \colon G^0\to \mathcal{K}$ with $\langle \xi , [
\pi ] _{t}(\xi) \rangle \geq c$. Consider the section $\mathrm{Proj}_{%
\mathcal{K}}\colon G^0\to \mathcal{H}\otimes \overline{\mathcal{H}}$ given
by $\mathrm{Proj}_{\mathcal{K}}(x)=\mathrm{Proj}_{\mathcal{K}_x}$ for all $%
x\in G^0$. Then $\langle \mathrm{Proj}_{\mathcal{K}}, [ \pi \otimes
\overline{\pi } ] _{t}(\mathrm{Proj}_{\mathcal{K}}) \rangle \geq c$ for all $%
t\in [G]$. By part~(1), $\pi \otimes \overline{\pi }$ has invariant vectors.
Therefore $\pi $ has a finite-dimensional subrepresentation by part~(2),
contradicting the hypothesis.

(3.c)$\Rightarrow $(3.a) Suppose that $\mathcal{K}$ is a finite-dimensional $%
\pi$-invariant sub-bundle of $\mathcal{H}$, and let $\varepsilon >0$. Then
there exists $t\in [ G ] $ such that
\begin{equation*}
\int_{G^0} \left\|\mathrm{Proj}_{\mathcal{K}_{x}} ( \pi _{xt}(\xi _{s ( xt )
} )) \right\|\ d\mu _{G^{0}} ( x ) \leq \varepsilon
\end{equation*}
for every unit section $\xi $ for $\mathcal{K}$. Thus, if $\eta$ and $\xi $
are unit sections for $\mathcal{K}$, then%
\begin{equation*}
\langle \eta , [ \pi ] _{t}(\xi) \rangle \leq \int_{G^0}\left\|\mathrm{Proj}%
_{\mathcal{K}_{x}} ( \pi _{xt}(\xi _{s ( xt ) } )) \right\|\ d\mu _{G^{0}} (
x ) \leq \varepsilon \text{.}
\end{equation*}%
This concludes the proof that $\pi $ is weak mixing.
\end{proof}

\begin{corollary}
\label{Corollary:wm-rep} Let $\pi $ be a representation of an ergodic
discrete pmp groupoid $G$ on a Hilbert bundle $\mathcal{H}$. Then the
following assertions are equivalent:

\begin{enumerate}
\item $\pi $ is weak mixing;

\item $\pi \otimes \sigma $ is weak mixing for every representation $\sigma $%
;

\item $\pi \otimes \overline{\pi }$ is weak mixing;

\item $\pi \otimes \overline{\pi }$ is ergodic;

\item the representation $[ \pi \otimes \overline{\pi } ] $ of $[ G ] $ on $%
L^{2} ( G^{0},\mathcal{H}\otimes \mathcal{H} ) $ is ergodic.
\end{enumerate}
\end{corollary}

\begin{proof}
(1)$\Rightarrow $(2) Suppose that there exists a unitary representation $%
\sigma$ such that $\pi \otimes \sigma $ is not weak mixing. Then $\pi
\otimes \sigma $ contains a finite-dimensional subrepresentation, by
part~(3) of Proposition~\ref{Proposition:equivalences-groupoid-rep}.
Therefore $\pi \otimes \sigma \otimes \pi \otimes \sigma $ has an invariant
vector by part~(2) of Proposition~\ref{Proposition:equivalences-groupoid-rep}%
, and thus $\pi $ has a finite-dimensional subrepresentation again by
part~(2) of Proposition~\ref{Proposition:equivalences-groupoid-rep}.
Therefore $\pi $ is not weak mixing by part~(3) of Proposition~\ref%
{Proposition:equivalences-groupoid-rep}.

(2)$\Rightarrow $(3)$\Rightarrow $(4) Obvious.

(4)$\Rightarrow $(1) Suppose that $\pi $ is not weak mixing. Then $\pi $ has
a finite-dimensional sub-representation by part~(3) of Proposition~\ref%
{Proposition:equivalences-groupoid-rep}, and $\pi \otimes \overline{\pi }$
contains invariant vectors by part~(2) of Proposition~\ref%
{Proposition:equivalences-groupoid-rep}. Thus $\pi \otimes \overline{\pi }$
is not ergodic.

(4)$\Leftrightarrow $(5) This follows from Lemma \ref%
{Lemma:fixed-point-representation}.
\end{proof}

\subsection{Bundles of tracial von Neumann algebras\label%
{Subsection:vn-bundles}}

Let $\mathcal{H}$ be a Hilbert bundle over a standard Borel space $X $, and
set $B ( \mathcal{H} )=\bigsqcup_{x\in X} B(\mathcal{H}_x)$, which is a
standard Borel space over $X$. The Borel structure on $B ( \mathcal{H} ) $
is generated by the functions $(T\in B(\mathcal{H}_x)) \mapsto x$ and $(
T\in B(\mathcal{H}_x)) \mapsto \langle \sigma _{n,x},T(\sigma _{m,x}
)\rangle $ for $n,m\in \mathbb{N}$. Moreover, $B(\mathcal{H})$ is
canonically endowed with the following fibered functions: operator norm,
composition, adjoint, sum, and scalar multiplication.

\begin{definition}
A \emph{C*-bundle} over the Hilbert bundle $\mathcal{H}$ is a Borel subset $%
\mathcal{A}$ of $B ( \mathcal{H} ) $ such that, for every $x\in X$, the
corresponding fiber $\mathcal{A}_{x}=\mathcal{A}\cap B ( \mathcal{H}_{x} ) $
is a C*-subalgebra of $B ( \mathcal{H}_{x} ) $, and for which there exists a
sequence $(a_n)_{n\in\mathbb{N}}$ of sections in $\mathcal{A}$ such that $(
a_{n,x} ) _{n\in \mathbb{N}}$ is a subset of the unit ball of $\mathcal{A}%
_{x}$ that generates a norm-dense *-subalgebra of $\mathcal{A}_{x}$. We also
denote the bundle $\mathcal{A}$ by $\bigsqcup_{x\in X}\mathcal{A}_{x}$.
\end{definition}

\begin{definition}
A \emph{tracial von Neumann bundle} over a standard Borel space $X$ is a
Borel subset $\mathcal{M}$ of $B ( \mathcal{H} ) $ together with a Borel
function $\tau \colon \mathcal{M}\to \mathbb{C}$ such that, for every $x\in
X $, the corresponding fiber $\mathcal{M}_{x}=\mathcal{M}\cap B ( \mathcal{H}%
_{x} ) $ is a von Neumann algebra and the restriction $\tau _{x}$ of $\tau $
to $\mathcal{M}_{x}$ is a faithful tracial state on $\mathcal{M}_{x}$, and
for which there exists a sequence of sections $( a_{n} ) _{n\in \mathbb{N}}$
of $\mathcal{M}$ such that $( a_{n,x} ) _{n\in \mathbb{N}}$ is a subset of
the unit ball of $\mathcal{M}_{x}$ that generates a *-subalgebra of $%
\mathcal{M}_{x}$ whose operator-norm unit ball is dense in the operator-norm
unit ball of $\mathcal{M}_{x}$ with respect to the $2$-norm $\| a
\|_{2}=\tau _{x} ( a^{\ast }a ) ^{1/2}$, and there exists a sequence of
sections $( \xi _{n} ) _{n\in \mathbb{N}}$ of $\mathcal{H}$ such that $\tau
_{x} ( a ) =\sum_{n\in \mathbb{N}} \langle \xi _{n,x},a(\xi _{n,x} )\rangle $
for $x\in X$ and $a\in \mathcal{M}_{x}$. (In particular, this implies that $%
\tau _{x}$ is a normal tracial state on $\mathcal{M}_{x}$.) We also denote
the bundle $( \mathcal{M},\tau ) $ by $\bigsqcup_{x\in X} ( \mathcal{M}%
_{x},\tau _{x} ) $. We say that $( \mathcal{M},\tau ) $ is \emph{abelian} if
$\mathcal{M}_{x}$ is an abelian von Neumann algebra for almost every $x\in X$%
, .
\end{definition}

Let $\mathcal{M}$ be a tracial von Neumann bundle over a standard
probability space $(X,\mu )$. We let $L^{2}(\mathcal{M},\tau )$ be the
Hilbert bundle $\bigsqcup_{x\in X}L^{2}(\mathcal{M}_{x},\tau _{x})$ over $X$%
. Given $a\in \mathcal{M}_{x}$, we let $|a\rangle $ be the corresponding
element of $L^{2}(\mathcal{M}_{x},\tau _{x})$. We identify $\mathcal{M}_{x}$
with a subalgebra of $B(L^{2}(\mathcal{M}_{x},\tau _{x}))$, by identifying $%
a\in \mathcal{M}_{x}$ with its associated multiplication operator. We define
$L^{\infty }(X,\mathcal{M})$ to be the algebra of essentially bounded
sections of $\mathcal{M}$, which is a tracial von Neumann algebra with
respect to the normal tracial state $\tau =\int_{X}\tau _{x}d\mu _{X}(x)$.
Thus the inclusion $L^{\infty }(X)\subseteq L^{\infty }(X,\mathcal{M})$ is a
trace-preserving embedding. The GNS construction $L^{2}(L^{\infty }(X,%
\mathcal{M}),\tau )$ associated with $\tau $ can be identified with the
space $L^{2}(X,L^{2}(\mathcal{M},\tau ))$ associated with the Hilbert bundle
$L^{2}(\mathcal{M},\tau )$ as defined in Subsection \ref%
{Subsections:Hilbert-bundles}. We will denote this space simply by $L^{2}(X,%
\mathcal{M},\tau )$.

The GNS representation of $L^{\infty } ( X,\mathcal{M} ) $ on $L^{2} ( X,%
\mathcal{M},\tau ) $ maps an element $a\in L^{\infty } ( X,\mathcal{M} ) $
to the corresponding \emph{decomposable operator} on $L^{2} ( X,\mathcal{M}%
,\tau ) $ defined by $a(\xi)= ( a_{x}\xi _{x} ) _{x\in X}$ for $\xi=( \xi
_{x} ) _{x\in X}$. The canonical conditional expectation $E_{L^{\infty } ( X
) }\colon L^{\infty } ( X,\mathcal{M} ) \to L^{\infty } ( X ) $ is given by $%
E_{L^{\infty } ( X ) }(a)=( \tau _{x} ( a_{x} ) ) _{x\in X}$, for $%
a=(a_x)_{x\in X}$. This gives to $L^{\infty } ( X,\mathcal{M} )$ the
structure of pre-C*-module over $L^{\infty } ( X ) $.

A\emph{\ von Neumann sub-bundle} of a von Neumann bundle $( \mathcal{M},\tau
) $ over a standard Borel space $X$ is a Borel subset $\mathcal{N}\subseteq
\mathcal{M}$ such that $\mathcal{N}_{x}$ is a w*-closed subalgebra of $%
\mathcal{M}_x$, for all $x\in X$. For every $x\in X$, the unique
trace-preserving conditional expectation $\mathcal{M}_{x}\to \mathcal{N}_{x}$
is denoted by $E_{\mathcal{N}_{x}}$. This defines a trace-preserving
expectation $E_{\mathcal{N}}\colon L^{\infty } ( X,\mathcal{M} )\to
L^{\infty } ( X,\mathcal{N} ) $, given by $E_{\mathcal{N}}(a )=( E_{\mathcal{%
N}_{x}} ( a_{x} ) ) _{x\in X}$.

\begin{notation}[Orthogonal complements]
Let $\mathcal{N}$ be a von Neumann sub-bundle of a von Neumann bundle $%
\mathcal{M}$. We let $L^{2} ( \mathcal{M},\tau ) \cap \mathcal{N}^{\perp}$
be the sub-bundle $\bigsqcup_{x\in X}(L^{2} ( \mathcal{M}_{x},\tau _{x} )
\cap \mathcal{N}_{x}^{\bot })$. Given a subalgebra $A$ of $L^{\infty } ( X,%
\mathcal{M} ) $, we let $L^{2} ( X,\mathcal{M},\tau ) \cap A^{\bot }$ be the
orthogonal complement of $A$ inside $L^{2} ( X,\mathcal{M},\tau ) $, where $%
A $ is canonically identified with a subspace of $L^{2} ( X,\mathcal{M},\tau
) $.
\end{notation}

A particular example of a von Neumann sub-bundle is the \emph{trivial
sub-bundle }$\bigsqcup_{x\in X}\mathbb{C}1_{x}$, where $1_{x}$ is the unit
of $\mathcal{M}_{x}$. We define the \emph{center} of $\mathcal{M}$ to be the
sub-bundle $Z ( \mathcal{M} ) =\bigsqcup_{x\in X}Z ( \mathcal{M}_{x} )$.

\subsection{Actions of groupoids\label{Subsection:actions}}

\begin{definition}
Given a C*-bundle $\mathcal{A}$ over a standard Borel space $X$, we define
the \emph{automorphism groupoid} of $\mathcal{A}$ as
\begin{equation*}
\mathrm{Aut} ( \mathcal{A} ) =\{\alpha \colon \mathcal{A}_{x}\to \mathcal{A}%
_{y} \mbox{ $\ast$-isomorphism, for } x,y\in X\}.
\end{equation*}

Similarly, given a tracial von Neumann bundle $(\mathcal{M},\tau )$, its
\emph{(tracial) automorphism groupoid} is
\begin{equation*}
\mathrm{Aut}(\mathcal{M},\tau )=\{\alpha \colon \mathcal{M}_{x}\rightarrow
\mathcal{M}_{y}\mbox{ trace preserving $\ast$-isomorphism, for }x,y\in X\}.
\end{equation*}
\end{definition}

In the following, we will identify the unit space of $\mathrm{Aut}(
\mathcal{A}) $ and $\mathrm{Aut}( \mathcal{M},\tau ) $ with
$X$.

\begin{remark}
The automorphism groupoid of a C*-bundle $\mathcal{A}$ is naturally a
standard Borel groupoid, endowed with the standard Borel structure generated
by the source and range maps together with the functions $( \alpha :%
\mathcal{A}_{x}\rightarrow \mathcal{A}_{y}) \mapsto \left\Vert \alpha
(a_{n,x})\right\Vert $ for $n\in \mathbb{N}$, where $( a_{n})
_{n\in \mathbb{N}}$ are the sections of $\mathcal{A}$ as in the definition
of a C*-bundle.\ The same applies to $\mathrm{Aut}( \mathcal{M},\tau
) $ for a tracial von Neumann bundle $( \mathcal{M},\tau ) $
when one replaces the operator norm with the $2$-norm defined by $\tau $.
\end{remark}

\begin{definition}
An \emph{action} of a discrete pmp groupoid $G$ with unit space $X$ on a
tracial von Neumann bundle $(\mathcal{M},\tau )$ over $X$ is a homomorphism $%
\alpha \colon G\rightarrow \mathrm{Aut}(\mathcal{M},\tau )$ that fixes the
unit space. The action $\alpha $ induces a (group) action $[\alpha ]\colon
\lbrack G]\rightarrow \mathrm{Aut}( (L^{\infty }(X,\mathcal{M}),\tau
)) $ defined by
\begin{equation*}
\lbrack \alpha ]_{t}(a)=(\alpha _{xt}(a_{s(xt)})_{x\in G^{0}})
\end{equation*}%
for $t\in \lbrack G]$ and $a=(a_{x})_{x\in G^{0}}\in L^{\infty }(G^{0},%
\mathcal{M})$. It also induces an action of $[[G]]$, defined similarly.
\end{definition}

The proof of the following lemma is immediate.

\begin{lemma}
\label{Lemma:fixed-point-action}Let $\alpha $ be a representation of an
ergodic discrete pmp groupoid $G$ on a tracial von Neumann bundle $\mathcal{M%
}$. For $a=(a_{x})_{x\in G^{0}}$ of $L^{\infty }(G^{0},\mathcal{M}%
)$, the following assertions are equivalent:

\begin{enumerate}
\item $a$ is fixed by $[\alpha ]$;

\item there exists a countable subset $S\subseteq \lbrack G]$ that covers $G$
such that $[a]_{t}(a)=a$ for every $t\in S$;

\item $\alpha _{\gamma }(a_{s(\gamma )})=a_{r(\gamma )}$ for almost every $%
\gamma \in G$.
\end{enumerate}
\end{lemma}

%The following lemma is proved similarly to Lemma \ref{Lemma:fixed-point-representation}.

%\begin{lemma}
%\label{Lemma:fixed-point-action} let $\alpha $ is an action of a
%discrete pmp groupoid on $G$ on a tracial von Neumann bundle $\mathcal{M}$.
%For an element $a= ( a_{x} ) _{x\in
%G^{0}}$ of $L^{\infty } ( G^{0},\mathcal{M} ) $, the following
%assertions are equivalent:

%\begin{enumerate}
%\item $a$ belongs to the fixed point algebra $L^{\infty } ( G^{0},\mathcal{M} )^{[\alpha]}$ of $[\alpha]$;

%\item there exists a countable subset $S\subseteq [ G ] $ that covers $%
%G$ and such that $ [ \alpha  ] _{t}(a)=a$ for every $t\in S$;

%\item $\alpha _{\gamma } ( a_{s(\gamma )} ) =a_{r(\gamma )}$ for
%almost every $\gamma \in G$.
%\end{enumerate}
%\end{lemma}

\subsection{Actions of groupoids on spaces\label{Subsection:actions-spaces}}

Let $( G,\mu _{G} ) $ be a discrete pmp groupoid, and let $( Z,\lambda ) $
be a standard probability space fibered over $G^{0}$. Let $( \lambda _{x} )
_{x\in G^{0}} $ be the disintegration of $\lambda $ with respect to $\mu_G$.
Then $\bigsqcup_{x\in G^{0}}L^{2} ( Z_{x},\lambda _{x} ) $ is a Hilbert
bundle over $G^{0}$, and $\bigsqcup_{x\in G^{0}} ( L^{\infty } ( Z_{x} )
,\lambda _{x} ) $ is a tracial von Neumann bundle over $G^0$

\begin{definition}
\label{Definition:pmp-action} A \emph{pmp action} of a discrete pmp groupoid
$G$ on a standard probability space $( Z,\lambda ) $ fibered over $G^{0}$ is
an action of $G$ on the tracial von Neumann bundle $( \mathcal{M},\tau )
=\bigsqcup_{x\in G^{0}}(L^{\infty } ( Z_{x} ) ,\lambda _{x})$.
%The action $\alpha$ is \emph{%
%atomless} if $ ( Z_{x},\lambda _{x} ) $ is an
%atomless standard probability space for almost every $x\in G^0$.
\end{definition}

\begin{remark}
The automorphism groupoid $\mathrm{Aut} \left( \bigsqcup_{x\in G^{0}} (
L^{\infty } ( Z_{x} ) ,\tau _{x} ) \right) $ can be identified with the
groupoid $\mathrm{Aut} \left( \bigsqcup_{x\in G^{0}} ( Z_{x},\lambda _{x} )
\right)$ consisting of all Borel isomorphisms $\eta \colon Z_{s}\to Z_{t}$,
for $s,t\in G^0$, satisfying $\eta _{\ast } ( \lambda _{s} ) =\lambda _{t}$.
Thus, a pmp action of $G$ on $( Z,\lambda ) $ can be seen as a Borel
groupoid homomorphism $G\to \mathrm{Aut} ( \bigsqcup_{x\in G^{0}} (
Z_{x},\lambda _{x} ) ) $ fixing the unit space.
\end{remark}

\begin{lemma}
Let $G$ be a discrete pmp groupoid. Then pmp actions of $G $ on standard
probability spaces can be canonically identified with class-bijective pmp
extensions of $G$. In other words, there are natural assignments $%
\alpha\mapsto (\pi_{\alpha},H_{\alpha})$ and $(\pi,H)\mapsto \alpha_{\pi,H}$
between the classes of pmp actions and class-bijective extensions of $G$
satisfying $\alpha_{(\pi_{\alpha},H_{\alpha})}=\alpha$ and $%
(\pi_{\alpha_{(\pi,H)}},H_{\alpha_{(\pi,H)}})=(\pi,H)$.
\end{lemma}

\begin{proof}
Let $H$ be a pmp groupoid and let $\pi \colon H\rightarrow G$ be a
class-bijective pmp extension. Then $\pi $ turns $(H^{0},\mu _{H^{0}})$ into
a standard probability space fibered over $G^{0}$, and we let $%
(H_{x}^{0},\mu _{H^{0},x})_{x\in G^{0}}$ be the corresponding
disintegration. One can then consider the pmp action $\alpha _{\pi ,H}$ of $%
G $ on $(H^{0},\mu _{H^{0}})$ given by, for $\gamma \in G$ and $y\in
H_{s\left( \gamma \right) }$,%
\begin{equation*}
(\alpha _{\pi ,H})_{\gamma }(y)=r\left( (\pi |_{Hy})^{-1}(\gamma )\right)
\text{.}
\end{equation*}%
Conversely, let $\alpha \colon G\rightarrow \mathrm{Aut}(\bigsqcup_{x\in
G^{0}}(Z_{x},\lambda _{x}))$ be an action of $G$ on $(Z,\lambda )$. Define
the corresponding \emph{action groupoid }$G\ltimes ^{\alpha }Z$ as follows:
the set of objects of $G\ltimes ^{\alpha }Z$ is $Z$, while the arrows of $%
G\ltimes ^{\alpha }Z$ are pairs $(\gamma ,z)$ with $\gamma \in G$ and $z\in
Z_{s(\gamma )}$. We represent the pair $(\gamma ,z)$, regarded as an element
of $G\ltimes ^{\alpha }Z$, by $\gamma \ltimes ^{\alpha }z$. Then the
following operations turn $G\ltimes ^{\alpha }Z$ into a discrete pmp
groupoid:
\begin{equation*}
s(\gamma \ltimes ^{\alpha }z)=z,\ \ r(\gamma \ltimes ^{\alpha }z)=\alpha
_{\gamma }(z),\ \ (\gamma \ltimes z)^{-1}=(\gamma ^{-1}\ltimes ^{\alpha
}\alpha _{\gamma }(z)),\ \ (\rho \ltimes ^{\alpha }w)(\gamma \ltimes
^{\alpha }z)=\rho \gamma \ltimes ^{\alpha }z,
\end{equation*}%
the last one whenever $\alpha _{\gamma }(z)=w$. Moreover, the map $\pi
_{\alpha }\colon G\ltimes ^{\alpha }Z\rightarrow G$ given by $\pi _{\alpha
}(\gamma \ltimes ^{\alpha }z)=\gamma $ for all $\gamma \ltimes ^{\alpha
}z\in G\rtimes ^{\alpha }Z$ is a class-bijective pmp extension of $G$.

%Furthermore, the function $ [ G ] \to [ G\ltimes ^{\alpha }Z ] $
%mapping a section $s$ to the section $s\ltimes ^{\alpha }Z= \{ \gamma
%\ltimes ^{\alpha }z\colon \gamma \in s,z\in Z_{s(\gamma )} \} $ defines a
%continuous group homomorphism.
It is not difficult to verify that the constructions described above are
inverse of each other, and this concludes the proof.
\end{proof}

In view of the above lemma, from now on we will identify pmp actions and
class-bijective pmp extensions of a discrete pmp groupoid.

%We close this subsection with an easy observation which will be needed later.

%\begin{lemma}\label{lma:OE}
%Let $G$ be an ergodic discrete pmp groupoid, and let $\Gamma$ be a
%subgroup of $[G]$ which covers $G$. Let $(Z,\lambda)$ be a standard
%probability space, and let $\alpha$ be an action of $G$ on $Z$.
%Denote by $\theta$ the restriction of $[\alpha]$ to $\Gamma$.
%Then $\Gamma\curvearrowright Z$ and $G\curvearrowright Z$ generate the
%same orbit equivalence relation. In particular, when $\alpha$ is free,
%then both actions have isomorphic crossed products.
%\end{lemma}
%\begin{proof}
%Using that $\Gamma$ covers $G$, it is immediate to see that the
%orbit equivalence relations $R_G$ and $R_\Gamma$ of $\alpha$ and
%$\theta=[\alpha]_{\Gamma}$, respectively, agree. When $\alpha$ is free,
%then $\theta$ is also free, and the von Neumann algebras of $R_G$
%and $R_\Gamma$ are canonically isomorphic to their respective crossed
%products.
%\end{proof}

\subsection{Tensor products\label{Subsection:tensor}}

Let $(M,\tau )$ be a tracial von Neumann algebra, and let $N\subseteq M$ be
a von Neumann subalgebra. We denote by $E_{N}\colon M\rightarrow N$ the
unique trace-preserving conditional expectation, and set $\langle a,b\rangle
_{N}=E_{N}(a^{\ast }b)$ for all $a,b\in M$. This pairing turns $M$ into a
(right) pre-C*-module over $N$, whose completion is a C*-module over $N$. As
explained in \cite[Subsection 8.5.32]{blecher_operator_2004}, the weak
closure of this C*-module within its linking algebra has a canonical
structure of W*-module over $N$, which we denote by $L^{2}(M,E_{N})$. By
\cite[Lemma 8.5.4]{blecher_operator_2004}, $L^{2}(M,E_{N})$ has a unique
predual that makes the $N$-valued inner product separately w*-continuous. We
consider the w*-topology on $L^{2}(M,E_{N})$ to be the one defined by its
unique predual. As shown in \cite[Section 8.5]{blecher_operator_2004}, a
bounded $N$-bimodule map on $L^{2}(M,E_{N})$ is automatically adjointable,
and the space $B_{N}(L^{2}(M,E_{N}))$ of such maps is a von Neumann algebra.
Given $a\in M$, we let $|a\rangle _{N}\in L^{2}(M,E_{N})$ be the
corresponding operator. Assigning to an element $a\in M$ the corresponding
multiplication operator on $L^{2}(M,E_{N})$ defines a faithful
representation $M\rightarrow B_{N}(L^{2}(M,E_{N}))$.

\begin{definition}
Let $(M_{0},\tau _{0})$ and $(M_{1},\tau _{1})$ be tracial von Neumann
algebras with a common subalgebra $N$ contained in their centers. Consider
the \emph{external tensor product} of W*-modules $L^{2}(M_{0},E_{N})\otimes
L^{2}(M_{1},E_{N})$. Given $a_{0}\in M_{0}$ and $a_{1}\in M_{1}$, denote by $%
a_{0}\otimes _{N}a_{1}$ the operator on $L^{2}(M_{0},E_{N})\otimes
L^{2}(M_{1},E_{N})$ given by
\begin{equation*}
(a_{0}\otimes _{N}a_{1})(|b_{0}\rangle _{N}\otimes |b_{1}\rangle
_{N})=|a_{0}b_{0}\rangle _{N}\otimes |a_{1}b_{1}\rangle _{N}
\end{equation*}%
for $|b_{0}\rangle _{N}\otimes |b_{1}\rangle _{N}\in
L^{2}(M_{0},E_{N})\otimes L^{2}(M_{1},E_{N})$. The\emph{\ tensor product} $%
M_{0}\otimes _{N}M_{1}$ \emph{relative to} $N$ is the w*-closed subalgebra
of $B(L^{2}(M_{0},E_{N})\otimes L^{2}(M_{1},E_{N}))$ generated by the $N$%
-bimodule operators of the form $a_{0}\otimes _{N}a_{1}$ for some $a_{0}\in
M_{0}$ and $a_{1}\in M_{1}$. We define a tracial state $\tau $ on it by
setting $\tau (a)=\tau (\langle 1\otimes _{N}1|a|1\otimes _{N}1\rangle _{N})$
for $a\in M_{0}\otimes _{N}M_{1}$.
\end{definition}

It is easy to see that $a\otimes _{N}1=1\otimes _{N}a$ whenever $a\in N$.
When $N$ is the trivial subalgebra of $M_{0},M_{1}$, then the conditional
expectations $E_{N}$ coincide with the given tracial states of $M_{0}$ and $%
M_{1}$, and the tensor product relative to $N$ coincides with the tensor
product $M_{0}\otimes M_{1}$ with respect to the given traces as defined in
\cite[Section III.3]{blackadar_operator_2006}.

\begin{definition}
Adopt the notation of the previous definition. Let $\Gamma $ be a group, and
let $\alpha _{j}\colon \Gamma\to\mathrm{Aut} ( M_{j},\tau_{j} )$, for $j=0,1
$, be actions. We say that $\alpha _{0}$ and $\alpha _{1}$ \emph{agree on $N$%
} if $N$ is $\alpha _{j}$-invariant for $j=0,1$, and the restrictions of $%
\alpha _{0}$ and $\alpha _{1}$ to $N$ are equal. In this case, the actions $%
\alpha _{0}$ and $\alpha _{1}$ canonically induce an action $\alpha
_{0}\otimes _{N}\alpha _{1}$ of $\Gamma$ on $( M_{0}\otimes _{N}M_{1},\tau )
$.
\end{definition}

\begin{remark}
Since we are assuming that $N$ is a central subalgebra of $M_{0}$ and $M_{1}$%
, their tensor product $M_{0}\otimes _{N}M_{1}$ can be equivalently
described as follows. Given a presentation $N=L^{\infty } ( Z,\lambda ) $,
the algebras $M_{0}$ and $M_{1}$ admit direct integral decompositions%
\begin{equation*}
M_{0}=\int_{Z}^{\oplus }M_{0,z}d\lambda (z)\ \ \mbox{ and } \ \
M_{1}=\int_{Z}^{\oplus }M_{1,z}d\lambda (z)\text{;}
\end{equation*}%
see \cite[Chapter 14]{kadison_fundamentals_1986} or \cite[Appendix F]%
{williams_crossed_2007}. Then $M_0\otimes_N M_1$ can be identified with $%
\int_{Z}^{\oplus } ( M_{0,z}\otimes M_{1,z} ) d\lambda (z)$.
\end{remark}

\begin{definition}
Given two tracial von Neumann bundles $( \mathcal{M}_{0},\tau _{0} )$ and $(
\mathcal{M}_{1},\tau _{1} ) $ with a common central sub-bundle $\mathcal{N}$%
, their tensor product $\mathcal{M}_{0}\otimes _{\mathcal{N}}\mathcal{M}_{1}$
is the tracial von Neumann bundle $\bigsqcup_{x\in X}(\mathcal{M}%
_{0,x}\otimes _{\mathcal{N}_{x}}\mathcal{M}_{1,x})$. Similarly as before,
one can define the tensor product of actions on $\mathcal{M}_0$ and $%
\mathcal{M}_1$ that agree on $\mathcal{N}$.
\end{definition}

In the context above, $L^{\infty } ( X,\mathcal{M}_{0}\otimes _{\mathcal{N}}%
\mathcal{M}_{1} ) $ can be identified with $L^{\infty } ( X,\mathcal{M}_{0}
) \otimes _{L^{\infty } ( X,\mathcal{N} ) }L^{\infty } ( X,\mathcal{M}_{1} )
$; see \cite[Section 1.3]{oty_actions_2006} and \cite{sauvageot_produit_1983}

\subsection{The Koopman representation\label{Subsection:Koopman}}

\begin{definition}
Let $G$ be a discrete pmp groupoid, and let $\alpha \colon G\to \mathrm{Aut}
( \mathcal{M},\tau ) $ be an action on a von Neumann bundle $(\mathcal{M}%
,\tau)$ over $G^0$. The \emph{Koopman representation }associated with $%
\alpha $ is the representation $\kappa ^{\alpha }\colon G\to U(L^{2} (
\mathcal{M},\tau ))$ defined by $\kappa _{\gamma }^{\alpha } \vert a \rangle
= \vert \alpha _{\gamma } ( a ) \rangle $ for $\gamma \in G$ and $a\in
\mathcal{M}_{s(\gamma )}$.

We denote by $\kappa _{0}^{\alpha }$ the restriction of $\kappa ^{\alpha }$
to the invariant sub-bundle $\bigsqcup_{x\in G^{0}} ( L^{2} ( \mathcal{M}%
_{x},\tau _{x} ) \cap ( \mathbb{C} \vert 1_{x} \rangle ) ^{\bot } ) $. More
generally, if $( \mathcal{N},\tau ) $ is an $\alpha $-invariant central von
Neumann sub-bundle of $( \mathcal{M},\tau ) $, we define a representation $%
\kappa ^{\alpha ,\mathcal{N}}\colon G \to U(L^{2} ( \mathcal{M},E_{\mathcal{N%
}} )) $ by setting $\kappa _{\gamma }^{\alpha ,\mathcal{N}}( \vert a \rangle
_{\mathcal{N}_{s(\gamma )}})= \vert \alpha _{\gamma } ( a ) \rangle _{%
\mathcal{N}_{r(\gamma )}}$ for all $\gamma\in G$ and all $\vert a \rangle _{%
\mathcal{N}_{s(\gamma )}}\in L^2(\mathcal{M}_{s(\gamma)},E_{\mathcal{N}%
_{s(\gamma)}})$. We define then $\kappa _{0}^{\alpha ,\mathcal{N}}$ to be
the restriction of $\kappa ^{\alpha ,\mathcal{N}}$ to the sub-bundle $%
\bigsqcup_{x\in G^{0}} ( L^{2} ( \mathcal{M}_{x},E_{\mathcal{N}_{x}} ) \cap
\mathcal{N}_{x}^{\bot } ) $.
\end{definition}

\begin{remark}
\label{Remark:koopman} Adopt the notation from the above definition. Then
the canonical identification of $L^{2} ( L^{\infty } ( G^{0},\mathcal{M} )
,\tau ) $ with $L^{2} ( G^{0},L^{2} ( \mathcal{M},\tau ) ) $ allows one to
identify the Koopman representation $\kappa ^{ [ \alpha ] }$ associated with
the action $[ \alpha ] $ of $[ G ] $ on $( L^{\infty } ( G^{0},\mathcal{M} )
,\tau ) $ with the representation $[ \kappa ^{\alpha } ] $ of $[ G ] $ on $%
L^{2} ( G^{0},L^{2} ( \mathcal{M},\tau ) ) $ induced by $\kappa ^{\alpha }$.
More generally, the canonical identification of $L^{2} ( G^{0},L^{2} (
\mathcal{M},E_{\mathcal{N}} ) ) $ with $L^{2}(L^{\infty } ( G^{0},\mathcal{M}
) ,E_{L^{\infty } ( G^{0},\mathcal{N} ) })$ allows one to identify $[\kappa
^{\alpha ,\mathcal{N}}]$ with $\kappa ^{ [ \alpha ] ,L^{\infty } ( G^{0},%
\mathcal{N} ) }$.
\end{remark}

\begin{lemma}
\label{Lemma:koopman} Let $G$ be a discrete pmp groupoid. For $j=1,2$, let $(%
\mathcal{M}_j,\tau_j)$ be von Neumann bundles over $G^0$ with common central
sub-bundle $\mathcal{N}$, and let $\alpha_j \colon G\to \mathrm{Aut} (
\mathcal{M}_j,\tau ) $ be actions agreeing on $\mathcal{N}$.

\begin{enumerate}
\item For $j=1,2$, the representation $\kappa _{0}^{\alpha_j ,\mathcal{N}}$
is conjugate to $\overline{\kappa }_{0}^{\alpha_j ,\mathcal{N}}$.

\item The representation $\kappa _{0}^{\alpha_1 \otimes _{\mathcal{N}%
}\alpha_2 ,\mathcal{N}}$ is conjugate to $(\kappa _{0}^{\alpha_1 ,\mathcal{N}%
}\otimes \kappa _{0}^{\alpha_2 ,\mathcal{N}})\oplus \kappa _{0}^{\alpha_1 ,%
\mathcal{N}}\oplus \kappa _{0}^{\alpha_2 ,\mathcal{N}}$.
\end{enumerate}
\end{lemma}

\begin{proof}
(1). We fix $j\in \{1,2\}$ and write $\alpha$ for $\alpha_j$. Then the
canonical anti-unitary $J\colon L^{2} ( \mathcal{M},E_{\mathcal{N}} )
\mapsto \overline{L^{2} ( \mathcal{M},E_{\mathcal{N}} ) }$, $\vert a \rangle
\mapsto \vert a^{\ast } \rangle $ induces an isomorphism from $L^{2} (
\mathcal{M},E_{\mathcal{N}} ) \cap \mathcal{N}^{\bot }$ to $\overline{L^{2}
( \mathcal{M},E_{\mathcal{N}} ) \cap \mathcal{N}^{\bot }}$ that intertwines $%
\kappa _{0}^{\alpha ,\mathcal{N}}$ and $\overline{\kappa }_{0}^{\alpha ,%
\mathcal{N}}$.

(2). The map $\mathcal{M}_{1}\otimes _{\mathcal{N}}\mathcal{M}_{2}\to (
\mathcal{M}_{1}\otimes _{\mathcal{N}}\mathcal{M}_{2} ) \oplus \mathcal{M}%
_{1}\oplus \mathcal{M}_{2}\oplus \mathcal{N}$ defined by
\begin{equation*}
a\otimes _{\mathcal{N}}b\mapsto \left(( a-E_{\mathcal{N}} ( a ) ) \otimes _{%
\mathcal{N}} ( b-E_{\mathcal{N}} ( b ) ) , E_{\mathcal{N}} ( b ) a, E_{%
\mathcal{N}} ( a ) b, E_{\mathcal{N}} ( a ) E_{\mathcal{N}} ( b )\right)
\end{equation*}%
induces an isomorphism from $L^{2} ( \mathcal{M}_{1}\otimes _{\mathcal{N}}%
\mathcal{M}_{2},E_{\mathcal{N}} ) \cap \mathcal{N}^{\bot }$ to
\begin{equation*}
( L^{2} ( \mathcal{M}_{1},E_{\mathcal{N}} ) \cap \mathcal{N}^{\bot } )
\otimes ( L^{2} ( \mathcal{M}_{2},E_{\mathcal{N}} ) \cap \mathcal{N}^{\bot }
) \oplus L^{2} ( \mathcal{M}_{1},E_{\mathcal{N}} ) \oplus L^{2} ( \mathcal{M}%
_{2},E_{\mathcal{N}} )
\end{equation*}%
that witnesses the desired conjugacy.
\end{proof}

\subsection{Ergodicity\label{Subsection:ergodicity}}

Recall that a unitary representation $\pi $ of a Polish group $\Gamma $ on a
Hilbert space $H$ is said to be \emph{ergodic }if the space $H^{\pi }$ of $%
\pi $-invariant vectors is trivial.

\begin{lemma}
\label{Lemma:fixed-point-koopman} Let $\alpha $ be an action of a countable
discrete group $\Gamma $ on a tracial von Neumann algebra $( M,\tau ) $, and
let $N$ be an $\alpha $-invariant central subalgebra of $( M,\tau )$. Then
the fixed point algebra $M^{\alpha }$ is contained in $N$ if and only $%
\kappa _{0}^{\alpha ,N}$ is ergodic.
\end{lemma}

\begin{proof}
Assume that $M^{\alpha }\subseteq N$, and consider the direct integral
decomposition%
\begin{equation*}
(M,\tau )=\int_{X}(M_{x},\tau _{x})d\mu (x)
\end{equation*}%
with respect to the central subalgebra $N=L^{\infty }(X,\mu )\subseteq M$;
see \cite[Appendix F]{williams_crossed_2007}. With $\mathcal{H}$ denoting
the Hilbert bundle $\bigsqcup_{x\in X}L^{2}(M_{x},\tau _{x})$, one can
identify $L^{2}(M,\tau )$ with $L^{2}(X,\mathcal{H},\tau )$. Let $\xi \in
L^{2}(X,\mathcal{H},\tau )$ be a $\kappa ^{\alpha }$-invariant unit vector.
Identifying $L^{2}(M_{x},\tau _{x})$ with the space of affiliated operators
with $M_{x}$---see \cite[Section 1.2]{popa_cocycle_2007}---let $\xi
_{x}=v_{x}|\xi _{x}|$ be the polar decomposition of $\xi _{x}$, where $%
v_{x}\in M_{x}$ is a partial isometry, and the spectral resolution $%
(e_{s,x})_{s>0}$ of $|\xi _{x}|$ is contained in $M_{x}$; see \cite[%
Subsection 2.1]{popa_cocycle_2007}. Suppose, by contradiction, that $\xi $
is orthogonal to $N$. Find $s>0$, a non-null Borel subset $A\subseteq X$,
and projections $e_{s,x}$, for $x\in X$, satisfying $0<\tau _{x}(e_{s,x})<1$%
. By the essential uniqueness of the spectral resolution, the element $%
e_{s}=(e_{x,s})_{x\in X}\in M$ is $\alpha $-invariant. Since $e\in N$ by
assumption, we have $\tau _{x}(e_{s,x})\in \{0,1\}$ for almost every $x\in X$%
, which is a contradiction. The converse is obvious.
\end{proof}

Recall that a discrete pmp groupoid $G$ is ergodic if a non-null invariant
Borel subset of $G^{0}$ has full measure.

\begin{definition}
\label{Definition:ergodic-rep} Let $G$ be an ergodic discrete pmp groupoid,
and let $(\mathcal{M},\tau )$ be a tracial von Neumann bundle over $G^{0}$.
An action $\alpha \colon G\rightarrow \mathrm{Aut}(\mathcal{M},\tau )$ is
said to be \emph{ergodic }if whenever $a\in L^{\infty }(G^{0},\mathcal{M})$
satisfies $\alpha _{\gamma }(a_{s(\gamma )})=a_{r(\gamma )}$ for almost
every $\gamma \in G$, then $a_{x}\in \mathbb{C}1_{x}$ for almost every $x\in
G^{0}$.
\end{definition}

\begin{lemma}
\label{Lemma:ergodic} Adopt the notation from the above definition, and let $%
( \mathcal{N},\tau ) $ be an $\alpha $-invariant central sub-bundle of $(
\mathcal{M},\tau ) $. Then the following assertions are equivalent:

\begin{enumerate}
\item $L^{\infty } ( G^{0},\mathcal{M} ) ^{\alpha }\subseteq L^{\infty } (
G^{0},\mathcal{N} ) $;

\item $L^{\infty } ( G^{0},\mathcal{M} ) ^{ [ \alpha ] |_{\Gamma }}\subseteq
L^{\infty } ( G^{0},\mathcal{N} ) $ for every countable subgroup $\Gamma
\subseteq [ G ] $ that covers $G$;

\item $L^{\infty } ( G^{0},\mathcal{M} ) ^{ [ \alpha ] |_{\Gamma }}\subseteq
L^{\infty } ( G^{0},\mathcal{N} ) $ for some countable subgroup $\Gamma
\subseteq [ G ] $ that covers $G$;

\item $\kappa _{0}^{\alpha ,\mathcal{N}}$ is ergodic;

\item $[\kappa _{0}^{\alpha ,\mathcal{N}}]|_{\Gamma }$ is ergodic for every
countable subgroup $\Gamma \subseteq [ G ] $ that covers $G$;

\item $[\kappa _{0}^{\alpha ,\mathcal{N}}]|_{\Gamma }$ is ergodic for some
countable subgroup $\Gamma \subseteq [ G ] $ that covers $G$.
\end{enumerate}
\end{lemma}

\begin{proof}
The equivalences (1)$\Leftrightarrow $(2)$\Leftrightarrow $(3) follow from
Lemma \ref{Lemma:fixed-point-action}. The equivalences (4)$\Leftrightarrow $%
(5)$\Leftrightarrow $(6) follow from Lemma \ref%
{Lemma:fixed-point-representation}. Finally, the equivalence (2)$%
\Leftrightarrow $(5) follows from Lemma \ref{Lemma:fixed-point-koopman},
since for a countable subgroup $\Gamma \subseteq \lbrack G]$ one can
identify $[\kappa ^{\alpha ,\mathcal{N}}]|_{\Gamma }$ with $\kappa ^{\lbrack
\alpha ]|_{\Gamma },L^{\infty }(G^{0},\mathcal{N})}$; see Remark \ref%
{Remark:koopman}.
\end{proof}

\begin{corollary}
\label{Corollary:ergodic} Let $\alpha $ be an action of an ergodic discrete
pmp groupoid $G$ on a tracial von Neumann bundle $( \mathcal{M},\tau ) $
over $G^0$. Then the following assertions are equivalent:

\begin{enumerate}
\item $\alpha $ is ergodic;

\item $[\alpha ]$ is ergodic;

\item for every countable subgroup $\Gamma $ of $[ G ] $ that covers $G$,
the action $[ \alpha ] |_{\Gamma }$ is ergodic;

\item for some countable subgroup $\Gamma $ of $[ G ] $ that covers $G$, the
action $[ \alpha ] |_{\Gamma }$ is ergodic.

\item the representation $[ \kappa _{0}^{\alpha } ] $ is ergodic;

\item for every countable subgroup $\Gamma $ of $[ G ] $ that covers $G$,
the representation $[ \kappa _{0}^{\alpha } ] |_{\Gamma }$ is ergodic;

\item for some countable subgroup $\Gamma $ of $[ G ] $ that covers $G$, the
representation $[ \kappa _{0}^{\alpha } ] |_{\Gamma }$ is ergodic.
\end{enumerate}
\end{corollary}

\begin{lemma}
\label{Lemma:restruction-ergodic} Let $G$ be an ergodic discrete pmp
groupoid, let $A\subseteq G^{0}$ be a non-null Borel subset, and let $\pi
\colon G\to U(\mathcal{H})$ be a representation. Then $\pi $ is ergodic if
and only if $\pi |_{AGA}$ is ergodic.
\end{lemma}

\begin{proof}
Since the \textquotedblleft if" implication is obvious, we show the
\textquotedblleft only if" direction. We prove the contrapositive. Suppose
that $\pi |_{AGA}$ is not ergodic. Then there exists an $AGA$-invariant unit
section $\xi $ for $\mathcal{H}|_{A}=\bigsqcup_{x\in A}\mathcal{H}_{x}$.
Choose $\sigma _{1},\ldots ,\sigma _{n}\in \lbrack \lbrack G]]$ such that $%
\sigma _{i}^{-1}\sigma _{i}=A=\sigma _{0}$ for $i=0,1,\ldots ,n$, and $%
(\sigma _{i}\sigma _{i}^{-1})_{i=1}^{n}$ is a partition of $G^{0}$. Define $%
\eta \in \mathcal{H}$ by setting $\eta _{r(x\sigma _{i})}=\pi _{x\sigma
_{i}}(\xi _{s(x\sigma _{i})})$ for $i=0,1,\ldots ,n$ and $x\in \sigma
_{i}\sigma _{i}^{-1}$.

We claim that $\eta $ is a $G$-invariant (unit) section. Indeed, for $\gamma
\in G$ with $s(\gamma )\in \sigma _{i}\sigma _{i}^{-1}$ and $r(\gamma )\in
\sigma _{j}\sigma _{j}^{-1}$, we use $AGA$-invariance of $\xi $ to get
\begin{equation*}
\pi _{\sigma _{j}^{-1}r(\gamma )}\pi _{\gamma }\pi _{s(\gamma )\sigma
_{i}}(\xi _{s(s(\gamma )\sigma _{i})})=\pi _{\sigma _{j}\gamma \sigma
_{i}}(\xi _{s(\sigma _{j}^{-1}\gamma \sigma _{i})})=\xi _{r(\sigma
_{j}^{-1}\gamma \sigma _{i})}=\xi _{s(r(\gamma )\sigma _{j})}\text{.}
\end{equation*}%
Therefore%
\begin{equation*}
\pi _{\gamma }(\eta _{s(\gamma )})=\pi _{\gamma }\pi _{s(\gamma )\sigma
_{i}}(\xi _{s(s(\gamma )\sigma _{i})})=\pi _{r(\gamma )\sigma _{j}}(\xi
_{s(r(\gamma )\sigma _{j})})=\eta _{r(\gamma )}\text{.}
\end{equation*}%
This shows that $\eta $ is a $G$-invariant unit section, and hence $\pi $ is
not ergodic.
\end{proof}

\subsection{Weak mixing actions\label{Subsection:weak-mixing}}

Recall that a representation $\pi \colon \Gamma \rightarrow U(H)$ of a
Polish group $\Gamma $ is \emph{weak mixing }if for every $\varepsilon >0$,
every $n\in \mathbb{N}$, and every $\xi _{1},\ldots ,\xi _{n}\in H$, there
exists $\gamma \in \Gamma $ such that $|\langle \xi _{i},\pi _{\gamma }(\xi
_{j})\rangle |<\varepsilon $ for $i,j=1,2,\ldots ,n$. Also recall that an
action $\alpha \colon \Gamma \rightarrow \mathrm{Aut}(M,\tau )$ on a tracial
von Neumann algebra is said to be \emph{weak mixing} if for any $\varepsilon
>0$, every $n\in \mathbb{N}$, and every $a_{1},\ldots ,a_{n}\in M$, there
exists $\gamma \in \Gamma $ such that $|\tau (a_{i}\alpha _{\gamma
}(a_{j}))-\tau (a_{i})\tau (a_{j})|<\varepsilon $ for $i,j=1,\ldots ,n$.
More generally, given an $\alpha $-invariant central subalgebra $N$ on $M$
with trace-preserving conditional expectation $E_{N}\colon M\rightarrow N$,
the action $\alpha $ is said to be \emph{weak mixing relatively to }$N$ if
for every $\varepsilon >0$, every $n\in \mathbb{N}$, and every $a_{1},\ldots
,a_{n}\in M$, there exists $\gamma \in \Gamma $ such that $\Vert
E_{N}(a_{i}\alpha _{\gamma }(a_{j}))-E_{N}(a_{i})E_{N}(a_{j})\Vert
_{2}<\varepsilon $ for $i,j=1,\ldots ,n$; see \cite[Lemma 2.11]%
{popa_cocycle_2007}.

\begin{definition}
\label{Definition:wm-rep} Let $G$ be an ergodic discrete pmp groupoid, let $%
\mathcal{H}$ be a Hilbert bundle over $G^{0}$, and let $( \mathcal{M},\tau )
$ be a tracial von Neumann bundle over $\mathcal{H}$. Let $\alpha \colon G\to%
\mathrm{Aut} ( \mathcal{M},\tau ) $ be an action, and let $( \mathcal{N}%
,\tau ) $ be an $\alpha $-invariant central sub-bundle. We say that

\begin{itemize}
\item $\alpha $ is \emph{weak mixing relative to $\mathcal{N} $} if the
action $[ \alpha ] \colon [ G ] \to \mathrm{Aut}(L^{\infty } ( G^{0},%
\mathcal{M} )) $ is weak mixing relative to $L^{\infty } ( G^{0},\mathcal{N}
) $;

\item the groupoid $G$ is \emph{weak mixing} if the canonical action of $[G]$
on $G^{0}$ is weak mixing.
\end{itemize}
\end{definition}

\begin{lemma}
\label{Lemma:wm} Let $G$ be an ergodic discrete pmp groupoid, let $\mathcal{H%
}$ is a Hilbert bundle over $G^{0}$, and let $(\mathcal{M},\tau )$ be a
tracial von Neumann bundle over $\mathcal{H}$. Let $\alpha \colon
G\rightarrow \mathrm{Aut}(\mathcal{M},\tau )$ be an action, and let $(%
\mathcal{N},\tau )$ be an $\alpha $-invariant central sub-bundle. Then the
following assertions are equivalent:

\begin{enumerate}
\item $\kappa _{0}^{\alpha ,\mathcal{N}}$ is weak mixing;

\item $L^{\infty }(G^{0},\mathcal{M}\otimes _{\mathcal{N}}\mathcal{M}%
)^{\alpha \otimes _{\mathcal{N}}\alpha }$ is contained in $L^{\infty }(G^{0},%
\mathcal{N})$;

\item for some countable subgroup $\Gamma \subseteq \lbrack G]$ that covers $%
G$, the action $[\alpha ]|_{\Gamma }$ is weak mixing relatively to $%
L^{\infty }(G^{0},\mathcal{N})$;

\item for every countable subgroup $\Gamma \subseteq \lbrack G]$ that covers
$G$, the action $[\alpha ]|_{\Gamma }$ is weak mixing relative to $L^{\infty
}(G^{0},\mathcal{N})$;

\item $\alpha $ is weak mixing relative to $(\mathcal{N},\tau )$;

\item for every action $\beta $ of $G$ on a tracial von Neumann bundle $(%
\mathcal{M}^{\prime },\tau ^{\prime })$ containing $(\mathcal{N},\tau )$ as
a $\beta $-invariant central sub-bundle, $L^{\infty }(G^{0},\mathcal{M}%
\otimes _{\mathcal{N}}\mathcal{M}^{\prime })^{\alpha \otimes \beta }$ is
contained in $L^{\infty }(G^{0},\mathcal{N}\otimes _{\mathcal{N}}\mathcal{M}%
^{\prime })^{\alpha \otimes \beta }$.
\end{enumerate}
\end{lemma}

\begin{proof}
(1)$\Rightarrow $(2) Suppose that $\kappa _{0}^{\alpha ,\mathcal{N}}$ is
weak mixing. Then $\kappa _{0}^{\alpha \otimes _{\mathcal{N}}\alpha ,%
\mathcal{N}}$ is isomorphic to $(\kappa _{0}^{\alpha ,\mathcal{N}}\otimes
\kappa _{0}^{\alpha ,\mathcal{N}})\oplus \kappa _{0}^{\alpha ,\mathcal{N}%
}\oplus \kappa _{0}^{\alpha ,\mathcal{N}}$ by part~(2) of Lemma~\ref%
{Lemma:koopman}. Furthermore, $\kappa _{0}^{\alpha ,\mathcal{N}}\otimes
\kappa _{0}^{\alpha ,\mathcal{N}}$ is ergodic by Corollary \ref%
{Corollary:wm-rep}, so the conclusion then follows from Lemma \ref%
{Lemma:ergodic}.

(2)$\Rightarrow $(1) By Lemma \ref{Lemma:ergodic}, $\kappa _{0}^{\alpha
\otimes _{\mathcal{N}}\alpha ,\mathcal{N}}$ is ergodic. Since $\kappa
_{0}^{\alpha \otimes _{\mathcal{N}}\alpha ,\mathcal{N}}$ is conjugate to $%
(\kappa _{0}^{\alpha ,\mathcal{N}}\otimes \overline{\kappa }_{0}^{\alpha ,%
\mathcal{N}})\oplus \kappa _{0}^{\alpha ,\mathcal{N}}\oplus \kappa
_{0}^{\alpha ,\mathcal{N}}$ by part~(2) of Lemma~\ref{Lemma:koopman}, we
conclude that $\kappa _{0}^{\alpha ,\mathcal{N}}\otimes \overline{\kappa }%
_{0}^{\alpha ,\mathcal{N}}$ is ergodic. Hence $\kappa _{0}^{\alpha ,\mathcal{%
N}}$ is weak mixing by Corollary \ref{Corollary:wm-rep}.

(1)$\Leftrightarrow $(3)$\Leftrightarrow $(4): This follows from the
equivalence (1)$\Leftrightarrow $(2) together with the equivalence of items
(1),(2),(3) in Lemma \ref{Lemma:ergodic}.

(5)$\Rightarrow $(6) This is the same as (i)$\Rightarrow $(ii) in \cite[%
Proposition 2.4.2]{popa_some_2006}.

The implications (4)$\Rightarrow $(5) and (6)$\Rightarrow $(2) are obvious.
\end{proof}

\begin{corollary}
\label{Corollary:wm} Let $\alpha $ be an action of an ergodic discrete pmp
groupoid $G$ on a tracial von Neumann bundle $( \mathcal{M},\tau ) $ over $%
G^0$. The following assertions are equivalent:

\begin{enumerate}
\item $\kappa _{0}^{\alpha }$ is weak mixing;

\item $\alpha \otimes \alpha $ is ergodic;

\item For every countable subgroup $\Gamma $ of $[G]$ that covers $G$, the
action $[\alpha ]|_{\Gamma }$ is weak mixing relatively to $L^{\infty
}\left( G^{0}\right) $;

\item For some subgroup $\Gamma $ of $[G]$ that covers $G$, the action $%
[\alpha ]|_{\Gamma }$ is weak mixing relatively to $L^{\infty }\left(
G^{0}\right) $;

\item $\alpha $ is weak mixing relative to the trivial sub-bundle;

\item $L^{\infty }(G^{0},\mathcal{M}\otimes \mathcal{N)}^{\alpha \otimes
\beta }=1\otimes L^{\infty }(G^{0},\mathcal{N})^{\beta }$ for every action $%
\beta $ of $G$ on a tracial von Neumann bundle $\mathcal{N}$;

\item $\alpha \otimes \beta $ is ergodic for every ergodic action $\beta $
of $G$ on a tracial von Neumann bundle.
\end{enumerate}
\end{corollary}

\begin{proof}
This immediately follows from Lemma \ref{Lemma:wm}, after observing that,
for any action $\beta $ of $G$, $\alpha \otimes \beta $ is equal, by
definition, to $\alpha \otimes _{\mathcal{N}}\beta $ where $\mathcal{N}$ is
the trivial sub-bundle of $( \mathcal{M},\tau ) $.
\end{proof}

\begin{corollary}
\label{Corollary:wm-infinite-tensor} Let $(\alpha ^{(n)})_{n\in \mathbb{N}}$
be a sequence of actions of an ergodic discrete pmp groupoid on tracial von
Neumann bundles $(\mathcal{M}^{(n)},\tau ^{(n)})_{n\in \mathbb{N}}$, each of
which is weak mixing relative to the trivial sub-bundle. Then the action $%
\bigotimes_{n}\alpha ^{(n)}$ on $(\bigotimes_{n}\mathcal{M}%
^{(n)},\bigotimes_{n}\tau ^{(n)})$ is weak mixing relative to the trivial
sub-bundle.
\end{corollary}

\begin{lemma}
\label{Lemma:bootstrap-weak-mixing}Suppose that $\alpha $ is an action of a
groupoid $G$ on a tracial von Neumann bundle $( \mathcal{M},\tau
) $.\ Let $\Gamma $ be a subgroup of $[ G] $ that covers $G$
such that the canonical action of $\Gamma $ on $G^{0}$ is weak mixing. If $%
\alpha $ is weak mixing relatively to the trivial sub-bundle, then the
action $[ \alpha ] |_{\Gamma }$ of $\Gamma $ on $L^{\infty
}( G^{0},\mathcal{M},\tau ) $ is weak mixing.
\end{lemma}

\begin{proof}
Fix $\varepsilon >0$ and $a_{1},\ldots ,a_{n}\in L^{\infty }( G^{0},%
\mathcal{M},\tau ) $. Use weak mixing of the action of $%
\Gamma $ on $G^{0}$ to find $\gamma _{0}\in \Gamma $
such that, for $1\leq i,j\leq n$,%
\begin{equation*}
\left\vert \tau (\mathrm{E}_{L^{\infty }(G^{0})}( a_{i}) [
\alpha ] _{\gamma _{0}}\mathrm{E}_{L^{\infty }(G^{0})}(a_{j}))-\tau
( a_{i}) \tau ( a_{j}) \right\vert <\varepsilon \text{.%
}
\end{equation*}

Since $\alpha $ is weak mixing relatively to the trivial
sub-bundle, there exists $\gamma _{1}\in \Gamma $ such that%
\begin{equation*}
\left\Vert \mathrm{E}_{L^{\infty }(G^{0})}(a_{i}[ \alpha ]
_{\gamma _{1}\gamma _{0}}(a_{j}))-\mathrm{E}_{L^{\infty }(G^{0})}(
a_{i}) \mathrm{E}_{L^{\infty }(G^{0})}([ \alpha ] _{\gamma
_{0}}a_{j})\right\Vert _{2}<\varepsilon \text{,}
\end{equation*}%
for $1\leq i,j\leq n$, and hence%
\begin{equation*}
\left\vert \tau (a_{i}[ \alpha ] _{\gamma _{1}\gamma
_{0}}(a_{j}))-\tau (\mathrm{E}_{L^{\infty }(G^{0})}( a_{i})
\mathrm{E}_{L^{\infty }(G^{0})}([ \alpha ] _{\gamma
_{0}}a_{j}))\right\vert <\varepsilon \text{.}
\end{equation*}%
Since $\mathrm{E}_{L^{\infty }(G^{0})}([ \alpha ] _{\gamma
_{0}}a_{j})=[ \alpha ] _{\gamma _{0}}\mathrm{E}_{L^{\infty
}(G^{0})}(a_{j})$ for $1\leq j\leq n$, we conclude that
\begin{equation*}
\left\vert \tau (a_{i}[ \alpha ] _{\gamma _{1}\gamma
_{0}}(a_{j}))-\tau ( a_{i}) \tau ( a_{j}) \right\vert
<2\varepsilon \text{,}
\end{equation*}%
for $1\leq i,j\leq
n $. Since $\gamma _{0}\gamma _{1}\in \Gamma $, this concludes the proof that the
action $[ \alpha ] |_{\Gamma }$ is weak mixing.
\end{proof}

\subsection{Finite index subgroupoids\label{Subsection:finite-index}}

\begin{definition}
\label{df:quotientGpd} Let $G$ be an ergodic discrete pmp groupoid and let $%
H\subseteq G$ be a subgroupoid. We define a countable equivalence relation $%
\sim _{H}$ on $G$ by $\gamma \sim _{H}\gamma ^{\prime }$ if and only $\gamma
=h\gamma ^{\prime }$ for some $h\in H$.

By ergodicity of $G$, the number of $\sim _{H}$-classes of $H$ contained in $%
Gx$ is constant for almost every $x\in G^{0}$. We define the \emph{index} $[
G:H ] \in [ 1,\infty ] $ of $H$ in $G$ to be this number.
\end{definition}

\begin{remark}
In the context of the above definition, suppose additionally that $H$ is
ergodic as well. Then one can find elements $(\sigma _{n})_{n=1}$ in $[G]$
such that $(H\sigma _{n})_{n=1}$ is a partition of $G$; see \cite[Section 2]%
{ioana_subequivalence_2009} and \cite[Lemma 1.1]{ioana_relative_2010}. We
call this a \emph{coset selection }for $H$ in $G$.
\end{remark}

\begin{lemma}
Let $G$ be an ergodic discrete pmp groupoid and let $\alpha \colon G\to%
\mathrm{Aut}( \mathcal{M},\tau ) $ be an action with an $\alpha $-invariant
sub-bundle $( \mathcal{N},\tau)$. Then $\alpha $ is weak mixing relative to $%
\mathcal{N}$ if and only if the restriction of $\alpha$ to any finite index
subgroupoid is weak mixing relative to $\mathcal{N}$.
\end{lemma}

\begin{proof}
Since the ``if" implication is obvious, we prove the converse. Let $H$ be a
subgroupoid of $G$ with index $n<\infty$, and fix a coset selection $(
\sigma _{1},\ldots ,\sigma _{n} ) $ for $H$ in $G$. Let $\Lambda $ be a
countable subgroup of $[ H ] $ that covers $H$. Then $[ \alpha ] |_{\Lambda
} $ is weak mixing. Furthermore, the subgroup $\Gamma \subseteq [ G ] $
generated by $\Lambda $ and $\sigma _{1},\ldots ,\sigma _{n}$ contains $%
\Lambda $ as a finite index subgroup. Therefore $[ \alpha ] |_{\Gamma }$ is
weak mixing. Since $\Gamma $ covers $G$, it follows that $\alpha $ is weak
mixing by \cite[Corollary 2.2.12]{peterson_ergodic}.
\end{proof}

\subsection{The groupoid von Neumann algebra\label{Subsecion:groupoid-vN}}

\begin{definition}
Let $G$ be a discrete pmp groupoid. The \emph{left regular representation}
of $G$ is the representation $\lambda\colon G\to U(\bigsqcup_{x\in
G^{0}}\ell ^{2} ( xG )) $ defined by $\lambda _{\gamma }(\delta _{\rho
})=\delta _{\gamma \rho }$ for $( \gamma ,\rho ) \in G^{2}$. Similarly, the
\emph{right regular representation} of $G$ is the representation $\rho\colon
G\to U( \bigsqcup_{x\in G^{0}}\ell ^{2} ( Gx )) $ defined by $\rho _{\gamma
}\delta _{\rho }=\delta _{\rho \gamma ^{-1}}$ for $\gamma \in G$ and $\rho
\in Gs(\gamma )$.
\end{definition}

Fix an action $\alpha $ of $G$ on a tracial von Neumann bundle $(\mathcal{M}%
,\tau )$, and consider the Hilbert bundle $\mathcal{H}=\bigsqcup_{x\in
G^{0}}(\ell ^{2}(xG)\otimes L^{2}(\mathcal{M}_{x}))$. Define a
representation $\pi ^{\alpha }=\lambda \otimes \kappa ^{\alpha }$ of $G$ on $%
\mathcal{H}$ by $\pi _{\gamma }^{\alpha }(\delta _{\rho }\otimes |a\rangle
)=\delta _{\gamma \rho }\otimes |\alpha _{\gamma }(a)\rangle $ for $(\gamma
,\rho )\in G^{2}$ and $a\in \mathcal{M}_{s(\gamma )}$. There is also a
canonical normal *-representation of $L^{\infty }(G^{0},\mathcal{M})$ on $%
L^{2}(G^{0},\mathcal{H},\tau )$ defined by $a(\xi )=(a_{x}\xi _{x})_{x\in
G^{0}}$ for $a=(a_{x})_{x\in G^{0}}\in L^{\infty }(G^{0},\mathcal{M})$ and $%
\xi =(\xi _{x})_{x\in G^{0}}\in L^{2}(G^{0},\mathcal{H},\tau )$, where $%
a_{x}(\delta _{\rho }\otimes |b\rangle )=\delta _{\rho }\otimes
|a_{x}b\rangle $ for $x\in G^{0}$, $\rho \in xG$, and $b\in \mathcal{M}_{x}$.

\begin{definition}
\label{Definition:crossed} Let $G$ be a pmp ergodic groupoid, and let $%
\alpha \colon G\rightarrow \mathrm{Aut}(\mathcal{M},\tau )$ be an action on
a von Neumann bundle. The \emph{crossed product }$G\ltimes ^{\alpha }(%
\mathcal{M},\tau )$ is the von Neumann subalgebra of $B(L^{2}(G^{0},\mathcal{%
H},\tau ))$ generated by $\{[\pi ^{\alpha }]_{\sigma }\colon \sigma \in
\lbrack G]\}\cup L^{\infty }\left( G^{0},\mathcal{M}\right) $, endowed with
a canonical faithful normal tracial state defined by $\tau (x)=\langle
1|x|1\rangle $, where $|1\rangle $ denotes the element $(\delta _{x}\otimes
|1_{x}\rangle )_{x\in G^{0}}$ of $L^{2}(G^{0},\mathcal{M},\tau )$.

The \emph{groupoid} \emph{von Neumann algebra }$L ( G ) $ is the crossed
product of the action of $G$ on $G^{0}$.
\end{definition}

\begin{remark}
The action of $G$ on $G^{0}$ can be seen as an action of $G$ on the trivial
von Neumann bundle $\mathcal{M}$ over $G^{0}$. In this case, in the
notations above, we have $\mathcal{H}=\bigsqcup_{x\in G^{0}}\ell ^{2}(xG)$,
and $L^{2}(G^{0},\mathcal{H})=L^{2}(G)$. Hence, $L(G)$ coincides with the
von Neumann subalgebra of $B(L^{2}(G))$ generated by $\{[[\lambda ]]_{\sigma
}\colon \sigma \in \lbrack \lbrack G]]\}$.
\end{remark}

Let $\xi _{0}\in L^{2} ( G ) $ be the element corresponding to the
characteristic function of $G^{0}$. Then $( L^{2} ( G ) ,\xi _{0} ) $ is the
pointed Hilbert space obtained from $( L^{\infty } ( G ) ,\tau ) $ via the
GNS construction. This allows one to define the canonical anti-unitary $%
J\colon L^{2} ( G ) \to L^{2} ( G ) $ by $J \vert a \rangle = \vert a^{\ast
} \rangle $ for $a\in L^{\infty } ( G ) $. The same proof as in the case of
countable pmp equivalence relations gives the following; see \cite%
{feldman_ergodic_1977-1}.

\begin{proposition}
Let $G$ be a discrete pmp groupoid, and let $\Sigma $ be a countable set of
pairwise essentially disjoint elements of $[[G]]$ that covers $G$. Then:

\begin{enumerate}
\item The characteristic function $\vert 1 \rangle \in L^{2} ( G ) $ of $G^0$
is a cyclic separating vector for $L ( G ) $.

\item The vector state $x\mapsto \langle 1|x|1 \rangle $ is a faithful
normal tracial state\ on $L ( G ) $.

\item For $\sigma \in [ [ G ] ] $ one has $J [ [ \lambda ] ] _{\sigma }J= [
[ \rho ] ] _{\sigma }$ and hence $L ( G ) ^{\prime }= \{ [ [ \rho ] ]
_{\sigma }\colon \sigma \in [ [ G ] ] \} $.

\item An element $a \in L ( G ) $ can be written uniquely as $a=\sum_{\sigma
\in \Sigma }a_{\sigma } [ [ \lambda ] ] _{\sigma }$, with $a_{\sigma }\in
L^{\infty } ( G^{0} ) $ and the convergence is in $2$-norm. In this case, $%
a_{\sigma }=\mathrm{E}_{L^{\infty } ( G_{0} ) } ( x [ [ \lambda ] ] _{\sigma
^{-1}} ) $ and $\Vert a \Vert _{2}^{2}=\sum_{\sigma \in \Sigma } \Vert
a_{\sigma } \Vert _{2}^{2}$.
\end{enumerate}
\end{proposition}

When $G$ is an ergodic principal discrete pmp groupoid such that $Gx$ is
infinite for almost every $x\in G^{0}$, then $L(G)$ is a II$_{1}$ factor
which contains $L^{\infty }(G^{0})$ as a maximal abelian subalgebra; see
\cite{feldman_ergodic_1977-1}.

%\subsection{Amenability\label{Subsection:amenability}}

%Suppose that $G$ is a countable pmp groupoid. Following \cite%
%{connes_amenable_1981,ad_amenable_2000}, we say that $G$ is \emph{amenable }%
%if there is a conditional expectation $m\colon L^{\infty } ( G )
%\to L^{\infty } ( G^{0} ) $ such that $m([\lambda
%]_{g}a[\lambda ]_{g}^{\ast })=[\kappa ^{\theta }]_{g}m ( a )
%[\kappa ^{\theta }]_{g}^{\ast }$ for $g\in  [ G ] $ and $a\in
%L^{\infty } ( G^{0} ) $ where $\lambda $ is the left regular
%representation of $G$ and $\theta $ is the canonical action of $G$ on $G^{0}$%
%; see \cite[Definition 3.2.8 and Proposition 3.2.14]{ad_amenable_2000}.

%Suppose that $R$ is an countable pmp equivalence relation. Then $R$ is
%amenable if it is amenable as a countable pmp groupoid; see also \cite[%
%Definition 4.57]{kerr_ergodic_2016}. It is shown in \cite%
%{connes_amenable_1981} that this is equivalent to the assertion that $R$ is
%\emph{hyperfinite}, namely up to discarding a set of measure zero, $R$ can
%be written as an increasing union of equivalence relations with finite
%classes; see also \cite[Definition 4.72]{kerr_ergodic_2016}. It follows from
%this and Dye's theorem \cite[Theorem 4.83]{kerr_ergodic_2016} that any two
%\emph{ergodic }amenable countable pmp equivalence relations on the standard
%atomless probability space are orbit equivalent.

\section{Coinduction theory for groupoids\label{Section:coinduction}}

\subsection{Bernoulli actions of groupoids}

Let $G$ be a discrete pmp groupoid. A \emph{bundle of countable sets} over $%
G $ is a standard Borel space $\mathcal{I}$ fibered over $G^{0}$ with
countable fibres, such that there exists a sequence $(i_{n})_{n\in \mathbb{N}%
}$ of sections for $\mathcal{I}$ such that $\{i_{n,x}\colon x\in G^{0}\}$
enumerates $\mathcal{I}_{x}$ for every $x\in G^{0}$. One can then define the
standard Borel groupoid $\mathrm{Sym}(\mathcal{I})$ with unit space $G^{0}$
consisting of bijections $\sigma \colon I_{x}\rightarrow I_{y}$ for $x,y\in
G^{0}$. The Borel structure on \textrm{Sym}$(\mathcal{I})$ is generated by
the source and range maps together with the subsets $\{\sigma \colon
I_{x}\rightarrow I_{y}\colon x,y\in G^{0},\sigma (i_{n,x})=i_{m,y}\}$ for $%
n,m\in \mathbb{N}$. An \emph{action} of $G$ on $\mathcal{I}$ is a
homomorphism from $G$ to $\mathrm{Sym}(\mathcal{I})$.

Let $(M,\tau )$ be a tracial von Neumann algebra with separable predual, and
let $L^{2}(M,\tau )$ be the corresponding Hilbert space obtained via the GNS
construction, with cyclic vector $|1\rangle $. Define a Hilbert bundle $%
L^{2}(M,\tau )^{\otimes \mathcal{I}}$ by $L^{2}(M,\tau )^{\otimes \mathcal{I}%
}=\bigsqcup_{x\in G^{0}}(L^{2}(M,\tau ),|1\rangle )^{\otimes \mathcal{I}%
_{x}} $. The standard Borel structure on $L^{2}(M,\tau )^{\otimes \mathcal{I}%
}$ can be described as follows. Fix a $\Vert \cdot \Vert _{2}$-dense
sequence $(d_{n})_{n\in \mathbb{N}}$ in $M$, and for every $n\in \mathbb{N}$%
, $\bar{k}=\left( k_{1},\ldots ,k_{n}\right) \in \mathbb{N}^{n}$ and $\bar{%
\ell}=\left( \ell _{1},\ldots ,\ell _{n}\right) \in \mathbb{N}^{n}$ with $%
\ell _{1}<\ell _{2}<\cdots <\ell _{n}$, define a section $\sigma ^{(\bar{k},%
\bar{\ell})}$ for $L^{2}(M,\tau )^{\otimes \mathcal{I}}$ by $\sigma _{x}^{(%
\bar{k},\bar{\ell})}=\left( |d_{k_{1}}\rangle \otimes \cdots \otimes
\left\vert d_{k_{n}}\right\rangle \right) _{(i_{\ell _{1},x}i_{\ell
_{2},x}\cdots i_{\ell _{n},x})}\in (L^{2}(M,\tau ),|1\rangle )^{\otimes
\mathcal{I}_{x}}$ for $x\in G^{0}$; see Notation \ref{Notation:leg}. Then
the Borel structure on $L^{2}(M,\tau )^{\otimes \mathcal{I}}$ is generated
by the maps $\xi \mapsto \langle \xi ,\sigma _{x}^{(\bar{k},\bar{\ell}%
)}\rangle $, where $\xi \in (L^{2}(M,\tau ),|1\rangle )^{\otimes \mathcal{I}%
_{x}}$, for $k,\ell \in \mathbb{N}$.

We define the tracial von Neumann bundle $(M,\tau) ^{\otimes \mathcal{I}%
}=\bigsqcup_{x\in G^{0}} ( M,\tau ) ^{\otimes \mathcal{I}_{x}}$, and endow
it with a Borel structure defined similarly as above.
%, by considering sections of the
%form $x\mapsto  ( d_{k} ) _{(i_{\ell ,n}^{\mathcal{I}})}\in  (
%D,\tau  ) ^{\otimes \mathcal{I}_{x}}$.

\begin{definition}
\label{Definition:Bernoulli-action}The \emph{Bernoulli action }$\beta
_{G\curvearrowright \mathcal{I}}$ of $G$ with base $(M,\tau) $ associated
with the action $G\curvearrowright \mathcal{I}$ is the action of $G$ on $%
(M,\tau) ^{\otimes \mathcal{I}}$ defined by $\beta _{G\curvearrowright
\mathcal{I},\gamma }(a_{(i)})=a_{(\gamma \cdot i)}$ for $\gamma\in G$, for $%
i\in \mathcal{I}_{s(\gamma )}$ and $a\in M$.
\end{definition}

When $(M,\tau )$ is an abelian tracial von Neumann algebra, one obtains the
notion of Bernoulli action of $G$ on a standard probability space.

\begin{example}
Let $G$ be a discrete pmp groupoid, let $K$ be an ergodic subgroupoid, and
let $G/K=\bigsqcup_{x\in G^{0}}xG/K$ be the corresponding quotient. Given a
subgroupoid $H\subseteq G$, consider the canonical action $H\curvearrowright
G/K$. If $(M,\tau ) $ is a tracial von Neumann algebra, we obtain a
Bernoulli action $\beta _{H\curvearrowright G/K}$ on $(M,\tau) ^{\otimes
G/K}=\bigsqcup_{x\in G^{0}} (M,\tau) ^{\otimes xG/K}$.
\end{example}

%It is easy to see that, in the case when $G$ is a principal discrete pmp
%groupoid, $H=G$, $K$ is the trivial subgroupoid, and $ (M,\tau) $
%is an abelian tracial von Neumann algebra, the corresponding Bernoulli
%action $\beta _{G\curvearrowright G}$ of $G$ can be identified with the
%Bernoulli extension of $G$ as defined in \cite[Section 3]%
%{bowen_neumanns_2015}. It is shown in \cite[Lemma 3.1]{bowen_neumanns_2015}
%that if $G$ is an ergodic principal discrete groupoid, then the Bernoulli
%action $\beta _{G\curvearrowright G}$ with base $ [ 0,1 ] $ is
%ergodic. The following lemma is inspired by \cite[Lemma 3.1]%
%{bowen_neumanns_2015}.

\begin{lemma}
\label{Lemma:bernoulli-wm} Let $\Lambda $ be a countable discrete group, and
let $\Delta\leq \Lambda $ be an infinite index subgroup. Let $G$ be a
principal discrete pmp groupoid with unit space $X$, and let $%
\Lambda\curvearrowright^{\theta}X$ be a free ergodic action satisfying $%
\{\theta_{\lambda}(x)\colon \lambda \in \Lambda \}\subseteq [ x ] _{G}$ for
almost every $x\in X$. Let $H=\Lambda \ltimes ^{\theta}X$ be the
corresponding action groupoid, which can be regarded as a subgroupoid of $G$%
, and set $K=\Delta \ltimes ^{\theta|_{\Delta }}X \leq H$. Given a tracial
von Neumann algebra $( M,\tau ) $, the Bernoulli action $\beta
_{H\curvearrowright G/K}$ of $H$ on $(M,\tau) ^{G/K}$ is weak mixing
relatively to the trivial sub-bundle.
\end{lemma}

\begin{proof}
Since $G,H$ and $K$ are a principal groupoids, we will identify them with
their corresponding orbit equivalence relations. Since $H$ is ergodic, we
can define the index $N\in \mathbb{N}\cup \left\{ \infty \right\} $ of $H$
in $G$. There exists a coset selection $\{\sigma _{n}\colon n\leq
N\}\subseteq \lbrack G]$ of $H$ in $G$ such that $G$ is the disjoint union
of $\{\sigma _{n}H\colon n\leq N\}$. One can identify $xG/K$ with the set $%
[x]_{G/K}=\{[y]_{K}\colon y\in \lbrack x]_{G}\}$ of $K$-classes contained
inside the $G$-class of $x$. Furthermore, one can identify $\sigma _{n}$
with a Borel function $\sigma _{n}\colon X\rightarrow X$ for which $[x]_{G}$
is the disjoint union of $\{[\sigma _{n}(x)]_{H}\colon n\leq N\mathbb{\}}$
for every $x\in X$. Moreover, $[x]_{G/K}$ is the disjoint union of $%
\{[\sigma _{n}(x)]_{H/K}\colon n\leq N\mathbb{\}}$ for every $x\in X$. This
gives an isomorphism of $(M,\tau )^{\otimes G/K}$ with the bundle $(\mathcal{%
M},\tau )=\bigsqcup_{x\in X}\bigotimes_{n\leq N}(M,\tau )^{\otimes \lbrack
\sigma _{n}(x)]_{H/G}}$.

For $x\in X$ and $n\leq N$, there is an isomorphism $(M,\tau )^{\otimes
\Lambda /\Delta }\rightarrow (M,\tau )^{\otimes \lbrack \sigma
_{n}(x)]_{H/K}}$ associated with the bijection $\Lambda /\Delta \rightarrow
\lbrack \sigma _{n}(x)]_{H/K}$ given by $\Delta \lambda \mapsto \lbrack
\theta _{\lambda }(\sigma _{n}(x))]_{K}$. These maps determine an isomorphism%
\begin{equation*}
L^{\infty }(X,(M,\tau )^{G/K})\cong L^{\infty }(X)\otimes ((M,\tau
)^{\otimes \Lambda /\Delta })^{\otimes N}\text{.}
\end{equation*}

Let $a_{1},\ldots ,a_{\ell }\in L^{\infty }(X,(M,\tau )^{G/K})$ be
contractions with $\mathrm{E}_{L^{\infty }(X)}(a_{i})=0$ for $i=1,2,\ldots
,\ell $, and fix $\varepsilon >0$. Find $n\in \mathbb{N}$, a finite subset $%
F\subseteq \Lambda /\Delta $, and contractions $a_{1}^{\prime },\ldots
,a_{\ell }^{\prime }\in L^{\infty }(X,(M,\tau )^{G/K})$ with
\begin{equation*}
\Vert a_{i}-a_{i}^{\prime }\Vert _{2}<\varepsilon ,\ \mathrm{E}_{L^{\infty
}(X)}(a_{i}^{\prime })=0,\ \mbox{ and }\ a_{i,x}^{\prime }\in
\bigotimes_{k\leq N}(M,\tau )^{\otimes _{\gamma \in F}[\theta _{\gamma
}(\sigma _{k}(x))]_{K}}
\end{equation*}%
for $1\leq i\leq \ell $, and almost every $x\in X$. Since $\Delta $ has
infinite index in $\Lambda $, there exists $\lambda \in \Lambda $ such that $%
\lambda F\cap F=0$. Identifying $\lambda $ with the corresponding element of
$[H]$, we have $\mathrm{E}_{L^{\infty }(X)}([\beta _{H\curvearrowright
G/K}]_{\lambda }(a_{i}^{\prime })a_{j}^{\prime })=0$, and therefore
\begin{equation*}
\mathrm{E}_{L^{\infty }(X)}([\beta _{H\curvearrowright G/K}]_{\lambda
}(a_{i})a_{j})<2\varepsilon
\end{equation*}%
for every $1\leq i,j\leq d$. Thus $\beta _{H\curvearrowright G/K}$ is weak
mixing relative to the trivial sub-bundle.
\end{proof}

\subsection{Coinduction for groupoids\label{Subsection:coinduction}}

Let $G$ be a discrete pmp groupoid, and let $H\leq G$ be a subgroupoid.
Recall that the equivalence relation $\sim _{H}$ on $G$ is defined by $%
\gamma \sim _{H}\gamma ^{\prime }$ if and only if $\gamma H=\gamma ^{\prime
}H$. Let $X$ be the unit space of $G$ and $H$, and let $\Lambda $ be a
countable discrete group. Assume that there exists a free action $\Lambda
\curvearrowright ^{\theta }X$ such that $H=\Lambda \ltimes ^{\theta }X$.

Let $G\ltimes G/H$ be the countable Borel groupoid associated with the
canonical action $G\curvearrowright G/H$, i.e.
\begin{equation*}
G\ltimes G/H=\{(\rho ,\gamma H)\in G\times G/H\colon s(\rho )=r(\gamma )\}.
\end{equation*}%
Let $T\colon G/H\rightarrow G$ be a \emph{Borel selector }for $\sim _{H}$,
i.e.\ a Borel map satisfying $T(\gamma H)H=\gamma H$ for every $\gamma \in G$%
. We will furthermore assume that $T(xH)=x$ for every $x\in G^{0}$.

\begin{lemma}
Let the notation be as before, and define a map $c\colon G\ltimes G/H\to
\Lambda $ by letting $c ( \rho ,\gamma H ) $ be the unique element of $%
\Lambda $ such that $T ( \rho \gamma H ) ^{-1}\rho T ( \gamma H ) =c ( \rho
,\gamma H ) \ltimes ^{\theta }x$ for some $x\in X$. Then $c$ is a
homomorphism.
\end{lemma}

\begin{proof}
We need only check that $c$ is multiplicative. Given $(\rho_0,\rho_1)\in G^2$%
, we have
\begin{equation*}
T ( \rho _{1}\rho _{0}H ) (c ( \rho _{1},\rho _{0}\gamma H ) c ( \rho
_{0},\gamma H ) \ltimes ^{\theta}x)= \rho _{1}T ( \rho _{0}\gamma H ) (c (
\rho _{0},\gamma H ) \ltimes ^{\theta}x)=\rho _{1}\rho _{0}T ( \gamma H ),
\end{equation*}
as desired.
\end{proof}

Recall that we use the leg-numbering notation for linear operators on tensor
product.

\begin{definition}
Adopt the notation of the discussion above, and let $\alpha \colon \Lambda
\rightarrow \mathrm{Aut}(M,\tau )$ be an action. The \emph{coinduced action }%
$\widehat{\alpha }=\mathrm{CInd}_{H}^{G}(\alpha )$ is the action of $G$ on
the tracial von Neumann bundle $(\mathcal{M},\tau )=\bigsqcup_{x\in
X}(M,\tau )^{\otimes xG/H}$ defined (on elementary tensors) by $\widehat{%
\alpha }_{\gamma }(a_{(\rho H)})=\alpha _{c(\gamma ,\rho H)}(a_{(\gamma \rho
H)})$ for $\gamma \in G$ and $a_{(\rho H)}\in (M,\tau )$ with $(\gamma ,\rho
)\in G^{2}$; see Notation \ref{Notation:leg}.
\end{definition}

%Observe that this is indeed an action.
%Using the fact that $\sigma \colon G\ltimes G/H\to \Lambda $ is a
%homomorphism,we have that, for $b\in D$,%
%\begin{eqnarray*}
%\widehat{\alpha}_{\gamma _{0}} ( \widehat{\alpha}_{\gamma _{1}} ( b_{(\rho
%H)} )  ) &=&\widehat{\alpha}_{\gamma _{0}}(\alpha _{c ( \gamma
%_{1},\rho H ) } ( b ) _{(\gamma _{1}\rho H)})=\alpha _{c (
%\gamma _{0},\gamma _{1}\rho H ) }(\alpha _{c ( \gamma _{1},\rho
%H ) } ( b ) )_{ ( \gamma _{0}\gamma _{1}\rho H ) } \\
%&=&\alpha _{c ( \gamma _{0},\rho H ) \sigma  ( \gamma _{1},\rho
%H ) } ( b ) _{(\gamma _{0}\gamma _{1}\rho H)}=\alpha _{c (
%\gamma _{0}\gamma _{1},\rho H ) } ( b ) _{(\gamma _{0}\gamma
%_{1}\rho H)} \\
%&=&\widehat{\alpha}_{\gamma _{0}\gamma _{1}} ( b_{(\rho H)} ) \text{.}
%\end{eqnarray*}%

\begin{example}
When $\alpha $ is the trivial action on $(M,\tau )$, then $\mathrm{CInd}%
_{H}^{G}(\alpha )$ is the Bernoulli shift $\beta _{G\curvearrowright G/H}$
as defined in the previous subsection.
\end{example}

\begin{remark}
Let $(Y,\nu )$ be a standard probability space, and suppose that $(M,\tau
)=L^{\infty }(Y,\nu )$, so that one can regard $\alpha $ as an action of $%
\Lambda $ on $Y$. Then $\widehat{\alpha }=\mathrm{CInd}_{H}^{G}(\alpha )$
can be seen as the action of $G$ on $\bigsqcup_{x\in X}Y^{xG/H}$ defined by
setting%
\begin{equation*}
(\hat{\alpha}_{\gamma ^{-1}}( \omega ) )( \rho H)
=\alpha _{c( \gamma ,\rho H) ^{-1}}( \omega ( \gamma
\rho H) )
\end{equation*}%
for $\gamma \in G$, $\rho H\in s( \gamma ) G/H$, and $\omega \in
Y^{s( \gamma ) G/H}$.
\end{remark}

\begin{proof}
For $a\in L^{\infty }( Y) $, we can consider $a_{( \rho
H) }\in L^{\infty }(Y^{s( \rho ) G/H})$. Then $\hat{%
\alpha}_{\gamma }( a_{( \rho H) }) \in L^{\infty
}(Y^{s( \gamma ) G/H})$ is defined by setting, for $\omega \in
Y^{s( \gamma ) G/H}$,%
\begin{equation*}
\hat{\alpha}_{\gamma }( a_{( \rho H) }) ( \omega
) =\alpha _{\sigma ( \gamma ,\rho H) }( a_{(
\gamma \rho H) }) ( \omega ) =a(\alpha _{c(
\gamma ,\rho H) ^{-1}}( \omega ( \gamma \rho H)
) )
\end{equation*}%
On the other hand we have that%
\begin{equation*}
\hat{\alpha}_{\gamma }( a_{( \rho H) }) ( \omega
) =a_{( \rho H) }(\hat{\alpha}_{\gamma ^{-1}}( \omega
) )=a((\hat{\alpha}_{\gamma ^{-1}}( \omega ) )(\rho H)\text{%
.}
\end{equation*}%
Since this holds for every $a\in L^{\infty }( Y) $, $(\hat{\alpha}%
_{\gamma ^{-1}}(\omega ))( \rho H) =\alpha _{c( \gamma ,\rho
H) ^{-1}}(\omega ( \gamma \rho H) )$.
\end{proof}

%More generally, one can define the coinduced action $\mathrm{CInd}%
%_{H}^{G} ( \alpha  ) $ whenever $H$ is a discrete pmp
%groupoid endowed with a distinguished homomorphism to $\Lambda $. The
%discussion above corresponds to the case when $H=\Lambda \ltimes ^{\theta
%^{\Lambda }}X$ for some free action $\theta$ of $\Lambda $, and
%the homomorphism $H\to \Lambda $ is given by $\gamma \ltimes
%^{\theta }x\mapsto \gamma $.

\begin{remark}
\label{Remark:coinduced-free} Let $\Gamma $ be a subgroup of $[ G ] $, and
assume that the canonical action $\Gamma \curvearrowright G^{0}$ is free. If
$\alpha $ is an action of $\Lambda $ on an abelian tracial von Neumann
algebra $(M,\tau)$, then $[ \widehat{\alpha} ] |_{\Gamma }$ is also free.
Indeed, $\Gamma \curvearrowright G^{0}$ can be identified with the
restriction of $[ \widehat{\alpha} ] |_{\Gamma }$ to $L^{\infty } ( G^{0} )
\subseteq L^{\infty } ( G^{0},\mathcal{M} ) $.
\end{remark}

Suppose now that $H$ is ergodic, and let $N$ be the index of $H$ in $G$, and
$\Sigma =\{\sigma _{n}:n\leq N\}\subseteq \lbrack G]$ be a coset selection
for $H$ in $G$ with $\sigma _{1}=G^{0}$. For every $x\in X$, the set $xG$ is
the disjoint union of $x\sigma _{n}H$, for $n\leq N$. For $\gamma \in G$,
let $\sigma _{\gamma }\in \Sigma $ be the unique element satisfying $\gamma
\in \sigma _{\gamma }H$. Then the function $G/H\rightarrow X\times \Sigma $
given by $\gamma H\mapsto (r(\gamma ),\sigma _{\gamma })$, is a Borel
isomorphism.

In the next proposition, we describe the form that $T$, $\alpha $ and $%
\widehat{\alpha }$ take under the isomorphism $G/H\cong X\times \Sigma $
described above. Its proof is straightforward, so we omit it.

\begin{proposition}
\label{proposition:identifications} Under the Borel isomorphism $G/H\cong
X\times\Sigma$ as above, we have the following identifications:

\begin{enumerate}
\item The Borel selector $T\colon G/H\to G$ can be seen as the function $T(
x,\sigma )= x\sigma$, for $(x,\sigma)\in X\times\Sigma$.

\item The action $G\curvearrowright G/H$ can be realized as follows. Define
a homomorphism $\pi\colon G\to \mathrm{Sym}(\Sigma)$ by letting $%
\pi_{\gamma}(\sigma)$ be the unique element of $\Sigma$ such that $\pi
_{\gamma } ( \sigma ) ^{-1}\gamma \sigma \in H$. Then $G\curvearrowright G/H$
can be identified with $\gamma\cdot (x,\sigma ) = ( s(\gamma ),\pi _{\gamma
} ( \sigma ) ) $ for $\gamma\in G$ and $(x,\sigma)\in X\times\Sigma$.

\item The homomorphism $c\colon G\ltimes G/H\to \Lambda $ can be seen to be
given by $c( \gamma ,\sigma ) = \delta _{\gamma ,\sigma }$, where
\begin{equation*}
\delta _{\gamma ,\sigma }\ltimes ^{\theta }s(\gamma \sigma )=\pi _{\gamma }
( \sigma ) ^{-1}\gamma \sigma
\end{equation*}%
for $\gamma \in G$ and $\sigma \in \Sigma $.

\item The coinduced action $\widehat{\alpha}$ can be identified with the
action $G$ on $\bigsqcup_{x\in X} (M,\tau) ^{\otimes \Sigma }$ defined by $%
\widehat{\alpha}_{\gamma } ( d_{(\sigma )} ) =\alpha _{\delta _{\gamma
,\sigma }} ( d _{(\pi _{\gamma }(\sigma ))})$ for $d\in M$.

\item When $(M,\tau )=L^{\infty }(Y,\nu )$, we can regard $\widehat{\alpha }$
as the action on $\bigsqcup_{x\in X}Y^{\Sigma }$ defined by $\widehat{\alpha
}_{\gamma }(\omega )(\sigma )=\alpha _{\delta _{\gamma ^{-1},\sigma
}^{-1}}(\omega (\pi _{\gamma }(\sigma )))$ for all $\omega \in Y^{\Sigma }$
and all $\gamma \in G$.
\end{enumerate}
\end{proposition}

\begin{lemma}
\label{Lemma:coinduced-wm} We keep the notation from the beginning of this
subsection. Let $\alpha \colon \Lambda \to\mathrm{Aut}(M,\tau) $ be an
action on a tracial von Neumann algebra $(M,\tau)$, and set $\mathrm{CInd}%
_{H}^{G} ( \alpha ) =\widehat{\alpha}$.

\begin{enumerate}
\item If $\alpha $ is ergodic and $G$ is ergodic, then $\widehat{\alpha}$ is
ergodic.

\item If $\alpha $ is weak mixing, then $\widehat{\alpha}$ is weak mixing
relatively to the trivial sub-bundle.
\end{enumerate}
\end{lemma}

\begin{proof}
(1) Set $( \mathcal{M},\tau ) =\bigsqcup_{x\in X} ( M,\tau ) ^{\otimes
\Sigma }$. The proof of \cite[Proposition 7.2]{bowen_neumanns_2015}(2) shows
that $L^{\infty } ( X,\mathcal{M} ) ^{\widehat{\alpha}}\subseteq L^{\infty }
( X ) $ whenever $\alpha $ is ergodic. Thus if $G$ is ergodic, then $%
L^{\infty } ( X,\mathcal{M} ) ^{\widehat{\alpha}}$ is trivial.

(2) The action $\widehat{\alpha}\otimes \widehat{\alpha} $ can be identified
with $\widehat{\alpha \otimes \alpha}$. If $\alpha $ is weak mixing, then $%
\alpha \otimes \alpha $ is ergodic. Hence, $L^{\infty } ( X,\mathcal{M}%
\otimes \mathcal{M} ) ^{\widehat{\alpha}\otimes \widehat{\alpha}}$ is
contained in $L^{\infty } ( X ) $ by \cite[Proposition 7.2]%
{bowen_neumanns_2015}(2), so the result follows from Lemma \ref{Lemma:wm}.
\end{proof}

%\begin{remark}
%\label{Remark:coinduce}Adopting the notation above, suppose that $\theta
%^{\Gamma }$ is a free action of a countable group $\Gamma $ on $X$ and let $%
%G $ be the action groupoid $\Gamma \ltimes ^{\theta ^{\Gamma }}X$. If $%
%\alpha $ is an action of $\Lambda $ on a standard probability space, and $%
%\widehat{\alpha} $ is the coinduced action $\mathrm{CInd}_{H}^{G} ( \alpha
% ) $ as defined above, then $ [ \widehat{\alpha} ] |_{\Gamma }$ is
%(conjugate to) the coinduced action\emph{\ }of $\alpha $ modulo $ (
%\theta ^{\Lambda },\theta ^{\Gamma } ) $ as defined in \cite%
%{ioana_subequivalence_2009}; see also \cite%
%{ioana_orbit_2011,epstein_orbit_2007,bowen_neumanns_2015}.
%\end{remark}

\begin{lemma}
\label{Lemma:choose} Let $G$ be an ergodic discrete pmp groupoid such that $%
G_x$ is infinite for almost every $x\in G^0$, and let $F\colon G^0\to
\{A\subseteq G\colon A \mbox{ is finite}\}$ be a Borel assignment such that $%
F_{x}\subseteq xG$ for every $x\in G^{0}$. Let $\varepsilon>0$. Then there
exists $t\in [ G ] $ such that $\mu \left(\{ x\in G^{0}\colon xt\in F_{x} \}
\right)<\varepsilon $.
\end{lemma}

\begin{proof}
Fix a sequence $( t_{n} ) _{n\in \mathbb{N}}$ in $[ G ] $ whose union is
equal to $G$, and choose a Borel partition $( X_{n} )_{n\in\mathbb{N}} $ of $%
G^{0}$ such that $F_{x}\subseteq \{ xt_{1},\ldots ,xt_{n} \} $ and $\mu (
X_{n} ) \leq \frac{\varepsilon }{n+1}$ for every $n\in \mathbb{N}$, and
almost every $x\in X_{n}$. Inductively define, for $k\in\mathbb{N}$,
pairwise disjoint subsets $A_{k}\subseteq G^{0}$ and elements $\rho _{k}\in
[ G ] $ such that $s ( x\rho _{k} ) \in A_{k}$ and $x\rho _{k}\notin F_{x}$
for $x\in X_{k}$. The construction can proceed as long as $A_{1}\cup \cdots
\cup A_{k}$ has measure less than or equal to $1-\varepsilon $. If the
construction can proceed for every $k$, then one can find $t\in [ G ] $
satisfying $xt=x\rho _{j}$ for all $x\in X_{j}$ and $j\in \mathbb{N}$. If
the construction stops at some $k\in \mathbb{N}$, then one can choose an
arbitrary $\rho \in [ G ] $ such that $x\rho \in G^{0}\setminus ( A_{1}\cup
\cdots \cup A_{k} ) $ for $x\in G^{0}\setminus ( X_{1}\cup \cdots \cup X_{k}
) $, and then choose $t\in [ G ] $ such that $xt=x\rho _{j}$ for $x\in X_{j}$
and $j\leq k$, and $xt=x\rho $ for $x\in X_{j}$ and $j>k$. In this case, we
have $\mu ( X_{1}\cup \cdots \cup X_{k} ) =\mu ( A_{1}\cup \cdots \cup A_{k}
) >1-\varepsilon $, which concludes the proof.
\end{proof}

\begin{lemma}
\label{Lemma:coinduce-bernoulli} Let $G$ be a principal discrete pmp
groupoid with unit space $X$, let $\Lambda $ be a countable discrete group,
and let $\Delta\leq \Lambda $ be an infinite index infinite subgroup. Let $%
\beta $ denote the Bernoulli action $\beta _{\Lambda \curvearrowright
\Lambda /\Delta }$ with base $M=L^{\infty } ( [ 0,1 ] ) $, and let $%
\Lambda\curvearrowright^{\theta}X$ be a free weak mixing action such that $%
\{\theta_{\lambda}(x)\colon \lambda \in \Lambda \}\subseteq [ x ] _{G}$ for
almost every $x\in X$. Identify the action groupoid $H=\Lambda \ltimes
^{\theta}X$ with a subgroupoid of $G$, and consider the coinduced action $%
\widehat{\beta}=\mathrm{CInd}_{H}^{G} ( \beta ) $. Then $[\widehat{\beta}%
]|_{\Lambda }$ is weak mixing and malleable.
\end{lemma}

\begin{proof}
We first show that $[\widehat{\beta }]|_{\Lambda }$ is weak mixing. Since
the action $\Lambda \curvearrowright ^{\theta }X$ is weak mixing, in view of
Lemma \ref{Lemma:bootstrap-weak-mixing} it suffices to show that $\widehat{%
\beta }|_{H}$ is weak mixing relative to the trivial sub-bundle. Let $\Sigma
=\{\sigma _{n}\}_{n\in \mathbb{N}}$ be a coset selection for $H$ in $G$. We
implicitly use the identifications from Proposition~\ref%
{proposition:identifications}, so we regard $\widehat{\beta }|_{H}$ as an
action on $\bigsqcup_{x\in X}(M^{\otimes \Lambda /\Delta })^{\otimes \Sigma
} $. The canonical identification of $(M^{\otimes \Lambda /\Delta
})^{\otimes \Sigma }$ with $M^{\otimes (\Lambda /\Delta \times \Sigma )}$
allows one to regard $\widehat{\beta }|_{H}$ as an action on the bundle $%
\mathcal{M}=\bigsqcup_{x\in X}M^{\otimes ((\Lambda /\Delta )\times \Sigma )}$%
, which is defined by
\begin{equation*}
\widehat{\beta }_{\lambda \ltimes x}(d_{(\lambda _{0}\Delta ,\sigma
)})=d_{(\delta _{\lambda \ltimes x,\sigma }\lambda _{0}\Delta ,\pi _{\lambda
\ltimes x}(\sigma ))}
\end{equation*}%
for $d\in M$, $\lambda _{0}\Delta \in \Lambda /\Delta $, $\sigma \in \Sigma $%
, $\lambda \in \Lambda $, and $x\in X$.

Fix $\varepsilon >0$ and $n\in \mathbb{N}$, and let $a_{1},\ldots ,a_{n}\in
L^{\infty } ( X,\mathcal{M} ) =L^{\infty } ( X ) \otimes M^{\otimes (
(\Lambda /\Delta) \times \Sigma ) }$ be contractions satisfying $\mathrm{E}%
_{L^{\infty } ( X ) } ( a_{j} ) =0$ for $j=1,\ldots ,n$. We will show that
there exists $t\in [ G ] $ such that $\mathrm{E}_{L^{\infty } ( X ) }([%
\widehat{\beta}]_{t}(a_{i})a_{j})<\varepsilon $ for $i,j\in \{ 1,\ldots ,n
\} $. Without loss of generality, we can assume that there exist finite
subsets $F\subseteq \Lambda /\Delta $ and $S\subseteq \Sigma $ such that $%
a_{i,x}\in M^{\otimes (F\times S)}$ for almost every $x\in X$ and every $%
i\in \{ 1,\ldots ,n \} $.

Fix $x\in X$. We claim that for every $\sigma\in \Sigma $ and $%
\lambda_{0}\Delta \in F$, the set
\begin{equation*}
\{\lambda\in \Lambda\colon (\delta _{\lambda\ltimes
x,\sigma}\lambda_{0}\Delta ,\pi _{\lambda\ltimes x}(\sigma))\in F\times S\}
\end{equation*}
is finite. Since $\Delta $ has infinite index in $\Lambda $, it follows that
$\lambda\Delta $ is disjoint from $\Delta $ for all but finitely many $%
\lambda\in\Lambda$. Suppose by contradiction that there exists an infinite
sequence $( \lambda_{k} )_{k\in\mathbb{N}} $ in $\Lambda $ such that, for
every $k\in \mathbb{N}$ there exist $\sigma _{n_k}\in \Sigma $ and $%
\lambda_{k}\Delta \in F$ such that $(\delta _{\lambda_{k}\ltimes x,\sigma
_{n_k}}\lambda_{k}\Delta ,\pi _{\delta_{k}\ltimes x}(\sigma _{n_k}))$
belongs to $F\times S$. After passing to a subsequence, we can assume that
there exist $\sigma ,\sigma ^{\prime }\in \Sigma $ and $\lambda\Delta
,\lambda^{\prime }\Delta \in F$ such that
\begin{equation*}
(\delta _{\lambda_{k}\ltimes x,\sigma }\lambda\Delta ,\pi
_{\lambda_{k}\ltimes x}(\sigma ))= ( \lambda^{\prime }\Delta ,\sigma
^{\prime } )
\end{equation*}%
for every $k\in \mathbb{N}$. Recall that $\pi _{\lambda_{k}\ltimes x}$ is
the permutation of $\Sigma $ defined by letting $\pi _{\lambda_{k}\ltimes x}
( \sigma ) $ be the unique element $\sigma ^{\prime }$ of $\Sigma $ such
that $\sigma ^{\prime -1} ( \lambda_{k}\ltimes x ) \sigma \in H$, while $%
\delta _{\lambda_{k}\ltimes x,\sigma }$ is the unique element of $\Lambda $
such that $\delta _{\lambda_{k}\ltimes x,\sigma }\ltimes (s(x\sigma
))=\sigma ^{\prime -1} ( \lambda_k\ltimes x ) \sigma $ or, equivalently, $%
\lambda_{k}\ltimes x=\sigma ^{\prime } ( \delta _{\lambda_{k}\ltimes
x,\sigma }\ltimes (s(x\sigma )) ) \sigma ^{-1}$. Since $( \lambda_{k} )_{k\in%
\mathbb{N}} $ is an infinite sequence in $\Lambda $, the set $\{ \delta
_{\lambda_{k}\ltimes x,\sigma }\colon k\in \mathbb{N} \} \subseteq \Lambda $
is infinite. Since $\Delta $ has infinite index in $\Lambda $, the coset $%
\widetilde{\lambda} \lambda \Delta$ is disjoint from $\lambda^{\prime}\Delta$
for all but finitely many $\widetilde{\lambda}\in \Lambda $. Therefore there
exists $k\in \mathbb{N}$ such that $\delta _{\lambda_{k}\ltimes x,\sigma
}\lambda\Delta $ is disjoint from $\lambda^{\prime }\Delta $. This
contradicts the previous conclusion that $\delta _{\lambda_{k}\ltimes
x,\sigma }\lambda\Delta =\lambda^{\prime }\Delta $ for every $k\in \mathbb{N}
$, and proves the claim.

Use Lemma \ref{Lemma:choose} to choose $t\in [ H ] $ and a Borel subset $%
A\subseteq X$ such that $\mu ( A ) >1-\varepsilon$ and $(\delta
_{xt,\sigma}\lambda_j\Delta ,\pi _{xt}(\sigma _{0}))\notin F\times T$ for
every $x\in A$, every $\sigma\in \Sigma$ and whenever $\lambda_{j}\Delta \in
F$. Hence, $\mathrm{E}_{L^{\infty } ( X ) }([\widehat{\beta}%
]_{t}(a_{i})a_{j})<\varepsilon $ for $i\in \{ 1,2,\ldots ,n \} $, since $%
a_{i,x},a_{j,x}\in M^{\otimes ( F\times T ) }$ for almost every $x\in X$.
This concludes the proof that $\widehat{\beta}|_{H}$ is weak mixing
relatively to the trivial sub-bundle.

We now show that $[\widehat{\beta}]|_{\Lambda }$ is malleable. Use
malleability of $\beta$ to choose a continuous path $( \alpha _{t} ) _{t\in
[ 0,1 ] }$ in $\mathrm{Aut} ( M\otimes M )^{\beta\otimes\beta} $ such that $%
\alpha _{0}=\mathrm{id}_{M\otimes M}$ and $\alpha _{1}$ is the flip
automorphism; see \cite[Lemma 4.4]{furman_popas_2007}. Let%
\begin{equation*}
\Phi \colon L^{\infty } ( X ) \otimes (M^{\otimes \Lambda /\Delta
})^{\otimes \Sigma }\otimes L^{\infty } ( X ) \otimes (M^{\otimes \Lambda
/\Delta })^{\otimes \Sigma }\to L^{\infty } ( X ) \otimes L^{\infty } ( X )
\otimes ( M\otimes M ) ^{\otimes ( (\Lambda /\Delta) \times \Sigma ) }
\end{equation*}%
be the canonical isomorphism obtained by rearranging the tensor factors. For
$t\in [ 0,1 ] $ define
\begin{equation*}
\widehat{\alpha}_{t}\in \mathrm{Aut} ( L^{\infty } ( X ) \otimes (M^{\otimes
\Lambda /\Delta })^{\otimes \Sigma }\otimes L^{\infty } ( X ) \otimes
(M^{\otimes \Lambda /\Delta })^{\otimes \Sigma } )
\end{equation*}%
by%
\begin{equation*}
\widehat{\alpha}_{t}=\Phi ^{-1}\circ (\mathrm{id}_{L^{\infty } ( X ) \otimes
L^{\infty } ( X ) }\otimes \alpha _{t}^{\otimes ((\Lambda /\Delta) \times
\Sigma )})\circ \Phi \text{.}
\end{equation*}%
Then $( \widehat{\alpha}_{t} ) _{t\in [ 0,1 ] }$ is a path in the
centralizer of $[\widehat{\beta}]|_{\Lambda }\otimes \lbrack \widehat{\beta}%
]|_{\Lambda_{} }$ satisfying $\widehat{\alpha}_{0}$ is the identity and $%
\widehat{\alpha}_{1}$ is the flip automorphism. This concludes the proof.
\end{proof}

\section{Main result and consequences\label{Section:main}}

\subsection{Expansions of the rigid action\label{Subsection:reduction}}

%This means that the inclusion of $L^{\infty } ( \mathbb{T}%
% ) $ inside the crossed product $\Lambda \ltimes ^{\rho }L^{\infty
%} ( \mathbb{T} ) $ is rigid in the sense of Popa \cite[Definition
%4.2.1]{popa_class_2006}.
The notion of \emph{expansion }of countable pmp equivalence relations has
been introduced in \cite{bowen_neumanns_2015}, and we briefly recall it.

\begin{definition}
Let $R$ and $\widehat{R}$ be countable pmp equivalence relations over a
standard probability spaces $( X,\mu ) $ and $(\widehat{X},\widehat{\mu})$.
We say that $\widehat{R}$ is an \emph{expansion} of $R$ if there exists a
Borel map $\pi \colon \widehat{X}\to X$ with $\pi _{\ast }\widehat{\mu}=\mu$
such that:

\begin{itemize}
\item the restriction of $\pi $ to the $\widehat{R}$-class $[ x ] _{\widehat{%
R}}$ of $x$ is one-to-one, for almost every $x\in \widehat{X}$;

\item the image of $[ x ] _{\widehat{R}}$ under $\pi $ contains the $R$%
-class $[ \pi ( x ) ] _{R}$ of $\pi ( x ) $, for almost every $x\in \widehat{%
X}$.
\end{itemize}

If furthermore $\pi ([x]_{\widehat{R}})=[\pi (x)]_{R}$ for almost every $%
x\in X$, we say that $\widehat{R}$ is a \emph{class-bijective} \emph{%
extension} of $R$.
\end{definition}

As remarked in Subsection \ref{Subsection:actions-spaces}, one can identify
a class-bijective extension of $R$ on $Y$ with the action groupoid
associated with an action of $R$ on $Y$. Conversely, the action groupoid
associated with an action of $R$ on $Y$ can be seen as a class-bijective
extension of $R$. Therefore, we can identify actions of $R$ on $Y$ and
class-bijective extensions of $R$ on $Y$.

Let $\Lambda $ be a nonamenable subgroup $\mathrm{SL}_{2}(\mathbb{Z})$. Then
the canonical action of $\mathrm{SL}_{2}(\mathbb{Z)}\curvearrowright \mathbb{%
Z}^{2}$ restricts to an action $\Lambda \curvearrowright \mathbb{Z}^{2}$,
which induces a free weak mixing action $\Lambda \curvearrowright ^{\rho }%
\mathbb{T}^{2}$ by duality; see Subsection \ref{Subsection:rigid}. It
follows from \cite[Theorem 0.1]{ioana_relative_2010} that the orbit
equivalence relation of $\rho $ is \emph{rigid }in the sense defined
therein. Using the observations above, the same proof as \cite[Lemma 7.4]%
{bowen_neumanns_2015} shows the following.

\begin{theorem}[Bowen--Hoff--Ioana]
\label{Theorem:separability} Let $\Lambda $ be a nonamenable subgroup of $%
\mathrm{SL}_{2}(\mathbb{Z)}$ and let $\rho $ be the induced action $%
\Lambda\curvearrowright^{\rho} \mathbb{T}^{2}$. Let $G$ be a principal
discrete pmp groupoid with unit space $X$, and let $\Lambda
\curvearrowright^{\theta}X$ be a free weak mixing action such that $%
\{\theta_{\lambda}(x)\colon \lambda \in \Lambda \}\subseteq [ x ] _{G}$ for
almost every $x\in X$. Identify $H=\Lambda \ltimes ^{\theta}X$ with a
subgroupoid of $G$, and $\Lambda $ with a subgroup of $[ G ] $.

Suppose that $\mathcal{S}$ is a collection of actions $G\curvearrowright X$,
and let $[ \mathcal{S} ] |_{\Lambda }$ denote the collection $\{ [ \alpha ]
|_{\Lambda }\colon \alpha \in \mathcal{S} \} $. Assume that:

\begin{enumerate}
\item the elements of $\mathcal{S}$ have stably isomorphic crossed products;

\item the elements of $[ \mathcal{S} ] |_{\Lambda }$ are pairwise
non-conjugate;

\item the orbit equivalence relation of any element of $\mathcal{S}$ is an
expansion of the orbit equivalence relation of $\rho $.
\end{enumerate}

Then $\mathcal{S}$ is countable.
\end{theorem}

In the next proposition, we show that there exists a free subgroup of SL$%
_{2}(\mathbb{Z})$ that acts freely on $\mathbb{Z}^{2}\setminus \{0\}$. For
the usual free subgroups of SL$_{2}(\mathbb{Z})$ with finite index, the
stabilizers are cyclic but not in general trivial. This is, however, not
enough for our purposes.

\begin{proposition}
\label{prop:FreeAction} There exists a subgroup $\Lambda\subseteq \mathrm{SL}%
_{2}(\mathbb{Z)}$, which is isomorphic to $\mathbb{F}_{\infty}$, and such
that the induced action $\Lambda\curvearrowright \mathbb{Z}^2\setminus\{0\}$
is free.
\end{proposition}

\begin{proof}
It is clear that the set of elements of $\mathrm{SL}_{2}(\mathbb{Z)}$ that
belong to the stabilizer of a point in $\mathbb{Z}^{2}\setminus \{ 0 \} $ is
contained in the set of the matrices in $\mathrm{SL}_{2}(\mathbb{Z)}$ with
eigenvalues equal to $1$, which in turn coincides with set of matrices in $%
\mathrm{SL}_{2}(\mathbb{Z)}$ with trace $2$. Thus, it suffices to find a
copy of $\mathbb{F}_{\infty}$ in $\mathrm{SL}_2(\mathbb{Z})$ whose
nontrivial elements have trace other than two.

It is shown in \cite[Theorem 1]{newman_pairs_1968} that the matrices%
\begin{equation*}
A=%
\begin{bmatrix}
0 & 1 \\
-1 & 2%
\end{bmatrix}%
\text{\quad and\quad }B=%
\begin{bmatrix}
0 & -1 \\
1 & 2%
\end{bmatrix}%
\end{equation*}%
generate a free subgroup of rank $2$ of $\mathrm{SL}_{2}(\mathbb{Z)}$. We
adopt the notation from \cite[Theorem 1]{newman_pairs_1968}, and for
matrices $X,Y\in M_{n} ( \mathbb{Z} ) $ write $X\gg Y$ if every entry of $X$
is greater than or equal to the absolute value of the corresponding entry of
$Y$. For an integer $n$, we let $\mathrm{sgn} ( n ) $ be its sign. It is
shown in \cite[Lemma 1]{newman_pairs_1968} that $\mathrm{sgn} ( rs )
A^{r}B^{s}\gg \vert rs \vert I$ for every $r,s\in \mathbb{Z}\setminus \{0 \}
$. Let $n\in \mathbb{N}$ and $r_{1},\ldots ,r_{n},s_{1},\ldots ,s_{n}\in
\mathbb{Z}\setminus \{ 0 \} $, and set $r=r_{1}\cdots r_{n}$ and $%
s=s_{1}\cdots s_{n}$. Then
\begin{equation*}
\mathrm{sgn} ( rs ) A^{r_{1}}B^{s_{1}}\cdots A^{r_{n}}B^{s_{n}}\gg \vert rs
\vert I\text{.}
\end{equation*}%
In particular, this shows that%
\begin{equation*}
\vert \mathrm{Tr} [ A^{r_{1}}B^{s_{1}}\cdots A^{r_{n}}B^{s_{n}} ] \vert \geq
2 \vert rs \vert
\end{equation*}%
where $\mathrm{Tr}$ denotes the canonical trace of $2\times 2$ matrices.

Consider now the elements $x_{n}=A^{2n}B^{2n}$ for $n\geq 1$. These freely
generate a subgroup $\Lambda \cong \mathbb{F}_{\infty }$ of $\mathrm{SL}_{2}(%
\mathbb{Z)}$. If $\lambda $ is a nontrivial element of $\Lambda $, then
there exist $n\geq 1$ and nonzero \emph{even} integers $r_{1},\ldots ,r_{n}$
such that%
\begin{equation*}
|\mathrm{Tr}[\lambda ]|=|\mathrm{Tr}[A^{r_{1}}B^{r_{1}}\cdots
A^{r_{n}}B^{r_{n}}]|\geq 2|r_{1}\cdots r_{n}|^{2}\geq 4\text{.}
\end{equation*}%
This shows that every nontrivial element of $\Lambda $ has trace different
from $2$, and hence the canonical action $\Lambda \curvearrowright \mathbb{Z}%
^{2}\setminus \{0\}$ is free.
\end{proof}

\subsection{Class-bijective extensions of Borel equivalence relations}

We turn to the main result of this section. A countable pmp equivalence
relation is said to be \emph{amenable} if it is amenable as a discrete pmp
groupoid; see \cite{connes_amenable_1981,ad_amenable_2000}.

We let $\Lambda $ the group subgroup of $\mathrm{SL}_{2}(\mathbb{Z)}$
provided by Proposition~\ref{prop:FreeAction}, and recall that $\Lambda
\cong \mathbb{F}_{\infty }$. We also let $\Xi $ be the cyclic subgroup of $%
\Lambda $ generated by $x_{1}=A^{2}B^{2}$.

\begin{theorem}[Bowen--Hoff--Ioana \protect\cite{bowen_neumanns_2015}]
\label{Theorem:GL2} Let $R$ be an ergodic nonamenable countable pmp
equivalence relation on the standard probability space $(X,\mu)$, let $\beta
_{R\curvearrowright R}$ be the Bernoulli action with base space $(X,\mu)$,
and let $G$ be the corresponding action groupoid, which is a class-bijective
pmp extension of $R$. Then there exists a free weak mixing action $\Lambda
\curvearrowright^{\theta} X$ such that $\{\theta_{\lambda}(x)\colon \lambda
\in \Lambda\}\subseteq [ x ] _{G}$ for almost every $x\in X$. In particular,
the action groupoid $H=\Lambda\rtimes^{\theta}X$ can be canonically
identified with a subgroupoid of $G$. Furthermore, one can assume that the
restriction of $\theta$ to $\Xi $ is weak mixing.
\end{theorem}

\begin{proof}
It is enough to observe that the proof of \cite[Theorem A]%
{bowen_neumanns_2015} shows that one can choose the action in the statement
in such a way that the restriction to $\Xi $ is ergodic, and apply Dye's
theorem to $\Xi $ \cite[Theorem 3.13]{kechris_global_2010}.
\end{proof}

In the next proposition, we show that certain coinduction of the rigid
action $\Lambda\curvearrowright^{\rho} \mathbb{T}^2$ is weak mixing.

\begin{proposition}
\label{prop:restrCoindWkMix} Let $R$ be an ergodic nonamenable countable pmp
equivalence relation on the standard probability space $(X,\mu)$, let $\beta
_{R\curvearrowright R}$ be the Bernoulli action with base space $(X,\mu)$,
and let $G$ be the corresponding action groupoid. Let $\Lambda%
\curvearrowright^{\theta} X$ be the action provided by Theorem~\ref%
{Theorem:GL2}, let $\Delta $ be an infinite index subgroup of $\Lambda $
containing $\Xi$, and let $H$ be the action groupoid $\Lambda \ltimes
^{\theta}X$, which can be identified with a subgroupoid of $G$. Let $\rho $
be the rigid action $\Lambda \curvearrowright \mathbb{T}^{2}$, and set $%
\widehat{\rho}=\mathrm{CInd}_{H}^{G} ( \rho ) $. Then $[ \widehat{\rho} ]
|_{\Delta }$ is weak mixing.
\end{proposition}

\begin{proof}
We will use the identifications from Proposition~\ref%
{proposition:identifications}. Set $K=\Delta \ltimes ^{\theta |_{\Delta }}X$%
, and regard $\widehat{\rho }|_{K}$ as an action on the bundle $\mathcal{M}%
=\bigsqcup_{x\in X}L^{\infty }(\mathbb{T}^{2})^{\otimes \Sigma }$. Since $%
\theta |_{\Delta }$ is weak mixing, in view of Lemma \ref%
{Lemma:bootstrap-weak-mixing} it suffices to show that $\widehat{\rho }|_{K}$
is weak mixing relatively to the trivial sub-bundle. By Corollary \ref%
{Corollary:wm}, it suffices to show that the Koopman representation $\kappa
_{0}^{\widehat{\rho }|_{K}}$ is weak mixing.

As observed in \cite[Lemma 6.7]{epstein_borel_2011}, the Koopman
representation $\kappa _{0}^{\rho }$ of $\rho $ can be identified with the
representation of $\Lambda $ on $\ell ^{2} ( \mathbb{Z}^{2}\setminus \{ 0 \}
) $ obtained from the canonical action of $\Lambda \curvearrowright \mathbb{Z%
}^{2}\setminus \{ 0 \} $ as a subgroup of $\mathrm{SL}_{2} ( \mathbb{Z} ) $.
Similarly, and letting $\mathcal{F} ( \Sigma ,\mathbb{Z}^{2} ) _{0}$ denote
the set of non-zero finitely-supported functions $f\colon \Sigma \to \mathbb{%
Z}^{2}$, the representation $\kappa ^{\widehat{\rho}|_{K}}_0$ can be seen as
the representation on the Hilbert bundle $\mathcal{H}=\bigsqcup_{x\in X}\ell
^{2}(\mathcal{F} ( \Sigma ,\mathbb{Z}^{2} ) _{0})$, defined as follows. Let $%
\{\delta _{f}\colon f\in \mathcal{F} ( \Sigma ,\mathbb{Z}^{2} ) _{0}\}$ be
the canonical orthonormal basis of $\ell ^{2}(\mathcal{F} ( \Sigma ,\mathbb{Z%
}^{2} ) _{0})$. Consider the action of $G$ on $\Sigma $ given by $\gamma
\cdot \sigma =\pi _{\gamma } ( \sigma ) $ for $\gamma\in G$ and $%
\sigma\in\Sigma$. (Recall that $\gamma \cdot \sigma =\sigma ^{\prime }$ if
and only if $\gamma \sigma H\in \sigma ^{\prime }H$.) Define the action of $%
K $ on $\mathcal{F} ( \Sigma ,\mathbb{Z}^{2} ) _{0}$ by setting $( \gamma
\cdot f ) ( \gamma \cdot \sigma ) =\delta _{\gamma ,\sigma }\cdot f ( \sigma
) $ for $\gamma \in K$, $\sigma \in \Sigma $, and $f\in \mathcal{F} ( \Sigma
,\mathbb{Z}^{2} ) _{0}$. Then $\kappa _{0,\gamma }^{\widehat{\rho}|_{K}} (
\delta _{f} ) =\delta _{\gamma \cdot f}$ for $f\in \mathcal{F} ( \Sigma ,%
\mathbb{Z}^{2} ) _{0}$ and $\gamma \in K$.

Let $\xi _{1},\ldots ,\xi _{n}$ be invariant unit sections for $\mathcal{H}$%
, and let $\varepsilon >0$. We will show that there exists $t\in \lbrack K]$
such that
\begin{equation*}
\int_{X}\langle \kappa _{0,xt}^{\widehat{\rho }|_{K}}(\xi _{i,s(xt)}),\xi
_{j,x}\rangle d\mu (x)\leq \varepsilon
\end{equation*}%
for $i,j\in \{1,\ldots ,n\}$. Without loss of generality, we can assume that
there exists a finite subset $F\subseteq \mathcal{F}(\Sigma ,\mathbb{Z}%
^{2})_{0}$ such that $\xi _{i,x}\in \mathrm{span}\{\delta _{f}\colon f\in
F\} $ for every $x\in X$ and $1\leq i\leq n$.

As the action $\theta $ is free, we can assume, after discarding a null
subset of $X$, that $\theta _{\lambda }(x)\neq x$ for every $\lambda \in
\Lambda \setminus \{ 1\} $ and for every $x\in X$. Fix $x\in X$.
We claim first that there exists a finite subset $\Delta _{x}\subseteq
\Delta $ such that $\kappa _{0,h\ltimes x}^{\widehat{\rho }|_{K}}(\xi _{i,x})
$ is orthogonal to $\xi _{j,\theta _{h}(x)}$ for every $h\in \Delta
\setminus \Delta _{x}$, for $i,j\in \{1,\ldots ,n\}$. Suppose by
contradiction that this is not the case. Then there exist $f,f^{\prime }\in F
$ and an infinite sequence $(h_{n})_{n\in \mathbb{N}}$ of pairwise distinct
elements of $\Delta $ such that $\kappa _{0,h_{n}\ltimes x}^{\widehat{\rho }%
|_{K}}(\delta _{f})$ is not orthogonal to $\delta _{f^{\prime }}$ for all $%
n\in \mathbb{N}$. By the definition of $\kappa _{0,h\ltimes x}^{\widehat{%
\rho }|_{K}}$, after passing to a subsequence of $(h_{n})_{n\in \mathbb{N}}$%
, we can furthermore assume that there exist $\sigma ,\sigma ^{\prime }\in
\mathrm{supp}(f)$ such that $\delta _{h_{n}\ltimes x,\sigma }$ belongs to a
coset of the stabilizer of $f(\sigma )$ in $\Lambda $, and that $%
(h_{n}\ltimes x)\cdot \sigma =\sigma ^{\prime }$ for all $n\in \mathbb{N}$.
Since $\Lambda $ acts freely on $\mathbb{Z}^{2}\setminus \{0\}$ by
construction, and $f(\sigma )\neq 0$, we have $\delta _{h_{n}\ltimes
x,\sigma }=1$ for every $n\in \mathbb{N}$. Moreover,
\begin{equation*}
\theta _{h_{n}}(x)=h_{n}\ltimes ^{\theta }x=\sigma ^{\prime }(\delta
_{h_{n}\ltimes x,\sigma }\ltimes ^{\theta }s(x\sigma ))\sigma ^{-1}=\sigma
^{\prime }(1\ltimes ^{\theta }s(x\sigma ))\sigma ^{-1}
\end{equation*}%
for every $n\in \mathbb{N}$. In particular, $\theta _{h_{n}}(x)$ is
independent of $n$. This contradicts the assumption that $(h_{n})_{n\in
\mathbb{N}}$ is an infinite sequence of pairwise distinct elements of $%
\Delta $, and the claim is proved.

Apply Lemma \ref{Lemma:choose} to $K$ to obtain $t\in \lbrack K]$ such that
the set
\begin{equation*}
\{x\in X\colon \kappa _{0,xt}^{\widehat{\rho}|_{K}}(\xi _{i,s ( xt ) })
\perp \xi _{j,x} \mbox{ for every } i,j= 1,\ldots ,n \}
\end{equation*}
has measure at least $1-\varepsilon $. Then $\int_X \langle \kappa _{0,xt}^{%
\widehat{\rho}|_{K}}(\xi _{i,s ( xt ) }),\xi _{j,x} \rangle d\mu ( x ) \leq
\varepsilon $ for $i,j=1,\ldots,n$, as desired.
\end{proof}

The following is the main result of this paper, from which we will derive
Theorems~\ref{Theorem:nB-conj}, \ref{Theorem:nB-oe}, \ref{Theorem:nB-oeLCSCU}
and \ref{Theorem:nB-oe-relation} from the introduction. In the statement, we
say that a class-bijective pmp extension of $R$ is weak mixing if it is weak
mixing as a countable pmp groupoid.

\begin{theorem}
\label{Theorem:reduction-relations} Let $R$ be an ergodic nonamenable
countable pmp equivalence relation on the standard probability space $%
(X,\mu) $, and let $( Y,\nu ) $ be the standard atomless probability space.
Then there exists an assignment $A\mapsto R_{A}$ from countably infinite
discrete abelian groups to weak mixing class-bijective pmp extensions of $R$
on $( Y,\nu ) $ such that:

\begin{enumerate}
\item if $A$ and $A^{\prime }$ are isomorphic groups, then $R_{A}$ and $%
R_{A^{\prime }}$ are isomorphic relatively to $R$;

\item if $\mathcal{A}$ is a collection of pairwise nonisomorphic countably
infinite abelian groups such that $\{ R_{A}\colon A\in \mathcal{A} \} $ are
pairwise stably von Neumann equivalent, then $\mathcal{A}$ is countable.
\end{enumerate}
\end{theorem}

\begin{proof}
Let $\Lambda $ be the free subgroup of SL$_{2}(\mathbb{Z})$ provided by
Proposition~\ref{prop:FreeAction}, and let $\Delta \subseteq \Lambda $ be an
infinite index normal subgroup containing $\Xi $ such that the quotient $%
\Omega =\Lambda /\Delta $ is a property (T) group. Then $\Delta \leq \Lambda
\leq \Lambda $ has property (T) by Example~\ref{eg:QuotTtripleT}. Let $\nu $
be the Haar measure on the Pontryagin dual $\widehat{A}$, and set $%
M=L^{\infty }(\widehat{A},\nu )$, endowed with the trace-preserving action $%
\mathtt{Lt}$ of $\widehat{A}$ given by left translation. As in the proof of
Theorem \ref{Theorem:prescribed-cohomology}, consider the Bernoulli action $%
\beta \colon \Lambda \curvearrowright M^{\otimes \Omega }$ associated with $%
\Lambda \curvearrowright \Omega =\Lambda /\Delta $. Let $\theta $ be the
free action $\Lambda \curvearrowright ^{\theta }X$ provided by Proposition~%
\ref{prop:FreeAction}, let $G$ be the action groupoid associated to the
Bernoulli shift $\beta _{R\curvearrowright R}$ with base $(X,\mu )$, and let
$H$ be the action groupoid $\Lambda \ltimes ^{\theta }X$, which can be
identified with a subgroupoid of $G$. Let $\rho $ be the rigid action $%
\Lambda \curvearrowright \mathbb{T}^{2}$. Set $\widehat{\rho }=\mathrm{CInd}%
_{H}^{G}(\rho )$ and $\widehat{\beta }=\mathrm{CInd}_{H}^{G}(\beta )$. Thus $%
\widehat{\rho }$ is an action of $G$ on $\bigsqcup_{x\in X}(L^{\infty }(%
\mathbb{T}^{2})^{\otimes \Omega })^{\otimes \Sigma }$, and $\widehat{\beta }$
is an action of $G$ on $\bigsqcup_{x\in X}(M^{\otimes \Omega })^{\otimes
\Sigma }$. We will use the identifications of Proposition~\ref%
{proposition:identifications}.
%identify canonically $(M^{\otimes \Lambda /\Delta })^{\otimes \Sigma
%}$ with $M^{\otimes  ( \Omega \times \Sigma  ) }$. Hence $%
%\widehat{\beta}\otimes \widehat{\rho}$ can be seen as an action of $G$ on the bundle
%\begin{equation*}
%\mathcal{M}=\bigsqcup\nolimits_{x\in X}(M^{\otimes  (\Omega
%\times \Sigma  ) }\otimes L^{\infty }(\mathbb{T}^{2})^{\otimes \Sigma })%
%\text{.}
%\end{equation*}%
%Thus $[\widehat{\beta}\otimes \widehat{\rho}]$ can be seen as an action of $ [ G%
% ] $ on $L^{\infty } ( X ) \otimes M^{\otimes  ( (\Lambda
%/\Delta) \times \Sigma  ) }\otimes L^{\infty } ( \mathbb{T}%
%^{2} ) ^{\otimes \Sigma }$.

Let $M_A\subseteq M^{\otimes (\Omega \times \Sigma )}$ denote the fixed
point algebra of the action $\mathtt{Lt}^{\otimes ( \Omega \times \Sigma )}$%
, and let $\mathcal{M}_{A}$ be the $(\widehat{\beta}\otimes \widehat{\rho})$%
-invariant sub-bundle
\begin{equation*}
\mathcal{M}_A=\bigsqcup\nolimits_{x\in X}(M_A\otimes L^{\infty }(\mathbb{T}%
^{2})^{\otimes \Sigma })
\end{equation*}%
of $\mathcal{M}$. Then there exists a standard atomless probability space $(
X_{A},\mu _{A} ) $ with $M_A\cong L^{\infty } ( X_{A},\mu _{A} ) $. Define $%
\zeta _{A}$ to be the restriction of $\widehat{\beta}\otimes \widehat{\rho}$
to $\mathcal{M}_{A}$, and let $R_A$ be the orbit equivalence relation of the
action $\zeta_A$, which is a class-bijective pmp extension of $R$ on $%
(X_A,\mu_A)\cong (Y,\nu)$.

The action $[\widehat{\beta}]|_{\Lambda }$ is weak mixing and malleable by
Lemma \ref{Lemma:coinduce-bernoulli}, and $[\widehat{\rho}]|_{\Delta }$ is
weak mixing by Proposition~\ref{prop:restrCoindWkMix}. It follows that $%
R_{A} $ is weak mixing. It is clear that isomorphic groups yield pmp
class-bijective extensions of $R$ that are isomorphic relatively to $R$, so
that (1) holds. \newline

\textbf{Claim:} $R_A$ is an expansion of the orbit equivalence relation of $%
\Lambda\curvearrowright^{\rho} \mathbb{T}^{2}$.

Denote by $\widehat{\beta}_{A}$ the restriction of $\widehat{\beta}$ to $M_A$%
. By \cite[Proposition 7.2 (3)]{bowen_neumanns_2015}, $R_{A}$ is an
extension of the orbit equivalence relation of $[\widehat{\beta}%
_{A}]|_{\Lambda }\times \rho $ of $\Lambda $ on $X\times X_{A}\times \mathbb{%
T}$. In turn, this equivalence relation is an expansion of the orbit
equivalence relation of $\rho $, as witnessed by the second coordinate
projection $X_{A}\times \mathbb{T}\to \mathbb{T}$. This proves the claim.
\newline

Our next goal is to show that there is a group isomorphism $H_{:\Delta ,%
\mathrm{w}}^{1}([\zeta _{A}]|_{\Lambda })\cong A$. We will prove this in a
sequence of claims. Let $w\colon \Lambda \rightarrow L^{\infty }(X)\otimes
M_{A}\otimes L^{\infty }(\mathbb{T}^{2})^{\otimes \Sigma }$ be a $\Delta $%
-invariant cocycle for $[\zeta _{A}]|_{\Lambda }$. Then $w$ is also a $%
\Delta $-invariant cocycle for $[\widehat{\beta }\otimes \widehat{\rho }%
]|_{\Lambda }$. Since $[\widehat{\rho }]|_{\Delta }$ is weak mixing and $w$
is $\Delta $-invariant, it follows from Lemma \ref{Lemma:wm} that $w$ takes
values in the $\Delta $-fixed point subalgebra
\begin{equation*}
L^{\infty }(X)\otimes (M^{\otimes \Omega })^{\otimes \Sigma }\otimes \mathbb{%
C}\subseteq L^{\infty }(X)\otimes M^{\otimes (\Omega \times \Sigma )}\otimes
L^{\infty }(\mathbb{T}^{2})^{\otimes \Sigma }.
\end{equation*}%
Since $\Omega =\Lambda /\Delta $ has property (T), it follows that $\Lambda $
has the $\Delta $-invariant property (T). Using weak mixing and malleability
for $[\widehat{\beta }]|_{\Lambda }$, we apply Theorem \ref%
{Theorem:trivial-superrigidity} to find a unitary $v\in L^{\infty
}(X)\otimes M^{\otimes (\Omega \times \Sigma )}$ such that $v^{\ast }[%
\widehat{\beta }]_{\lambda }(v)=w_{\lambda }\ \mathrm{mod}\mathbb{C}$ for
every $\lambda \in \Lambda $. Fix $g\in \widehat{A}$ and $\lambda \in
\Lambda $. Then
\begin{equation*}
(\mathrm{id}_{L^{\infty }(X)}\otimes \mathtt{Lt}_{g}^{\otimes (\Omega \times
\Sigma )})(w_{\lambda })=w_{\lambda }
\end{equation*}%
and hence%
\begin{equation*}
v^{\ast }[\widehat{\beta }]_{\lambda }(v)=(\mathrm{id}_{L^{\infty
}(X)}\otimes \mathtt{Lt}_{g}^{\otimes (\Omega \times \Sigma )})(v)^{\ast }([%
\widehat{\beta }]_{\lambda }(\mathrm{id}_{L^{\infty }(X)}\otimes \mathtt{Lt}%
_{g}^{\otimes (\Omega \times \Sigma )})(v))\ \mathrm{mod}\mathbb{C}
\end{equation*}%
and%
\begin{equation*}
(\mathrm{id}_{L^{\infty }(X)}\otimes \mathtt{Lt}_{g}^{\otimes (\Omega \times
\Sigma )})(v)v^{\ast }=[\widehat{\beta }]_{\lambda }(\mathrm{id}_{L^{\infty
}(X)}\otimes \mathtt{Lt}_{g}^{\otimes (\Omega \times \Sigma )})(v)v^{\ast
})\ \mathrm{mod}\mathbb{C}.
\end{equation*}%
In other words, the left-hand side of the last equation generates a
one-dimensional subspace which is invariant under $[\widehat{\beta }%
]|_{\Lambda }$. Since $[\widehat{\beta }]|_{\Lambda }$ is weak mixing, this
subspace must consist of the scalar multiples of the unit, so there exists $%
\chi _{w}(g)\in \mathbb{C}$ such that
\begin{equation*}
(\mathrm{id}_{L^{\infty }(X)}\otimes \mathtt{Lt}_{g}^{\otimes (\Omega \times
\Sigma )})(v)=\chi _{w}(g)v\text{.}
\end{equation*}

\textbf{Claim:} The resulting map $\chi _{w}\colon \widehat{A}\rightarrow
\mathbb{C}$ is a character. Let $g,g^{\prime }\in \widehat{A}$. Then
\begin{align*}
\chi _{w}(gg^{\prime })v& =(\mathrm{id}_{L^{\infty }(X)}\otimes \mathtt{Lt}%
_{gg^{\prime }}^{\otimes (\Omega \times \Sigma )})(v) \\
& =(\mathrm{id}_{L^{\infty }(X)}\otimes \mathtt{Lt}_{g}^{\otimes (\Omega
\times \Sigma )})((\mathrm{id}_{L^{\infty }(X)}\otimes \mathtt{Lt}%
_{g^{\prime }}^{\otimes (\Omega \times \Sigma )})(v) \\
& =(\mathrm{id}_{L^{\infty }(X)}\otimes \mathtt{Lt}_{g}^{\otimes (\Omega
\times \Sigma )})(\chi _{w}(g^{\prime })v) \\
& =\chi _{w}(g^{\prime })\chi _{w}(g),
\end{align*}%
as desired. It follows that there is a well-defined map $\chi \colon
Z_{:\Delta ,\mathrm{w}}^{1}([\zeta _{A}]|_{\Lambda })\rightarrow A$. \newline

\textbf{Claim:} $\chi$ is a group homomorphism. Let $w,w^{\prime 1}_{:\Delta
,\mathrm{w}}([\zeta_{A}]|_{\Lambda})$. We want to show that $%
\chi_{ww^{\prime }}(g)=\chi_w(g)\chi_{w^{\prime }}(g)$ for all $g\in\widehat{%
A}$. As before, find unitaries $v,v^{\prime \infty}(X)\otimes
M^{\otimes(\Omega\times\Sigma)}$ satisfying
\begin{equation*}
v^{\ast }[\widehat{\beta}]_{\lambda} ( v ) =w_{\lambda}\ \mathrm{mod}\mathbb{%
C} \ \ \mbox{ and } \ \ (v^{\prime \ast }[\widehat{\beta}]_{\lambda} (
v^{\prime }) =w^{\prime }_{\lambda}\ \mathrm{mod}\mathbb{C}
\end{equation*}
for all $\lambda\in\Lambda$. Set $z=vv^{\prime }$. Then $w_{\lambda}w^{%
\prime }_{\lambda}=z^{\ast }[\widehat{\beta}]_{\lambda} (z )$ mod $\mathbb{C}
$ for all $\lambda\in\Lambda$. Fix $g\in\widehat{A}$. Then
\begin{align*}
\chi_{ww^{\prime }}(g)z&=(\mathrm{id}_{L^{\infty } ( X ) }\otimes \mathtt{Lt}
_{g}^{\otimes (\Omega\times \Sigma) }) ( z ) \\
&=(\mathrm{id}_{L^{\infty } ( X ) }\otimes \mathtt{Lt} _{g}^{\otimes
(\Omega\times \Sigma) }) ( v )(\mathrm{id}_{L^{\infty } ( X ) }\otimes
\mathtt{Lt} _{g}^{\otimes (\Omega\times \Sigma) }) ( v^{\prime }) \\
&=\chi_w(g)v\chi_{w^{\prime }}(g)v^{\prime }=\chi_w(g)\chi_{w^{\prime }}(g)z,
\end{align*}
as desired. \newline

\textbf{Claim:} The kernel of $\chi $ is the set of relative weak
coboundaries for $[\zeta _{A}]|_{\Lambda }$. Let $w\in Z_{:\Delta ,\mathrm{w}%
}^{1}([\zeta _{A}]|_{\Lambda })$ satisfy $\chi _{w}=1$. We want to show that
$w$ is weakly cohomologous to the trivial cocycle. Find $v\in L^{\infty
}(X)\otimes M^{\otimes (\Omega \times \Sigma )}$ such that
\begin{equation*}
v^{\ast }[\widehat{\beta }]_{\lambda }(v)=w_{\lambda }\ \mathrm{mod}\mathbb{C%
}\ \ \mbox{ and }\ \ (\mathrm{id}_{L^{\infty }(X)}\otimes \mathtt{Lt}%
_{g}^{\otimes (\Omega \times \Sigma )})(v)=v
\end{equation*}%
for every $\lambda \in \Lambda $ and every $g\in \widehat{A}$. In
particular, $v$ belongs to $L^{\infty }(X)\otimes M_{A}\otimes L^{\infty }(%
\mathbb{T}^{2})$ and $vw_{\lambda }[\widehat{\beta }]_{\lambda }(v^{\ast
})=1\ \mathrm{mod}\mathbb{C}$ for every $\lambda \in \Lambda $, as desired.
The converse is identical, so the claim is proved. \newline

\textbf{Claim:} $\chi$ is surjective. Fix $\omega \in A$, which we regard as
a unitary in $C(\widehat{A})\subseteq L^{\infty }(\widehat{A},\nu )=M$. This
readily gives a $\Delta $-invariant unitary element $v$ of
\begin{equation*}
L^{\infty } ( X ) \otimes M^{\otimes (\Omega\times \Sigma )}\otimes \mathbb{C%
}\subseteq L^{\infty } ( X ) \otimes M^{\otimes (\Omega\times \Sigma
)}\otimes L^{\infty } ( \mathbb{T}^{2} )
\end{equation*}%
such that $(\mathrm{id}_{L^{\infty } ( X ) }\otimes \mathrm{\mathtt{Lt}}
_{g}^{\otimes (\Omega\times \Sigma )} ( v ) =\omega ( g ) v$ for every $g\in
A$. Define a $\Delta $-invariant cocycle $z_{\omega}$ for $[ \zeta _{A} ]
|_{\Lambda }$ by $z_{\omega}(\lambda)= v^{\ast }[\beta ]_{\lambda} ( v ) $
for all $\lambda\in\Lambda$. It is then immediate to show that $%
\chi_{z_{\omega}}=\omega$.

It follows that $\chi$ induces a group isomorphism from $H^1_{:\Delta ,%
\mathrm{w}}([\zeta_{A}]|_{\Lambda})$ to $A$, as desired. \newline

We now prove (3). Let $\mathcal{A}$ be a collection of pairwise
nonisomorphic countably infinite discrete abelian groups such that the
relations $\{R_{A}\colon A\in \mathcal{A}\}$ are pairwise stably von Neumann
equivalent. Since $\zeta _{A}$ is free, the crossed product of $\zeta _{A}$
is isomorphic to $L(R_{A})$. Thus, the actions $\{\zeta _{A}\colon A\in
\mathcal{A}\}$ have stably isomorphic crossed products. Furthermore, for
every $A\in \mathcal{A}$ there is a group isomorphism $H_{:\Delta ,\mathrm{w}%
}([\zeta _{A}]|_{\Lambda })\cong A$. Therefore the actions $\{[\zeta
_{A}]|_{\Lambda }\colon A\in \mathcal{A}\}$ are pairwise not conjugate by
Theorem \ref{Theorem:separability}, and hence $\mathcal{A}$ is countable by
Theorem~\ref{Theorem:separability}. This concludes the proof.
\end{proof}

\begin{corollary}
\label{Corollary:reduction-groups} Let $\Gamma $ be a nonamenable countable
discrete group and let $(Y,\nu )$ be the standard atomless probability
space. Then there exists an assignment $A\mapsto \theta _{A}$ from countably
infinite discrete abelian groups to free weak mixing actions $\Gamma
\curvearrowright (Y,\nu )$ such that:

\begin{enumerate}
\item if $A$ and $A^{\prime }$ are isomorphic, then $\theta _{A}$ and $%
\theta _{A}^{\prime }$ are conjugate;

\item if $\mathcal{A}$ is a collection of pairwise nonisomorphic countably
infinite abelian groups such that $\{\theta _{A}\colon A\in \mathcal{A}\}$
are pairwise stably von Neumann equivalent, then $\mathcal{A}$ is countable.
\end{enumerate}
\end{corollary}

\begin{proof}
Let $\theta \colon\Gamma \curvearrowright ( X,\mu ) $ be the Bernoulli
action of $\Gamma $ with base $[ 0,1] $, and let $R$ be the
corresponding orbit equivalence relation. Then $R$ is a nonamenable ergodic
countable pmp equivalence relation. Observe that if $\hat{R}$ is a
class-bijective extension of $R$ on $( Y,\nu ) $, then $\hat{R}$
is the orbit equivalence relation of a free pmp action $\hat{\theta}$ of $%
\Gamma $ on $( Y,\nu ) $ with a distinguished factor map $\pi
\colon Y\rightarrow X$ onto $\theta $. Furthermore, $\hat{R}$ is ergodic
(respectively, weak mixing) if and only if $\hat{\theta}$ is ergodic
(respectively, weak mixing). If $\hat{R}^{\prime }$ is another
class-bijective extension of $R$ on $( Y,\nu ) $, with
corresponding $\Gamma $-action $\theta ^{\prime }$ and factor map $\pi
^{\prime }\colon Y\rightarrow X$, then $\hat{R}$ and $\hat{R}^{\prime }$ are isomorphic
relatively to $R$ if and only $\hat{\theta}$ and $\hat{\theta}^{\prime }$ are
\emph{conjugate relatively to} $\theta $, that is, there exists an automorphism
$\phi $ of $( Y,\nu ) $ such that, up to discarding a null set, $%
\pi ^{\prime }\circ \phi =\pi $ and $\phi \circ \hat{\theta}_{\gamma }=\hat{%
\theta}_{\gamma }^{\prime }\circ \phi $ for every $\gamma \in \Gamma $. The
conclusion thus follows from Theorem \ref{Theorem:reduction-relations}.
\end{proof}

\subsection{Borel complexity\label{Subsection:complexity}}

We recall here some fundamental notions from Borel complexity theory, which
can also be found in \cite{gao_invariant_2009,hjorth_classification_2000}.

\begin{definition}
Let $X$ and $Y$ be standard Borel spaces, and let $E$ and $F$ be equivalence
relations on $X$ and $Y$, respectively.

\begin{enumerate}
\item $E$ is said to be \emph{Borel} if it is a Borel subset of $X\times X$
endowed with the product topology.

\item A \emph{Borel reduction }from $E$ to $F$ is a Borel function $f\colon
X\to Y$ such that $xEx^{\prime }$ if and only if $f(x)Ff(x^{\prime })$, for
every $x,x^{\prime }\in X$.

\item A \emph{countable-to-one Borel homomorphism} from $E$ to $F$ is a
Borel function $f\colon X\to Y$ such that $xEx^{\prime }$ implies $%
f(x)Ff(x^{\prime })$ for every $x,x^{\prime }\in X$, and if $\mathcal{A}$ is
a set of pairwise not $E$-equivalent elements of $X$ such that $f(\mathcal{A}%
)$ is contained in a single $E$-class, then $\mathcal{A}$ is countable.
\end{enumerate}

When $X=Y$, if $F$ is contained in $E$ (as subsets of $X\times X$), then we
say that $E$ is \emph{coarser} than $F$ and $F$ is \emph{finer} than $E$.
\end{definition}

%Let $ ( Y,\nu  ) $ be a standard atomless probability space, and let $%
%\Gamma $ be a countable group. We consider the space $\mathrm{Aut} (
%Y,\nu  ) $ of measure-preserving Borel automorphisms of $X$ as a
%topological space, endowed with the weak topology. We also let $\mathrm{Aut} ( Y,\nu  )
%^{\Gamma }$ be the corresponding product space, endowed with the product
%topology. We denote by $\mathrm{\mathrm{FE}}_{\Gamma } ( Y,\nu  ) $
%the space of free ergodic pmp actions of $\Gamma $ on $ ( Y,\nu  ) $%
%, and by $\mathrm{FWM}_{\Gamma } ( Y,\nu  ) $ the space of free
%weak mixing pmp actions of $\Gamma $ on $ ( Y,\nu  ) $.
%It is shown in \cite[Section 10%
%]{kechris_global_2010} that both $\mathrm{\mathrm{FE}}_{\Gamma } ( Y,\nu
% ) $ and $\mathrm{\mathrm{FWM}}_{\Gamma } ( Y,\nu  ) $ are
%Borel subspaces of $\mathrm{Aut} ( Y,\nu  ) ^{\Gamma }$, and hence
%standard Borel spaces when endowed with the inherited Borel structure.

One can regard the relation $\cong _{\mathrm{AG}}$ of isomorphism of
countably infinite discrete abelian groups as an equivalence relation on a
standard Borel space in a canonical way. Indeed, assuming without loss of
generality that countably infinite discrete groups have $\mathbb{N}$ as
universe, a countably infinite discrete group is an element of $2^{\mathbb{N}%
^{3}}\times 2^{\mathbb{N}^{2}}\times \mathbb{N}$, coding the multiplication
operation, the inverse map, and the distinguished element representing the
identity. It is easy to see that the set \textrm{AG }of elements of $2^{%
\mathbb{N}^{3}}\times 2^{\mathbb{N}^{2}}\times \mathbb{N}$ that arise in
this fashion is a Borel subset, and hence a standard Borel space with the
induced Borel structure. The following result is proved in \cite[Theorem 5.1]%
{epstein_borel_2011}.

\begin{theorem}[Epstein--T\"{o}rnquist]
\label{Theorem:criterion-nB} Let $E$ be an equivalence relation on a
standard Borel space. If there exists a countable-to-one Borel homomorphism
from $\cong_{\mathrm{AG}}$ to $E$, then $E$ is not Borel.
\end{theorem}

We proceed to explain how to regard the spaces of actions of a discrete pmp
groupoid and of class-bijective extensions of a given countable Borel
equivalence relation as a standard Borel space. We need some preparation
first.

\begin{notation}
Let $(X,\mu)$ and $(Y,\nu)$ be standard probability spaces. We denote by $%
\mathcal{I}_{ ( X,\mu ) } ( Y,\nu ) $ the space of quadruples $( T,A,B,\pi )
$ such that $A,B\subseteq X$ are Borel sets with $\mu ( A ) =\mu ( B ) $, $%
\pi \colon Y\to X$ is a Borel map such that $\pi _{\ast }(\nu) =\mu $, and $%
T\colon \pi ^{-1} ( A ) \to \pi ^{-1} ( B ) $ is a measure-preserving Borel
isomorphism. Two such quadruples $( T,A,B,\pi ) $ and $( T^{\prime
},A^{\prime },B^{\prime },\pi ^{\prime } ) $ are identified whenever $\pi
=\pi ^{\prime }$, the symmetric differences $A\bigtriangleup A^{\prime }$
and $B\bigtriangleup B^{\prime }$ have zero measure, and $T|_{\pi^{-1}(A\cap
A^{\prime })}= T^{\prime }|_{\pi^{\prime -1}(A\cap A^{\prime })}$ almost
everywhere.
\end{notation}

\begin{remark}
There is a canonical Polish topology on $\mathcal{I}_{ ( X,\mu ) } ( Y,\nu )
$. Indeed, set $M=L^{\infty }(Y,\nu)$ and $N= L^{\infty } ( X,\mu ) $,
endowed with the canonical traces. Identify $\mathcal{I}_{ ( X,\mu ) } (
Y,\nu ) $ with the space of quadruples $( \theta ,p,q,\eta ) $ such that $%
\eta \colon N\to M$ is an injective trace-preserving normal *-homomorphism, $%
p,q\in N$ are projections of the same trace, and $\theta \colon \eta (p)M\to
\eta (q)M$ is a trace-preserving *-isomorphism. Let $( d_{n} )_{n\in\mathbb{N%
}} $ and $(e_{n})_{n\in\mathbb{N}}$ be countable $2$-dense subsets of the
unit balls of $M$ and $N$, respectively, and set
\begin{align*}
\rho ( (\theta ,p,q,\eta ),(\theta ^{\prime },p^{\prime },q^{\prime },\eta
^{\prime }) ) = & \Vert p-p \Vert _{2}+ \Vert q-q^{\prime } \Vert _{2}
+\sum_{n\in\mathbb{N}} 2^{-n} \Vert \theta ( \eta (p)d_n ) -\theta ^{\prime
} ( \eta ^{\prime }(p^{\prime })d_{n} ) \Vert _{2} \\
&+ \sum_{n\in\mathbb{N}}2^{-n}\Vert \theta ^{-1} ( \eta (q)d_{n} ) -\theta
^{\prime } ( \eta ^{\prime }(q^{\prime })d_{n} ) \Vert _{2}+ \sum_{n\in%
\mathbb{N}}2^{-n}\Vert \eta ( e_{n} ) -\eta ^{\prime } ( e_{n} ) \Vert .
\end{align*}%
Then $\rho$ is complete metric on $\mathcal{I}_{ ( X,\mu ) } ( Y,\nu ) $.
Since it is clear that the corresponding topology is separable, this shows
that $\mathcal{I}_{ ( X,\mu ) } ( Y,\nu ) $ is a Polish space.
\end{remark}

Observe that, for a fixed $\eta _{0}$, the space $\mathcal{I}_{(X,\mu ),\eta
_{0}}(Y,\nu )$ of quadruples $(\theta ,p,q,\eta )$ in $\mathcal{I}_{(X,\mu
)}(Y,\nu )$ with $\eta =\eta _{0}$ is an inverse subsemigroup, whose
idempotent semilattice can be identified with the space of projections of $%
L^{\infty }(X,\mu )$. If $G$ is a discrete pmp groupoid with $G^{0}=X$, then
the semilattice of projections of $L^{\infty }(X,\mu )$ can be identified
with the idempotent semilattice of $[[G]]$.

\begin{proposition}
Let $(Y,\nu )$ be the standard probability space, let $G$ be a discrete pmp
groupoid, and let $R$ be an ergodic countable Borel equivalence relation.

\begin{enumerate}
\item There is a canonical Polish topology on the space $\mathrm{\mathrm{Act}%
}_{G}(Y,\nu )$ of actions of $G$ on $(Y,\nu )$. With respect to this Polish
topology, the space $\mathrm{Erg}_{G}(Y,\nu )$ of ergodic actions of $G$ on $%
\left( Y,\nu \right) $.

\item There is a canonical Polish topology on the space $\mathrm{Ext}%
_{R}(Y,\nu )$ of class-bijective pmp extensions of $R$ on $(Y,\nu )$. With
respect to this Polish topology, the space $\mathrm{Erg}_{R}(Y,\nu )$ of
class-bijective pmp extensions of $R$ on $(Y,\nu )$ which are ergodic, and
the space $\mathrm{WM}_{R}(Y,\nu )$ of class-bijective pmp extensions of $R$
which are weak mixing, are Borel subsets of $\mathrm{Ext}_{R}(Y,\nu )$.
\end{enumerate}
\end{proposition}

\begin{proof}
(1). Set $(X,\mu)=(G^0,\mu_G)$. An action $\alpha$ of $G$ on $( Y,\nu )$ is
given by a trace-preserving normal unital *-homomorphism $\eta \colon
L^{\infty } ( X,\mu ) \to L^{\infty } ( Y,\nu ) $ together with a Borel
groupoid homomorphism $\alpha\colon G \to \mathrm{Aut} ( \bigsqcup_{x\in
X}L^{\infty } ( Y_{x} ) ) $ that is the identity on the unit space. In turn,
this induces an inverse semigroup homomorphism $t\mapsto [ [ \alpha ] ] _{t}$
from $[ [ G ] ] $ to $\mathcal{I}_{ ( X,\mu ) ,\eta } ( Y,\nu ) $ which is
the identity on the corresponding idempotent semilattices.

Fix now a countable dense inverse semigroup $\Sigma \subseteq [ [ G ] ] $
such that the set of elements of $\Sigma $ with support and range with full
measure is dense in $[ G ] $. Let $\eta \colon L^{\infty } ( X,\mu ) \to
L^{\infty } ( Y,\nu ) $ be a trace-preserving injective *-homomorphism, and
let $\theta\colon \Sigma \to \mathcal{I}_{ ( X,\mu ) ,\eta } ( Y,\nu ) $ be
a inverse semigroup homomorphism which is the identity on the corresponding
semilattices. Then the pair $(\eta,\theta)$ gives rise to a unique action $%
\alpha $ of $G$ on $( Y,\nu ) $ satisfying $[ [ \alpha ] ] =\theta $. Thus
one can identify the set of actions of $G$ on $( Y,\nu ) $ with the set of
such pairs $( \eta, \theta) $, which is a closed subset of the countably
infinite product $\mathcal{I}_{ ( X,\mu ) } ( Y,\nu ) ^{\Sigma }$. This
yields a canonical Polish topology on the space \textrm{Act}$_{G} ( Y,\nu ) $
of actions of $G$ on $( Y,\nu ) $.

Fix a $2$-norm dense countable subset $\mathcal{P}$ of the set of
projections of $L^{\infty }(Y,\nu )$, and a $2$-norm dense countable subset $%
\mathcal{F}$ of the unit ball of $L^{\infty }(Y,\nu )$. Denote by $\mathrm{E}%
_{\eta (L^{\infty }(X,\mu ))}\colon L^{\infty }(Y,\nu )\rightarrow \eta
(L^{\infty }(X,\mu ))$ the unique trace-preserving conditional expectation
from $L^{\infty }(Y,\nu )$ onto $\eta (L^{\infty }(X,\mu ))$. Ergodicity of
an action $\alpha $ of $G$ on $(Y,\nu )$ can be characterized in terms of
the associated maps $\eta \colon L^{\infty }(X,\mu )\rightarrow L^{\infty
}(Y,\nu )$ and $\theta \colon \Sigma \rightarrow \mathcal{I}_{(X,\mu ),\eta
}(Y,\nu )$ as follows: $\alpha $ is \emph{ergodic} if and only if for every $%
p,p^{\prime }\in \mathcal{P}$ with $\tau (p)=\tau (p^{\prime })$ and every $%
\varepsilon >0$ there exists $\sigma \in \Sigma \cap \lbrack G]$ such that $%
\Vert \theta _{\sigma }(p)-\theta _{\sigma }(p^{\prime })\Vert
_{2}<\varepsilon $. Hence \textrm{Erg}$_{G}(Y,\nu )$ is a Borel subset of
\textrm{Act}$_{G}(Y,\nu )$.

(2). As noted in Subsection \ref{Subsection:actions-spaces}, one can
identify class-bijective pmp extensions of $R$ on $(Y,\nu )$ with the
(necessarily principal) action groupoid associated with a pmp action of $R$.
Thus, the space $\mathrm{Ext}_{R}(Y,\nu )$ of class-bijective pmp extensions
of $R$ is endowed with a canonical Polish topology, obtained by identifying $%
\mathrm{Ext}_{R}(Y,\nu )$ with $\mathrm{Act}_{R}(Y,\nu )$ and using
part~(1). Since $R$ is assumed to be ergodic, a class-bijective extension of
$R$ is ergodic if and only if the associated action of $R$ is ergodic.
Therefore it follows from part~(1) that $\mathrm{Erg}_{R}(Y,\nu )$ is a
Borel subset $\mathrm{Ext}_{R}(Y,\nu )$. Similarly, one can see that $%
\mathrm{WM}_{R}( Y,\nu ) $ is a Borel subset of $\mathrm{Ext}%
_{R}( Y,\nu ) $ as follows. Fix a countable $2$-norm dense subset
$X$ of $L^{\infty }( Y,\nu ) $. Then an extension $%
\widehat{R}$ of $R$ corresponding to an action $\theta $ of $R$, regarded as
above an element of $\mathrm{Act}_{R}( Y,\nu ) $, is weak mixing
if and only if, for every $n\in \mathbb{N}$ and $f_{1},\ldots ,f_{n}\in X$
there exists $\sigma \in \Sigma \cap [ G] $ such that $\vert
\tau ( f_{i}\theta _{\sigma }( f_{j}) ) -\tau (
f_{i}) \tau ( f_{j}) \vert <2^{-n}$ for $1\leq
i,j\leq n$. Thus, $\mathrm{WM}_{R}( Y,\nu ) $ is a Borel subset
of $\mathrm{Ext}_{R}( Y,\nu ) $.
\end{proof}

\begin{remark}
As observed in \cite[Section 16]{kechris_spaces_2017}, the proof of
Rokhlin's Skew Product Theorem \cite[Theorem 3.18]{glasner_ergodic_2003}
shows that any ergodic class-bijective pmp extension of $R$ is isomorphic to
a \emph{skew product }of $R$ via a cocycle. This observation allows one to
give a different (but equivalent) parametrization of the space $\mathrm{Erg}%
_{R} ( Y,\nu ) $ of ergodic class-bijective pmp extensions of $R$.
\end{remark}

Observing that the construction in Theorem \ref{Theorem:reduction-relations}
is given by a Borel map with respect to the parametrization of
class-bijective extensions described above, we obtain:

\begin{theorem}
\label{Theorem:c-to-1-reduction-relation}Let $R$ be an ergodic nonamenable
countable pmp equivalence relation. There is a Borel function $A\mapsto
R_{A} $ from abelian countably infinite groups to weak mixing pmp extensions
of $R$ on the standard atomless probability space such that:

\begin{enumerate}
\item if $A$ and $A^{\prime }$ are isomorphic groups, then $R_{A}$ and $%
R_{A^{\prime }}$ are isomorphic relatively to $R$;

\item if $\mathcal{A}$ is a collection of pairwise nonisomorphic countably
infinite abelian groups such that $\{R_{A}\colon A\in \mathcal{A}\}$ are
pairwise stably von Neumann equivalent, then $\mathcal{A}$ is countable.
\end{enumerate}

In particular, if $E$ is any equivalence relation for weak mixing pmp
extensions of $R$ on the standard atomless probability space that is coarser
than isomorphism relative to $R$ and finer than stable von Neumann
equivalence, then $A\mapsto R_{A}$ is a countable-to-one Borel homomorphism
from $\cong _{\mathrm{AG}}$ to $E$.
\end{theorem}

The following corollary is an immediate consequence of Theorem \ref%
{Theorem:c-to-1-reduction-relation} and Theorem \ref{Theorem:criterion-nB}.

\begin{corollary}
\label{Corollary:countable-to-1-reduction-relation}Let $R$ be an ergodic
nonamenable countable pmp equivalence relation. Let $E$ be any equivalence
relation for weak mixing pmp extensions of $R$ on the standard atomless
probability space that is coarser than orbit equivalence and finer than
stable von Neumann equivalence. Then $E$ is not Borel.
\end{corollary}

Theorem \ref{Theorem:nB-oe-relation} is then a particular instance of
Corollary \ref{Corollary:countable-to-1-reduction-relation}.

\subsection{Actions of discrete groups\label{Subsection:discrete}}
In this subsection, we explain how to deduce Theorem~\ref{Theorem:nB-conj}
and Theorem~\ref{Theorem:nB-oe} from Theorem \ref{Theorem:nB-oe-relation}.

Fix a countable group $\Gamma$. We consider the space $\mathrm{Aut}%
(Y,\nu )$ of measure-preserving Borel automorphisms of $X$ as a topological
space, endowed with the weak topology. We also let $\mathrm{Aut}(Y,\nu
)^{\Gamma }$ be the corresponding product space, endowed with the product
topology. We denote by $\mathrm{\mathrm{FE}}_{\Gamma }(Y,\nu )$ the space of
free ergodic pmp actions of $\Gamma $ on $(Y,\nu )$, and by $\mathrm{FWM}%
_{\Gamma }(Y,\nu )$ the space of free weak mixing pmp actions of $\Gamma $
on $(Y,\nu )$. It is shown in \cite[Section 10]{kechris_global_2010} that
both $\mathrm{\mathrm{FE}}_{\Gamma }(Y,\nu )$ and $\mathrm{\mathrm{FWM}}%
_{\Gamma }(Y,\nu )$ are Borel subsets of $\mathrm{Aut}(Y,\nu )^{\Gamma }$,
and hence standard Borel spaces when endowed with the inherited Borel
structure. Suppose now that $R$ is the orbit equivalence relation of a free
ergodic action $\theta $ of $\Gamma $ on a standard probability space $%
\left( X,\mu \right) $. As observed in the proof of Corollary \ref%
{Corollary:reduction-groups}, every weak mixing class-bijective pmp
extension $\hat{R}$ of $R$ on $\left( Y,\nu \right) $ canonically gives rise
to a free weak mixing action $\hat{\theta}$ of $\Gamma $ on $\left( Y,\nu
\right) $ having $\hat{R}$ as its orbit equivalence relation. This yields a
Borel function $\hat{R}\mapsto \hat{\theta}$ from $\mathrm{WM}_{R}\left(
Y,\nu \right) $ to $\mathrm{FWM}_{\Gamma }\left( Y,\nu \right) $, which maps
class-bijective extensions of $R$ that are isomorphic relatively to $R$ to
conjugate actions of $\Gamma $. We therefore deduce the following
from Theorem \ref{Theorem:c-to-1-reduction-relation}.

\begin{theorem}
\label{Theorem:c-to-1-reduction-group}Let $\Gamma $ be a nonamenable
countable group. There is a Borel function $A\mapsto \theta _{A}$ from
abelian countably infinite groups to free weak mixing actions of $\Gamma $
on the standard atomless probability space such that:

\begin{enumerate}
\item if $A$ and $A^{\prime }$ are isomorphic groups, then $\theta _{A}$ and
$\theta _{A^{\prime }}$ are conjugate;

\item if $\mathcal{A}$ is a collection of pairwise nonisomorphic countably
infinite abelian groups such that $\{R_{A}\colon A\in \mathcal{A}\}$ are
pairwise stably von Neumann equivalent, then $\mathcal{A}$ is countable.
\end{enumerate}

In particular, if $E$ is any equivalence relation for free weak mixing
actions of $\Gamma $ on the standard atomless probability space that is
coarser than conjugacy and finer than stable von Neumann equivalence, then $%
A\mapsto \theta _{A}$ is a countable-to-one Borel homomorphism from $\cong _{%
\mathrm{AG}}$ to $E$.
\end{theorem}

Theorem~\ref{Theorem:nB-conj} and Theorem~\ref{Theorem:nB-oe} are immediate
consequences of Theorem~\ref{Theorem:c-to-1-reduction-group}, in view of
Theorem~\ref{Theorem:criterion-nB}.

\subsection{Actions of locally compact groups\label{Subsection:nonamenable}}
In this subsection, we explain how to deduce Theorem~\ref{Theorem:nB-oeLCSCU}
from Theorem~\ref{Theorem:nB-oe-relation}. For a locally compact, second
countable group $G$, the space $\mathrm{Act}_{G}(Y,\nu )$ of pmp actions of $%
G$ on a standard probability space $(Y,\nu )$ is endowed with a canonical
Polish topology. For example, it can be regarded as the space of continuous
group homomorphisms from $G$ to $\mathrm{Aut}( Y,\nu ) $, which
is a closed subspace of the Polish space $C(G,\mathrm{Aut}(Y,\nu ))$ of
continuous functions from $G$ to $\mathrm{Aut}(Y,\nu )$ endowed with the
compact-open topology. Indeed, it follows from Mackey's point realization
theorem \cite{mackey_point_1962} that any continuous group homomorphism from
$G$ to $\mathrm{Aut}( Y,\nu ) $ arises from a continuous pmp
action of $G$ on $( Y,\nu ) $, and vice versa.

One can consider an equivalent parametrization of pmp actions of $G$ on $%
( Y,\nu ) $ in the setting of von Neumann algebras, adopting the
perspective from the theory of locally compact quantum groups. Suppose that $%
M,N$ are von Neumann algebras with separable preduals $M_{\ast }$ and $%
N_{\ast }$, respectively. We can identify the space of $\sigma $-weakly
continuous linear isometries from $M$ to $N$ with the space of linear
contractive quotient mappings from $N_{\ast }$ to $M_{\ast }$, which is a
Polish space with respect to the topology induced by the operator norm. The
space $\mathrm{Hom}( M,N) $ of injective unital *-homomorphisms
from $M$ to $N$ is a closed subset of the space of $\sigma $-weakly
continuous linear isometries from $M$ to $N$ \cite[III.2.2.2]%
{blackadar_operator_2006}.

Let $C=L^{\infty }( G,\lambda)) $ be the von Neumann algebra
of essentially bounded functions on $G$ with respect to a left Haar measure $%
\lambda$ on $G$. The multiplication operation on $G$ gives rise to a
normal injective unital *-homomorphism $\Delta \colon C\rightarrow C\otimes C$ (%
\emph{comultiplication}) given by $\Delta ( f) ( s,t)
=f( st) $ for all $f\in C$ and all $s,t\in G$. Consider now a standard
probability space $( X,\mu ) $, and set $M=L^{\infty }(
X,\mu ) $. A continuous pmp action $G \curvearrowright ( X,\mu ) $
gives rise to a normal injective unital
*-homomorphism $\alpha \colon M\rightarrow C\otimes M$ given by $\alpha (
f) ( s,x) =f( g\cdot x) $ for all $f\in C$ and all $s\in G$ and all $x\in X$.
Such a map is a \emph{coaction }of $C$
on $M$, as it satisfies $( \Delta \otimes \mathrm{id}) \circ
\alpha =( \mathrm{id}\otimes \alpha ) \circ \alpha $. The $G$-action on $(X,\mu)$
being pmp is equivalent to the identity $%
( \mathrm{id}\otimes \mu ) \circ \alpha =\mathrm{id}$, where we
identify $\mu $ with a normal state on $M$.\
%Conversely, any coaction $%
%\alpha $ of $C$ on $M$ satisfying $( \mathrm{id}\otimes \mu )
%\circ \alpha =\mathrm{id}$ arises from a continuous pmp action of $G$ on $%
%( X,\mu ) $ in this fashion.

We can thus identify the space of
continuous actions of $G$ on $( X,\mu ) $ with the space of
coactions of $C$ on $M$ satisfying $( \mathrm{id}\otimes \mu )
\circ \alpha =\mathrm{id}$, which is a closed subspace of $\mathrm{Hom}%
( M,C\otimes M) $. The action $G\curvearrowright X$ is free if
and only if $( 1\otimes M) \alpha ( M) $ has dense
linear span inside $C\otimes M$; see \cite[Theorem 2.9]{ellwood_new_2000}.
Furthermore, the action $G\curvearrowright X$ is ergodic if and only if the
fixed point algebra $M^{\alpha }=\{ x\in M\colon \alpha (x)=x\otimes
1\} $ only contains the scalar multiples of the identity or,
equivalently, if $( ( \lambda _{G}\otimes \mathrm{id}) \circ
\alpha ) (x)=\lambda _{G}(x)1$ for every $x\in M$. This shows that the
space of free ergodic pmp actions of $G$ on $( X,\mu ) $ can be
seen as a closed subspace of $\mathrm{Hom}( M,C\otimes M) $.

\begin{definition}
Let $G$ be a locally compact, second countable group, let $(X,\mu )$ be a
standard probability space, and let $G\curvearrowright Y$ be a free action.
A \emph{cross section} for $G\curvearrowright X$ is a Borel subset $%
Y\subseteq X$ for which there exists an open neighborhood $U\subseteq G$
containing the identity of $G$, such that the restricted action $U\times
Y\rightarrow X$ is injective, and the orbit $G\cdot Y$ of $Y$ has full
measure in $x$. The cross section $Y$ is \emph{cocompact} if there is a
compact subset $K\subseteq G$ such that $K\cdot Y$ is $G$-invariant and has
full measure in $X$.
\end{definition}

We will use cross section equivalence relations, following \cite[Section~4]%
{kyed_l2-betti_2015} and \cite{bowen_neumanns_2015}.

\begin{definition}
Let $G$ be a locally compact, second countable group, let $(X,\mu )$ be a
standard probability space, let $G\curvearrowright X$ be a free action, and
let $Y\subseteq X$ be a cocompact cross section. The \emph{cross section
equivalence relation} associated to $Y$ is the orbit equivalence relation of
the action restricted to $Y$, that is,
\begin{equation*}
R=\{(y,y^{\prime })\in Y\times Y\colon y\in G\cdot y^{\prime }\}.
\end{equation*}%
This is a countable Borel equivalence relation on $Y$ \cite[Proposition 4.3]%
{kyed_l2-betti_2015}. If $\lambda $ is a Haar measure on $G$, then there
exist a unique $R$-invariant probability measure $\nu $ on $Y$ and a
constant $c\in ( 0,+\infty ) $ such that, with $\eta \colon U\times
Y\rightarrow X$ denoting the restriction of the $G$-action, one has
$\eta _{\ast }( \lambda |_{U}\times \nu ) =c\ \mu |_{U\cdot
Y}$. Such a measure $\nu $ is called the \emph{canonical $R$-invariant probability
measure} on $Y$. The countable pmp equivalence relation $R$ is ergodic if and
only if the action $G\curvearrowright X$ is ergodic by \cite[Proposition 4.3]%
{kyed_l2-betti_2015}.
\end{definition}

Recall that a locally compact, second countable group is said to be \emph{%
unimodular} if its left Haar measure is also right invariant.
%In the context of the above definition, suppose that $G$ is unimodular, and let $\mu$
%be the Haar measure on $G$. Then there exist a unique $R$-invariant probability
%measure $\lambda$ on $Z$, and a constant $c>0$ such that whenever $U\subseteq G$
%is an open set containing the identity of $G$ for which the canonical map
%$\alpha\colon U\times Z\to Y$ is injective, then the pull back measure
%$\alpha_\ast(\mu|_U\times \lambda)$ equals $c\nu|_{U\cdot Z}$;
%see~\cite[Proposition~4.3]{kyed_l2-betti_2015}.

\begin{proposition}
\label{prop:reductionUnimodular}Let $G$ be a locally compact, second
countable, unimodular group, let $(X,\mu )$ be an atomless standard
probability space, and let $G\curvearrowright X$ be a free ergodic action
with a cocompact cross section $Y$. Let $R$ denote the associated cross
section equivalence relation, and $\nu $ the canonical $R$-invariant
probability measure on $Y$. Then there is a Borel assignment $S\mapsto \zeta
_{S}$ from ergodic class-bijective pmp extensions of $R$ on a standard
atomless probability space $( \bar{Y},\bar{\nu}) $ to free
ergodic pmp actions of $G$ on the standard probability space $(\tilde{X},%
\tilde{\mu})$ such that, for ergodic class-bijective pmp extensions $%
S$ and $S^{\prime }$ of $R$ on $( \bar{Y},\bar{\nu}) $:

\begin{itemize}
\item if $S$ and $S^{\prime }$ are isomorphic relatively to $R$, then $\zeta
_{S}$ and $\zeta _{S^{\prime }}$ are conjugate;

\item if $\zeta _{S}$ and $\zeta _{S^{\prime }}$ are stably von Neumann
equivalent, then $S$ and $S^{\prime }$ are stably von Neumann equivalent.
\end{itemize}
\end{proposition}

\begin{proof}
Let $S$ be an ergodic class-bijective pmp extension of $R$ on $( \bar{Y}%
,\bar{\nu}) $. A free ergodic pmp action $\zeta _{S}$ of $G$ on $(%
\tilde{X},\tilde{\mu})$ is constructed in \cite[Proposition~8.3]%
{bowen_neumanns_2015} such that, as proved in \cite[Proof of Theorem B]%
{bowen_neumanns_2015} using \cite[Lemma~4.5]{kyed_l2-betti_2015}, the
group-measure space construction $G\ltimes ^{\zeta _{S}}L^{\infty }(\tilde{X}%
,\tilde{\mu})$ is isomorphic to an amplification of the II$_{1}$ factor $%
L(S) $. We review here the construction, since we need some specific details
about it.

Fix an open neighborhood $U$ of the identity in $G$, a compact subset $K$ of
$G$, $n\in \mathbb{N}$, and $g_{1},\ldots ,g_{n}\in G$ such that:

\begin{itemize}
\item the map $U\times Y\rightarrow X$ given by $( g,y) \mapsto g\cdot y$
is injective,

\item $K\cdot Y$ is a $G$-invariant conull subset of $X$,

\item $g_{1}=1$ and $g_{1}U\cup \cdots \cup g_{n}U$ contains $K.$
\end{itemize}

After discarding null sets, one can assume without loss of generality that
the stabilizer of every point of $X$ is trivial and that $K\cdot Y=X$. One can
then define a Borel map $\pi \colon X\rightarrow Y$ such that $\pi (x)$ belongs to
the $G$-orbit of $x$ for every $x\in X$, by setting $\pi ( g\cdot
y) =y$ for $g\in U$ and $y\in Y$, and $\pi ( g_{i}g\cdot y)
=y$ for $i=2,3,\ldots ,n$, $g\in U$, and $y\in Y$ such that $g_{i}g\cdot
y\notin (g_{1}U\cdot Y)\cup \cdots \cup (g_{i-1}U\cdot Y)$.

Fix an ergodic class-bijective pmp extension $S$ of $R$ on $( \bar{Y},%
\bar{\nu}) $, with corresponding factor map $p\colon \bar{Y} \rightarrow Y$.
After discarding a null set, one can assume that $p|_{ [ \bar{y} ] _{S}}$ maps $ [ \bar{y} ] _{S}$
bijectively onto $ [ p(\bar{y}) ] _{R}$ for every $\bar{y}\in \bar{%
Y}$. One then defines $\tilde{X}$
to be%
\begin{equation*}
X\times _{Y}\bar{Y}= \{ ( x,\bar{y}) \in X\times \bar{Y}\colon \pi (x)=p(\bar{y}%
) \} .
\end{equation*}%
The $G$-action $\zeta _{S}$ on $\tilde{X}$ is defined by setting $g\cdot
( x,\bar{y}) =( g\cdot x,\hat{y}) $ where $\hat{y}$ is
the unique element of $ [ \bar{y} ] _{S}$ such that $p(\hat{y})=\pi
(g\cdot x)$ (notice that $\pi (g\cdot x)$ belongs to $ [ \pi (x) ]
_{R}= [ p(y) ] _{R}$). The subsets $K$ and $U$ of $G$ as above
witness that
\begin{equation*}
\tilde{Y}= \{ ( p(\bar{y}),\bar{y}) \colon\bar{y}\in \bar{Y}%
 \}
\end{equation*}%
is a Borel cross-section for the $G$-action on $\tilde{X}$. As above, one
can define a Borel function $\tilde{\pi}\colon\tilde{X} \rightarrow \tilde{Y}$ by
setting $\pi (g\cdot \tilde{y})=\tilde{y}$ for $g\in U$ and $\tilde{y}\in
\tilde{Y}$, and $\pi (g_{i}g\cdot \tilde{y})=\tilde{y}$ for $i=2,\ldots ,n$,
$g\in U$, and $\tilde{y}\in \tilde{Y}$ such that $g_{i}g\cdot \tilde{y}%
\notin (g_{1}U\cdot \tilde{Y})\cup \cdots \cup (g_{i-1}U\cdot \tilde{Y})$.
The measure $\bar{\nu}$ on $\bar{Y}$ induces via the canonical Borel
isomorphism $\bar{Y} \rightarrow \tilde{Y}$, given by $\bar{y}\mapsto ( p(\bar{y}%
),\bar{y})$, a probability measure $\tilde{\nu}$ on $\tilde{Y}$. In
turn, $\tilde{\nu}$ induces a $G$-invariant probability measure $\tilde{\mu}$
on $X$ obtained by setting
\begin{equation*}
\tilde{\mu}( A) =( \lambda \times \tilde{\nu})
\{( g,\tilde{y}) \in G\times \widetilde{Y}\colon\tilde{\pi}(g\tilde{y}%
)=\tilde{y}\text{ and }g\tilde{y}\in A\}\text{,}
\end{equation*}%
for any measurable subset $A\subseteq X$.
This concludes the construction of the ergodic pmp action $\zeta _{S}$ of $G$
on $(\tilde{X},\tilde{\mu})$.

Suppose now that one starts with ergodic class-bijective pmp extensions $%
S$ and $S^{\prime }$ of $R$ on $( \bar{Y},\bar{\nu}) $. It is clear
from the construction above that, if $S$ and $S^{\prime }$ are isomorphic
relatively to $R$, then the corresponding $G$-actions $\zeta _{S}$ and $%
\zeta _{S^{\prime }}$ are conjugate. Furthermore, if $S$ and $S^{\prime }$ are
stably von\ Neumann equivalent, then $\zeta _{S}$ and $\zeta _{S^{\prime }}$
are also stably von Neumann equivalent, as $G\ltimes ^{\zeta _{S}}L^{\infty
}(\tilde{X},\tilde{\mu})$ is isomorphic to an amplification of $L(
S) $ and $G\ltimes ^{\zeta _{S^{\prime }}}L^{\infty }(\tilde{X},\tilde{%
\mu})$ is isomorphic to an amplification of $L( S^{\prime }) $.
Finally, one should notice that the assignment $S\mapsto \zeta _{S}$ is
given by a Borel map with respect to the parametrizations of class-bijective
pmp extensions of $R$ and free ergodic actions of $G$ as described above.
Indeed, it suffices to observe that, once one has fixed the cocompact cross
section $Y$ for $G\curvearrowright X$, the canonical $R$-invariant measure $%
\nu $ on $Y$, the compact subset $K$ of $G$ and the open neighborhood $U$ of
the identity of $G$ witnessing that $Y$ is a cocompact cross section, and
the elements $g_{1},\ldots ,g_{n}$ of $G$ such that $g_{1}=1$ and $%
g_{1}U\cup \cdots \cup g_{n}U$ contains $K$, the construction of the action $%
\zeta _{S}$ is canonical.
\end{proof}

Every locally compact, second countable, unimodular group admits a free
ergodic pmp action on the standard probability space, which can be obtained
as a Gaussian action; see \cite%
{schmidt_asymptotic_1984,bergelson_mixing_1988}. Combining
Proposition~\ref{prop:reductionUnimodular} with Theorem~\ref%
{Theorem:c-to-1-reduction-relation} we deduce the following.

\begin{theorem}
\label{Theorem:c-to-1-reduction-lc}Let $G$ be a locally compact, second
countable, unimodular group. There is a Borel function $A\mapsto \zeta _{A}$
from abelian countably infinite groups to free ergodic actions of $G$ on the
standard atomless probability space such that:

\begin{enumerate}
\item if $A$ and $A^{\prime }$ are isomorphic groups, then $\zeta _{A}$ and $%
\zeta _{A^{\prime }}$ are conjugate;

\item if $\mathcal{A}$ is a collection of pairwise nonisomorphic countably
infinite abelian groups such that $\{\zeta _{A}\colon A\in \mathcal{A}\}$
are pairwise stably von Neumann equivalent, then $\mathcal{A}$ is countable.
\end{enumerate}

In particular, if $E$ is any equivalence relation for free ergodic actions
of $G$ on the standard atomless probability space that is coarser than
conjugacy and finer than stable von Neumann equivalence, then $A\mapsto
\zeta _{A}$ is a countable-to-one Borel homomorphism from $\cong _{\mathrm{AG%
}}$ to $E$.
\end{theorem}

Theorem~\ref{Theorem:nB-oeLCSCU} is then an immediate consequence of Theorem~%
\ref{Theorem:c-to-1-reduction-lc}, by Theorem~\ref{Theorem:criterion-nB}.

%\subsection{Borel complexity\label{Subsection:complexity1}}

%\begin{theorem}
%\label{Theorem:c-to-1-reduction}Let $\Gamma $ be a nonamenable countable
%discrete group, and let $E$ be any equivalence relation for free weak mixing
%actions of $\Gamma $ on the standard atomless probability space that is
%coarser than conjugacy and finer than stable von Neumann equivalence. Then
%there is a countable-to-one Borel homomorphism of from $\cong _{\mathrm{AG}}$
%to $E$.
%\end{theorem}

%The following corollary is then an immediate consequence of Theorem \ref%
%{Theorem:c-to-1-reduction} and Theorem \ref{Theorem:criterion-nB}.

%\begin{corollary}
%\label{Corollary:c-to-1-reduction}Let $\Gamma $ be a nonamenable countable
%discrete group and let $E$ be any equivalence relation for free weak mixing
%actions of $\Gamma $ on the standard atomless probability space that is
%coarser than conjugacy and finer than stable von Neumann equivalence. Then $%
%E $ is not Borel.
%\end{corollary}

%Theorem \ref{Theorem:nB-conj} and Theorem \ref{Theorem:nB-oe} are particular
%instances of Corollary \ref{Corollary:c-to-1-reduction}.

%\bibliographystyle{amsplain}
%\bibliography{biblio-superr}

\providecommand{\MR}[1]{}
\providecommand{\bysame}{\leavevmode\hbox to3em{\hrulefill}\thinspace}
\providecommand{\MR}{\relax\ifhmode\unskip\space\fi MR }
\providecommand{\MRhref}[2]{%
  \href{http://www.ams.org/mathscinet-getitem?mr=#1}{#2}
}
\providecommand{\href}[2]{#2}

\end{document}